\documentclass[12pt,leqno]{amsart} 

\usepackage[utf8]{inputenc}
 \usepackage{amssymb,amsmath,amsthm,amscd,amsfonts}
\usepackage{color}
\usepackage{amscd,amssymb,verbatim,txfonts}
\usepackage{lmodern} 
\usepackage{graphicx}
\usepackage{tikz-cd}
\usepackage{hyperref,cleveref,graphics,mathrsfs}
\usepackage{commath}
\usepackage{mathabx}
\usepackage{enumerate}

\usepackage[all]{xy}
\usepackage{bm}

\newtheorem{thm}{Theorem}[section]
\newtheorem{cor}[thm]{Corollary} 
\newtheorem{lem}[thm]{Lemma}
\newtheorem{prop}[thm]{Proposition} 
\newtheorem{defn}[thm]{Definition}

\theoremstyle{definition}
\newtheorem{remark}[thm]{Remark}
\newtheorem{question}[thm]{Question}
\newtheorem{ex}[thm]{Example}

\numberwithin{equation}{section}

\newcommand{\ep}{\varepsilon}
\newcommand{\al}{\alpha}

\newcommand{\Om}{\Omega}

\newcommand{\vp}{\phi}

\newcommand{\la}{\langle}
\newcommand{\ra}{\rangle}

\newcommand{\ol}{\overline}
\newcommand{\ti}{\widetilde}
\newcommand{\wh}{\widehat}

\newcommand{\supp}{\operatorname{supp}}

\newcommand{\spn}{\operatorname{span}}

\newcommand{\fbl}{\mathrm{FBL}}
\newcommand{\fbp}{{\mathrm{FBL}}^{(p)}}

\newcommand{\N}{{\mathbb N}}
\newcommand{\R}{{\mathbb R}}

\newcommand{\cC}{{\mathcal C}}
\newcommand{\cD}{{\mathcal D}}
\newcommand{\cF}{{\mathcal F}}

\newcommand{\lL}{\mathcal{L}}

\newcommand{\cX}{{\mathcal X}}

\newcommand{\fU}{\mathfrak{U}}

\newcommand{\eps}{\varepsilon}

\DeclareMathOperator{\tr}{tr}

\newcommand{\pinf}{{p,\infty}}
\newcommand{\pp}{{p,p_2}}
\newcommand{\qq}{{q,q_2}}
\newcommand{\ppqq}{{(\pp;\qq)}}
\newcommand{\wlp}{L_\pinf}
\newcommand{\ocwlp}{L_\pinf^\circ}

\title[Banach lattices with upper $p$-estimates: Renorming and factorization]{Banach lattices with upper $p$-estimates: Renorming and factorization}

\author{E.~Garc\'ia-S\'anchez}
\address{Instituto de Ciencias Matem\'aticas (CSIC-UAM-UC3M-UCM)\\
Consejo Superior de Investigaciones Cient\'ificas\\
C/ Nicol\'as Cabrera, 13--15, Campus de Cantoblanco UAM\\
28049 Madrid, Spain.}
\email{enrique.garcia@icmat.es}

\author{D.~H.\ Leung}
\address{Department of Mathematics\\
National University of Singapore\\
Singapore 119076.
} \email{dennyhl@u.nus.edu}

\author{M.~A.\ Taylor}
\address{Department of Mathematics\\
ETH Z\"urich, Ramistrasse 101, 8092 Z\"urich, Switzerland.
} \email{mitchell.taylor@math.ethz.ch}

\author{P.~Tradacete}
\address{Instituto de Ciencias Matem\'aticas (CSIC-UAM-UC3M-UCM)\\
Consejo Superior de Investigaciones Cient\'ificas\\
C/ Nicol\'as Cabrera, 13--15, Campus de Cantoblanco UAM\\
28049 Madrid, Spain.}
\email{pedro.tradacete@icmat.es}

\date{\today}
\subjclass[2020]{46B42 (primary); 06B25, 46E30; 47B60; 47A68(secondary)} 

\keywords{Upper $p$-estimate; weak-$L_p$; factorization theory.}

\begin{document}
\begin{abstract}
The notions of $p$-convexity and concavity are fundamental tools for studying Banach lattices, as they partition the class of Banach lattices into a scale of spaces with $L_p$-like properties. Upper and lower $p$-estimates provide a refinement of this scale, modeled by the Lorentz spaces $L_{p,\infty}$ and $L_{p,1}$, respectively. In this article, we provide a comprehensive treatment of Banach lattices with upper $p$-estimates. In particular, we show that many well-known theorems about $p$-convex Banach lattices have analogues in the upper $p$-estimate setting, including the ability to represent all such spaces inside of infinity sums of model spaces, to canonically factor the convex operators and identify their associated operator ideals, as well as to give a precise description of the free objects and push-outs. Proving these results is far from straightforward and will require the development of a variety of new tools that avoid convexification and concavification procedures. In fact, we will identify many fundamental differences between the theories of $p$-convexity and upper $p$-estimates, particularly with regards to isometric problems and renormings.

\end{abstract}

\maketitle
\allowdisplaybreaks
\tableofcontents

\section{Introduction}
Convexity properties have long played a fundamental role in the theory of Banach lattices \cite{LT2}. Recall that a Banach lattice $X$ is \emph{$p$-convex} if there exists a constant $K\geq 1$ such that for all $n\in \mathbb{N}$ and $x_1,\dots, x_n\in X$ we have 

\begin{equation}\label{intro p-convex}
            \norm[3]{\intoo[3]{\sum_{i=1}^n |x_i|^{p}}^\frac{1}{p}}\leq K \intoo[3]{\sum_{i=1}^n \|x_i\|^{p}}^\frac{1}{p}.
        \end{equation}
Evidently, every Banach lattice is $1$-convex, and $L_p(\mu)$ is $p$-convex. By reversing the inequality in \eqref{intro p-convex} one obtains the notion of a \emph{$p$-concave} Banach lattice, and it is well-known that a Banach lattice $X$ is both $p$-convex and $p$-concave if and only if it is lattice isomorphic to $L_p(\mu)$ for some measure $\mu$.
\medskip

The above convexity and concavity properties partition the class of Banach lattices into a scale of spaces with properties similar to those of $L_p$. In fact, by Maurey's Factorization Theorem, it follows that a Banach lattice is $p$-convex if and only if it is lattice isomorphic to a closed sublattice of an $\ell_\infty$-sum of $L_p$-spaces. This theorem unlocks a host of concrete tools that can be used to study a wide class of function spaces and is in some sense analogous to the well-known representation of Banach spaces as closed subspaces of $\ell_\infty$-sums of copies of the real line.
\medskip

It is natural to replace one of the $p$-sums by a $q$-sum in \eqref{intro p-convex}. However, it is well-known that this only leads to one additional non-trivial property; namely, $(p,\infty)$-convexity.  Here, we recall that a Banach lattice $X$ is \emph{$(p,\infty)$-convex} (or has an \emph{upper $p$-estimate}) if there exists a constant $K\geq 1$ such that for all $n\in \mathbb{N}$ and $x_1,\dots, x_n\in X$ we have 

\begin{equation}\label{intro pinfty-convex}
            \norm[3]{\bigvee_{i=1}^n |x_i|}\leq K \intoo[3]{\sum_{i=1}^n \|x_i\|^{p}}^\frac{1}{p}.
        \end{equation}
It is well-known that $X$ is $(p,\infty)$-convex if and only if the inequality \eqref{intro p-convex} holds for all pairwise disjoint $x_1,\dots, x_n\in X$.
 Although $(p,\infty)$-convexity is also a classical notion \cite{LT2}, the corresponding theory is far more difficult and underdeveloped than that of $p$-convexity -- the primary objective of this article is to explain and rectify this issue.
 \medskip

 As we shall see, the class of $(p,\infty)$-convex Banach lattices is modeled by weak-$L_p.$ In fact, using Pisier's Factorization Theorem we will be able to represent every $(p,\infty)$-convex Banach lattice as a closed sublattice of an $\ell_\infty$-sum of weak-$L_p$-spaces. Then, by studying the properties of this representing space, we will give quick and transparent proofs of many of the fundamental structural properties of the full class of $(p,\infty)$-convex Banach lattices, including the fact that $(p,\infty)$-convexity implies $q$-convexity for $q<p$ with sharp constant growing like $(p-q)^{-\frac{1}{p}}$ as $q\uparrow  p$.  This is the main content of \Cref{sec: review}.
\medskip

In the remainder of the article, we will analyze the fine structure of $(p,\infty)$-convex and $(p,1)$-concave Banach lattices and their operators. This will present major challenges compared to the analogous study of $p$-convexity and concavity. This is perhaps most strikingly illustrated by contrasting Proposition~\ref{prop: representation convexity infty sums} and Example~\ref{ex: finite dim BL}: Every $p$-convex Banach lattice with constant one embeds lattice \emph{isometrically} into an $\ell_\infty$-sum of $L_p$-spaces and every $(p,\infty)$-convex Banach lattice with constant one embeds lattice \emph{isomorphically} into an $\ell_\infty$-sum of weak-$L_p$-spaces, but there are $(p,\infty)$-convex Banach lattices with constant one that do not embed lattice isometrically into any $\ell_\infty$-sum of weak-$L_p$-spaces. This fact will make the isometric theory of $(p,\infty)$-convex Banach lattices particularly subtle, and will have immediate implications for the study of free Banach lattices with upper $p$-estimates as well as all other universal constructions.
\medskip

Another difficulty in the study of upper $p$-estimates is that one no longer has access to convexification and concavification arguments. This is a standard and often indispensable tool in the theory of $p$-convex Banach lattices. For example, one may use it in combination with the fact that every separable Banach lattice embeds lattice isometrically into $C(\Delta,L_1)$ to immediately deduce that every separable $p$-convex Banach lattice embeds lattice isometrically into $C(\Delta,L_p)$, i.e., $C(\Delta,L_p)$ is injectively universal for separable $p$-convex Banach lattices. Identifying an analogous injectively universal object in the upper $p$-estimate case is much more difficult, not only because the above proof technique breaks down, but also because $L_{p,\infty}(\mu)$ is non-separable, depends in a more complicated way on the choice of measure $\mu$, and can be equipped with more than one norm. However, in \Cref{sec: universal} we will give strong evidence that $C(\Delta, \ocwlp[0,\infty))$ might be the correct universal space.
\medskip

With the above difficulties in mind, the primary objective of this article will be to develop a host of new techniques to understand $(p,\infty)$-convex spaces and operators. In particular, we will develop a series of factorization results in \Cref{sec: abstract factorizations} and a complete theory of the relevant operator ideals in \Cref{sec: Reisner}.  In many cases, the correct route to obtain a $(p,\infty)$-convex analogue of a result on $p$-convex Banach lattices is to replace $L_p$ with the far more difficult to work with $\wlp$ (contrast Pisier and Maurey's factorization theorems) and then overcome the additional difficulties. Interestingly, however, in other cases it will be more appropriate to replace certain $\ell_p^n$-norms with $\ell^n_\infty$-norms (compare \eqref{intro p-convex} with \eqref{intro pinfty-convex}, or the general strategy in \Cref{sec: Reisner}) or to instead replace multiple $\ell_p^n$ norms with $\ell_{p,\infty}^n$ norms (compare the norm of $\fbp$ with $\fbl^{(\pinf)}$). In all cases, the proof that the desired $(p,\infty)$-convex result holds will be significantly more complicated than the $p$-convex one.

\subsection{Content of the paper}

We begin the paper by reviewing in \Cref{sec: review} some basic facts about $p$-convexity and upper $p$-estimates. In particular, we recall some of the properties of $\wlp$, as well as its role as a model space in the factorization of $(p,\infty)$-convex operators (\Cref{thm: factorization Pisier operator}) and the representation of Banach lattices with upper $p$-estimates (\Cref{prop: representation convexity infty sums}). We also provide a simple proof of the well-known fact that every Banach lattice with upper $p$-estimates is $q$-convex for every $1\leq q<p$ with a constant that explodes at a precise rate as $q$ approaches $p$, and give an example that shows that the converse is not true.\medskip

Next, in \Cref{sec: isomorphic upe and weak Lp} we delve into the connection between the class of Banach lattices with upper $p$-estimates (with constant 1) and that of sublattices of $\ell_\infty$-sums of $\wlp$-spaces. \Cref{prop: representation convexity infty sums} implies that these classes coincide isomorphically, with a universal constant $\gamma_p$ depending only on $p$. However, in \Cref{sec: fd example} we establish that they cannot coincide isometrically, as we provide an example of a (three-dimensional) Banach lattice that satisfies an upper $p$-estimate with constant 1, but does not embed isometrically into any $\ell_\infty$-sum of $\wlp$-spaces. Moreover, we show in \Cref{sec: renormings} that the ambiguity in the choice of the renorming of $\wlp$ (see \eqref{eq: renorming weak Lp}) is not the cause of this discrepancy. Implications for free Banach lattices are also discussed.\medskip

\Cref{sec: abstract factorizations} is devoted to extending a series of results on abstract factorizations of convex and concave operators through Banach lattices with certain convexity and concavity properties due to Reisner \cite{Reisner}, Meyer-Nieberg \cite{MN}, and Raynaud and the last author \cite{RT}. We take the ideas from \cite{MN,RT} to develop a general theory of convex and concave Banach lattices and operators, which is developed through Sections \ref{sec: geometry factorization} to \ref{sec: duality factorizations}. This allows us to extend the results from \cite{RT} on minimal and maximal factorizations for $p$-convex and $q$-concave operators, respectively, to the setting of $(\pinf)$-convex and $(q,1)$-concave operators (see \Cref{thm: optimal factorization upe lpe case}), refining \cite[Theorem 2.8.7]{MN}. Finally, in \Cref{sec: factorization of simultaneous conditions} we apply this abstract theory to the study of factorizations of operators between Banach lattices which are simultaneously $(\pp)$-convex and $(\qq)$-concave, for $p_2\in \{p,\infty\}$ and $q_2\in \{1,q\}$. The corresponding result when $p_2=p$ and $q_2=q$ was established in \cite{RT}. Therefore, here we deal with the three remaining cases. \medskip

In \Cref{sec: PO}, we revisit the construction of push-outs in subcategories of Banach lattices established in \cite{GLTT} with the aim of improving the isomorphism constant of Theorem 4.7 from the aforementioned reference. We show using a $p$-convexification and $p$-concavification argument that in the $p$-convex case this constant can be reduced to 1, that is, \cite[Theorem 4.4]{AT} holds also for push-outs in the subcategory of $p$-convex Banach lattices. We also apply the abstract factorization techniques from the previous section to improve the proof for the upper $p$-estimates case. \medskip

Next, in \Cref{sec: universal}, we continue the study of universal constructions in subcategories of Banach lattices. More specifically, we show that $C(\Delta, L_p[0,1])$ is universal for the class of separable $p$-convex Banach lattices, and provide some partial results that point towards $C(\Delta, \ocwlp[0,\infty))$ as the natural candidate to satisfy an analogous property in the setting of upper $p$-estimates.\medskip

Finally, we devote \Cref{sec: Reisner} to extending some results from Reisner \cite{Reisner} concerning the ideal of operators that factor through a Banach lattice with $p$-convex and $q$-concave factors to the other three relevant cases: We show that the ideal $M_\ppqq$ of operators that factor through a $(\pp)$-convex and a $(\qq)$-concave operator, with $p_2\in \{p,\infty\}$ and $q_2\in \{1,q\}$, is perfect, and can be identified with the ideal of operators that factor through a multiplication operator from $L_\pp$ to $L_\qq$.

\section{An overview of \texorpdfstring{$p$}{}-convexity and upper \texorpdfstring{$p$}{}-estimates}\label{sec: review}

Throughout the paper, Banach lattices will be denoted by $X,Y,Z$...~and Banach spaces by $E,F,G$...~The closed unit ball of a Banach space $E$ is denoted by $B_E$. Given $1\leq p\leq \infty$, we will write $p^*=\frac{p}{p-1}$. Expressions of the form $\intoo{\sum_{i=1}^n |x_i|^{p}}^\frac{1}{p}$ are defined via functional calculus (cf.~\cite[Section 1.d]{LT2}) and must be understood as the finite supremum $\bigvee_{i=1}^n|x_i|$ for $p=\infty$. We refer the reader to \cite{LT2,MN} for standard notation and terminology concerning Banach lattices. \medskip

As explained in the introduction, this paper is devoted to the study of upper $p$-estimates, and more specifically its differences and similarities with $p$-convexity. We begin by recalling the classical notions of convexity and concavity in full generality (see \cite[Chapter 16]{DJT}):

\begin{defn}\label{def: convexity and concavity}
    Let $X$ be a Banach lattice, $E$ a Banach space and $1\leq p,q\leq \infty$.
    \begin{enumerate}
        \item An operator $T:E\rightarrow X$ is called \emph{$(p,q)$-convex} if there exists a constant $K\geq 1$ such that
        \begin{equation}\label{eq: (p,q)-convex operator}
            \norm[3]{\intoo[3]{\sum_{i=1}^n |Tx_i|^{q}}^\frac{1}{q}}\leq K \intoo[3]{\sum_{i=1}^n \|x_i\|^{p}}^\frac{1}{p}
        \end{equation}
        for every $x_1,\ldots,x_n\in E$ and $n\in \N$. The best constant $K$ satisfying \eqref{eq: (p,q)-convex operator} is called the \emph{$(p,q)$-convexity constant of $T$} and is denoted $K^{(p,q)}(T)$.
        \item An operator $T:X\rightarrow E$ is called \emph{$(p,q)$-concave} if there exists a constant $K\geq 1$ such that
        \begin{equation}\label{eq: (p,q)-concave operator}
            \intoo[3]{\sum_{i=1}^n \|Tx_i\|^p}^\frac{1}{p} \leq K \norm[3]{\intoo[3]{\sum_{i=1}^n |x_i|^q}^\frac{1}{q}}
        \end{equation}
        for every $x_1,\ldots,x_n\in X$ and $n\in \N$. The best constant $K$ satisfying \eqref{eq: (p,q)-concave operator} is called the \emph{$(p,q)$-concavity constant of $T$} and is denoted $K_{(p,q)}(T)$.
        \item A Banach lattice $X$ is said to be \emph{$(p,q)$-convex} (respectively, \emph{$(p,q)$-concave}) if the identity operator $id_X:X\rightarrow X$ is $(p,q)$-convex (respectively, $(p,q)$-concave). We write $K^{(p,q)}(X)=K^{(p,q)}(id_X)$ and $K_{(p,q)}(X)=K_{(p,q)}(id_X)$ for the \emph{$(p,q)$-convexity and $(p,q)$-concavity constants of $X$}, respectively. 
    \end{enumerate}
    If $p=q$, we will simply say that $T$ or $X$ is \emph{$p$-convex} (respectively, \emph{$p$-concave}) and write $K^{(p)}(T)$ or $K^{(p)}(X)$ (respectively, $K_{(p)}(T)$ or $K_{(p)}(X)$) instead.
\end{defn}

If we replace arbitrary collections of vectors by disjoint collections, we arrive at the notions of upper and lower estimates:

\begin{defn}\label{def: upe}
    Let $X$ be a Banach lattice and $1\leq p\leq\infty$.
    \begin{enumerate}
        \item $X$ is said to satisfy an \emph{upper $p$-estimate} if there exists a constant $K\geq 1$ such that
        \begin{equation}\label{eq: upe}
            \norm[3]{\sum_{i=1}^n x_i}\leq K \intoo[3]{\sum_{i=1}^n \|x_i\|^p}^\frac{1}{p}
        \end{equation}
        for every (positive) pairwise disjoint $x_1,\ldots,x_n\in X$ and $n\in \N$. The best constant $K$ satisfying \eqref{eq: upe} is called the \emph{upper $p$-estimates constant of $X$} and is denoted $K^{(\uparrow p)}(X)$.
        \item $X$ is said to satisfy a \emph{lower $p$-estimate} if there exists a constant $K\geq 1$ such that
        \begin{equation}\label{eq: lpe}
            \intoo[3]{\sum_{i=1}^n \|x_i\|^p}^\frac{1}{p} \leq K \norm[3]{\sum_{i=1}^n x_i}
        \end{equation}
        for every (positive) pairwise disjoint $x_1,\ldots,x_n\in X$ and $n\in \N$. The best constant $K$ satisfying \eqref{eq: lpe} is called the \emph{lower $p$-estimates constant of $X$} and is denoted $K_{(\downarrow p)}(X)$.
    \end{enumerate}
\end{defn}

The first thing that should be mentioned is that convexity and concavity are dual notions: An operator $T$ is $(p,q)$-convex (respectively, $(p,q)$-concave) if and only if $T^*$ is $(p^*,q^*)$-concave (respectively, $(p^*,q^*)$-convex), with the same constant, and a similar relation holds for upper and lower $p$-estimates (see, for instance, \cite[Sections 1.d and 1.f]{LT2} or \cite[Chapter 16]{DJT}). Next, note that when $q<p$, the only $(p,q)$-convex operator is the null operator, and the same happens for $(p,q)$-concavity when $q>p$. Additionally, an operator $T$ is $(p,q)$-convex for some $q>p$ (respectively, $(p,q)$-concave for some $q<p$) if and only if it is $(p,\infty)$-convex (respectively, $(p,1)$-concave), with equivalent constants; and a Banach lattice is $(p,\infty)$-convex (respectively, $(p,1)$-concave) if and only if it satisfies an upper $p$-estimate (respectively, lower $p$-estimate), with the same constant. Finally, it is straightforward to check that for every $1\leq q \leq \infty$, every operator $T:E\rightarrow X$ is $(1,q)$-convex and every operator $T:X\rightarrow E$ is $(\infty,q)$-concave, both with constant 1. Therefore, the only meaningful cases, up to constants, are $p$-convexity and $p$-concavity for $1\leq p\leq \infty$, and $(p,\infty)$-convexity/upper $p$-estimates and $(p,1)$-concavity/lower $p$-estimates for $1<p<\infty$. Throughout the paper, either of the synonymous terms $(p,\infty)$-convexity and upper $p$-estimates will be employed when deemed more appropriate. \medskip

As expected, the notions of $p$-convexity and $(p,\infty)$-convexity are strongly related, and therefore so are $p$-concavity and $(p,1)$-concavity. First, it is clear that every $p$-convex operator is $(p,\infty)$-convex with the same constant. Moreover, it can be shown that every $p$-convex operator is $q$-convex for every $q<p$, and the $q$-convexity constant is non-decreasing \cite[Proposition 1.d.5]{LT2}. On the other hand, every $(p,\infty)$-convex operator is $q$-convex for every $q<p$, with constant $K^{(q)}(T)\leq C_{p,q}K^{(p,\infty)}(T)$, where $C_{p,q}$ is a constant that explodes like $(p-q)^{-\frac{1}{p}}$ when $q\rightarrow p^-$. Below, we elaborate on the proof of this fact and provide a new argument (see \Cref{cor: upe implies q convex}). Moreover, we show that the converse is not true via \Cref{ex: q convex but not upe}.\medskip

The prototypical examples of $p$-convex and $p$-concave Banach lattices are $L_p$-spaces. Actually, these spaces play a crucial role in many representation and factorization results for $p$-convex and $p$-concave Banach lattices and operators, as we will recall later. Some of these results have an analogue in the setting of $(p,\infty)$-convexity and $(p,1)$-concavity, with the Lorentz spaces $L_\pinf$ and $L_{p,1}$ playing the role of model spaces. Let us recall some facts about these spaces. \medskip

Let $(\Omega,\Sigma,\mu)$ be a measure space and $f$ a measurable function on $\Omega$. We define its decreasing rearrangement $f^*:[0,\infty)\rightarrow [0,\infty)$ by
\[f^*(t):=\inf\{\lambda>0 : \mu\{|f|>\lambda\}\leq t\}.\]
Given $1<p< \infty$, the space $L_{p,1}(\mu)$ is the set of all measurable functions $f:\Omega\rightarrow \R$ such that the norm
\[\vvvert f\vvvert_{L_{p,1}}:=\int_0^\infty t^\frac{1}{p}f^*(t)\frac{dt}{t}\]
is finite. Similarly, the space $L_{p,\infty}(\mu)$ is the set of all measurable functions $f$ such that the quasinorm
\[ \vvvert f\vvvert_{L_{p,\infty}}: = \sup_{t>0} t^\frac{1}{p} f^*(t)=\sup_{\lambda>0} \lambda\mu(\{|f|>\lambda\})^{\frac{1}{p}} \]
is finite. Note, in particular, that every $f\in  L_{p,\infty}(\mu)$ has $\sigma$-finite support. It is easy to check that the quasinorm of $L_{p,\infty}(\mu)$ is equivalent to the dual norm of $L_{p^*,1}(\mu)$ with equivalence constant $p^*$, in other words, $L_{p,\infty}(\mu)$ is the dual of $L_{p^*,1}(\mu)$ (up to a lattice isomorphic renorming). Actually, $L_\pinf(\mu)$ admits the following family of lattice renormings: For each $1\leq r<p$, let
\[\|f\|_{L_{p,\infty}^{[r]}}:=\sup_{0<\mu(A)<\infty} \mu(A)^{\frac{1}{p}-\frac{1}{r}} \left(\int_A |f|^r\,d\mu\right)^{\frac{1}{r}}. \]
All of these norms are equivalent to the quasinorm \cite[Exercise 1.1.12]{GrafakosCFA}:
\begin{equation}\label{eq: renorming weak Lp}
    \vvvert f\vvvert_{L_{p,\infty}} \leq \|f\|_{L_{p,\infty}^{[r]}} \leq \left(\frac{p}{p-r}\right)^{\frac{1}{r}}  \vvvert f\vvvert_{L_{p,\infty}} \quad \text{for every}\quad f\in L_{p,\infty}(\mu).
\end{equation}
Moreover, given $f\in L_{p,\infty}(\mu)$, the norms $\|f\|_{L_{p,\infty}^{[r]}}$ are increasing in $r\in [1,p)$. Observe that for each $1\leq r<p$, $L^{[r]}_{p,\infty}(\mu)=(L_{p,\infty}(\mu),\|\cdot\|_{L_{p,\infty}^{[r]}})$ is a Banach lattice that satisfies an upper $p$-estimate with constant 1. Hence, $L_{p^*,1}(\mu)$ satisfies a lower $p^*$-estimate.\medskip

Unless stated otherwise, in this paper, we will assume that $L_{p,\infty}(\mu)$ is endowed with the norm $\|\cdot\|_{L_{p,\infty}^{[1]}}$, and we will simply denote this norm by $\|\cdot\|_{L_{p,\infty}}$ or $\|\cdot\|_{p,\infty}$. Similarly, when referring to the predual space $L_{p^*,1}(\mu)$, we will write $L_{p^*,1}^{[r]}(\mu)$ when it is endowed with the dual norm induced by $\|\cdot\|_{L_{p,\infty}^{[r]}}$, and we will simply write $L_{p^*,1}(\mu)$ in the default case $r=1$. As usual, when the measure space is the natural numbers with the counting measure, we will write $\ell_{p,\infty}$ and $\ell_{p^*,1}$.\medskip

As mentioned earlier, the spaces $L_p$, $L_\pinf$ and $L_{p,1}$ play a special role in the theory of concavity and convexity, as the following results due to Maurey \cite{MaureyFactorization} and Pisier \cite{Pisier86} illustrate. It should be noted that both theorems can be stated for $0<r<p$, but in this work, we will only consider the range $1\leq r<p$. 

\begin{thm}[Maurey's Factorization Theorem]\label{thm: factorization Maurey operator}
    Let $1\leq r<p<\infty$, $E$ a Banach space, $T:E\rightarrow L_r(\mu)$ a linear operator and $C>0$. The following are equivalent:
    \begin{enumerate}
        \item $T$ is $p$-convex with $K^{(p)}(T)\leq C$.
        \item There exists a normalized $g\in L_1(\mu)_+$ such that $\{g=0\}\subseteq \{Tx=0\}$ for every $x\in E$, and $T=M R$, where $R:E\rightarrow L_p(g\cdot \mu)$ given by $Rx=g^{-\frac{1}{r}}Tx$ is bounded with $\|R\|\leq C$ and $M:L_p(g\cdot \mu)\rightarrow L_r(\mu)$ is the multiplication operator by $g^\frac{1}{r}$.
        \[\xymatrix{
		E  \ar@{->}[rd]_{R} \ar@{->}[rr]^{T} & & L_r(\mu)  \\
		 & L_p(g\cdot \mu) \ar@{->}[ru]_{M} & 
        }\]
    \end{enumerate}
\end{thm}

\begin{thm}[Pisier's Factorization Theorem]\label{thm: factorization Pisier operator}
    Let $1\leq r<p<\infty$, $E$ a Banach space and $T:E\rightarrow L_r(\mu)$ a linear operator. The following are equivalent:
    \begin{enumerate}
        \item There exists $C>0$ such that $T$ is $(p,\infty)$-convex with $K^{(p,\infty)}(T)\leq C$.
        \item There exists $C'>0$ and a normalized $g\in L_1(\mu)_+$ such that for every $x\in E$ and measurable subset $B\subseteq \Omega$
        \[\|Tx\chi_B\|_{L_r(\mu)}\leq C'\|x\| \intoo[3]{\int_B g\, d\mu}^{\frac{1}{r}-\frac{1}{p}}.\]
        \item There exists $C''>0$ and a normalized $g\in L_1(\mu)_+$ such that $\{g=0\}\subseteq \{Tx=0\}$ for every $x\in E$, and $T=M R$, where $R:E\rightarrow L^{[r]}_{p,\infty}(g\cdot\mu)$ given by $Rx=g^{-\frac{1}{r}}Tx$ is bounded with $\|R\|\leq C''$ and $M:L^{[r]}_{p,\infty}(g\cdot\mu)\rightarrow L_r(\mu)$ is the multiplication operator by $g^\frac{1}{r}$.
        \[\xymatrix{
		E  \ar@{->}[rd]_{R} \ar@{->}[rr]^{T} & & L_r(\mu)  \\
		 & L^{[r]}_{p,\infty}(g\cdot\mu) \ar@{->}[ru]_{M} & 
        }\]
    \end{enumerate}
    Moreover, for the best constants, we have $C\leq C''\leq C'\leq (1-\frac{r}{p})^{\frac{1}{p}-\frac{1}{r}}C$.
\end{thm}

Note that in the $(\pinf)$-convex setting, we obtain a constant $(1-\frac{r}{p})^{\frac{1}{p}-\frac{1}{r}}$ that does not appear in the $p$-convex case. For $r=1$, this constant will play a significant role in subsequent sections, so we will denote it by 
\[
\gamma_p=\Big(1-\frac{1}{p}\Big)^{\frac{1}{p}-1}=(p^*)^{\frac{1}{p^*}}.
\]
We include for later use the corresponding dual statements for concave operators.

\begin{thm}\label{thm: dual factorization Maurey operator}
    Let $1< q<s< \infty$, $E$ a Banach space, $T:L_s(\mu)\rightarrow E$ a linear operator and $C>0$. The following are equivalent:
    \begin{enumerate}
        \item $T$ is $q$-concave with $K_{(q)}(T)\leq C$.
        \item There exists a normalized $g\in L_1(\mu)_+$ such that $T=S \widetilde{M}$, where $\widetilde{M}:L_s(\mu) \rightarrow L_q(g\cdot \mu)$ is the multiplication operator by $g^{-\frac{1}{s}}$ and $S:L_q(g\cdot \mu)\rightarrow E$ is bounded with $\|S\|\leq C$.
        \[\xymatrix{
		L_s(\mu)  \ar@{->}[rd]_{\widetilde{M}} \ar@{->}[rr]^{T} & & E \\
		 & L_q(g\cdot \mu) \ar@{->}[ru]_{S} & 
        }\]
    \end{enumerate}
\end{thm}

\begin{thm}\label{thm: dual factorization Pisier operator}
    Let $1< q<s< \infty$, $E$ a Banach space and $T:L_s(\mu)\rightarrow E$ a linear operator. The following are equivalent:
    \begin{enumerate}
        \item There exists $C>0$ such that $T$ is $(q,1)$-concave with $K_{(q,1)}(T)\leq C$.
        \item There exists $C'>0$ and a normalized $g\in L_1(\mu)_+$ such that for every $f\in L_s(\mu)$ and measurable subset $B\subseteq \Omega$
        \[\|T(f\chi_B)\|\leq C'\|f\|_{L_s(\mu)} \intoo[3]{\int_B g\, d\mu}^{\frac{1}{q}-\frac{1}{s}}.\]
        \item There exists $C''>0$ and a normalized $g\in L_1(\mu)_+$ such that $T=S \widetilde{M}$, where $\widetilde{M}:L_s(\mu)\rightarrow L^{[s^*]}_{q,1}(g\cdot\mu)$ is the multiplication operator by $g^{-\frac{1}{s}}$ and $S:L^{[s^*]}_{q,1}(g\cdot\mu)\rightarrow E$ is bounded with $\|S\|\leq C''$.
        \[\xymatrix{
		L_s(\mu)  \ar@{->}[rd]_{\widetilde{M}} \ar@{->}[rr]^{T} & & E  \\
		 & L^{[s^*]}_{q,1}(g\cdot\mu) \ar@{->}[ru]_{S} & 
        }\]
    \end{enumerate}
    Moreover, for the best constants, we have $C\leq C''\leq C'\leq (1-\frac{s^*}{q^*})^{\frac{1}{q^*}-\frac{1}{s^*}}C$.
\end{thm}

For concave operators defined on $L_\infty$, or more generally on $C(K)$-spaces, the analogues of the last two results are Pietsch's Factorization Theorem for $p$-summing operators (see, for instance, \cite[Theorem 9.2]{TJ}) and Pisier's generalization for $(p,1)$-summing operators \cite[Theorem 2.4]{Pisier86}, since every operator defined on a $C(K)$-space is $(p,q)$-concave if and only if it is $(p,q)$-summing. Maurey's and Pisier's Factorization Theorems can be used, for instance, to prove the following representation result \cite[Proposition 3.3]{GLTT}, which generalizes the well-known representation of any Banach lattice as a sublattice of an $\ell_\infty$-sum of $L_1$ spaces (see \cite[Lemma 3.4]{Lotz}). Throughout the paper, by the \textit{$\ell_\infty$-sum of spaces $(X_\gamma)_{\gamma\in \Gamma}$} we mean the space 
\[\intoo[3]{\bigoplus_{\gamma\in \Gamma}X_\gamma}_\infty:=\cbr[3]{x=(x_\gamma)_{\gamma\in \Gamma}: x_\gamma\in X_\gamma\, \forall \gamma\in \Gamma, \|x\|:=\sup_{\gamma\in \Gamma}\|x_\gamma\|<\infty}.\]

\begin{prop}\label{prop: representation convexity infty sums}
    Let $X$ be a Banach lattice and $1<p<\infty$:
    \begin{enumerate}
        \item If $X$ is $p$-convex, there are a family $\Gamma$ of probability measures and a lattice isomorphism
        \[ J : X\rightarrow \intoo[3]{\bigoplus_{\mu\in\Gamma}L_p(\mu)}_{\infty}\]
        such that $\|x\| \leq \|Jx\| \leq K^{(p)}(X)\|x\|$ for all $x\in X$.
        \item If $X$ is $(p,\infty)$-convex, there are a family $\Gamma$ of probability measures and a lattice isomorphism
        \[ J : X\rightarrow \intoo[3]{\bigoplus_{\mu\in \Gamma}L_{p,\infty}(\mu)}_{\infty}\]
        such that $\|x\| \leq \|Jx\| \leq \gamma_p K^{(\pinf)}(X)\|x\|$ for all $x\in X$, where $\gamma_p=(p^*)^\frac{1}{p^*}$ is the constant from \Cref{thm: factorization Pisier operator} associated to $r=1$.
    \end{enumerate}
\end{prop}

Clearly, every closed sublattice of a $p$-convex or $(\pinf)$-convex Banach lattice inherits the corresponding property. Therefore, the above conditions characterize $p$-convexity and $(\pinf)$-convexity. Similarly, one can characterize the $p$-concave and $(p,1)$-concave Banach lattices as quotients of $\ell_1$-sums of $L_p$ and $L_{p,1}$-spaces, respectively. As a corollary of \Cref{prop: representation convexity infty sums}, we obtain an alternative proof of the fact that every $(p,\infty)$-convex Banach lattice (and, by \cite[Theorem 2.8.7]{MN} every $(\pinf)$-convex operator) is $q$-convex for every $q<p$, with a constant that explodes like $(p-q)^{-\frac{1}{p}}$ when $q\rightarrow p^-$. To our knowledge, there existed two proofs of this fact: The proof presented in \cite[Corollary 16.7]{DJT} exploits a connection between $(r,s)$-concave operators and $(r,s)$-summing operators on $C(K)$-spaces (see \cite[Theorem 16.5]{DJT}) to apply Pisier's generalization of Pietsch's factorization theorem for $(r,1)$-summing operators (\cite[Theorem 2.4]{Pisier86}, see also \cite[Theorem 21.3]{TJ}). On the other hand, the proof presented in \cite[Theorem 1.f.7]{LT2} uses a probabilistic argument. The argument we provide here also relies on a result of Pisier (namely \Cref{thm: factorization Pisier operator}, which corresponds to \cite[Theorem 1.1]{Pisier86} instead of \cite[Theorem 2.4]{Pisier86}), but has some advantages over the other proofs. First of all, it is more direct than the one presented in \cite{DJT}, as it does not require one to go through $(p,q)$-summing operators. On the other hand, this proof is conceptually simpler than that of \cite{LT2}, and does not require probabilistic techniques.

\begin{cor}\label{cor: upe implies q convex}
    Let $1<p<\infty$ and assume that $X$ is a Banach lattice satisfying an upper $p$-estimate. Then it is $q$-convex for every $1\leq q<p$ with constant $K^{(q)}(X)\leq \left(\frac{p}{p-q}\right)^{\frac{1}{q}} \gamma_p K^{(\uparrow p)}(X)$.
\end{cor}

\begin{proof}
    First of all, by \Cref{prop: representation convexity infty sums} we know that $X$ lattice embeds with distortion $\gamma_p K^{(\uparrow p)}(X)$ into a Banach lattice of the form $\intoo[1]{\bigoplus_{\mu\in\Gamma}\wlp(\mu)}_{\infty}$, with $\mu$ a probability measure for every $\mu \in \Gamma$. Now, let us fix $1\leq q<p$, and observe that by \eqref{eq: renorming weak Lp} each of the factors $\wlp(\mu)$ is lattice isomorphic to $\wlp^{[q]}(\mu)$, with distortion $\left(\frac{p}{p-q}\right)^{\frac{1}{q}}$. Now, it suffices to show that $\wlp^{[q]}(\mu)$ is $q$-convex with constant~1: Given $f_1,\ldots,f_n \in \wlp^{[q]}(\mu)$ and a measurable set $A$ with $\mu(A)>0$, it follows that
    \[\mu(A)^{\frac{1}{p}-\frac{1}{q}} \intoo[3]{\int_A \sum_{i=1}^n|f_i|^q\, d\mu}^\frac{1}{q}=  \intoo[3]{ \sum_{i=1}^n \intoo[3]{\mu(A)^{\frac{1}{p}-\frac{1}{q}}\intoo[3]{\int_A|f_i|^q\, d\mu}^\frac{1}{q}}^q}^\frac{1}{q} \leq  \intoo[3]{ \sum_{i=1}^n \|f\|_{\wlp^{[q]}(\mu)}^q}^\frac{1}{q}.\]
    Taking the supremum over all the sets $A$, we obtain
    \[ \norm[3]{\intoo[3]{ \sum_{i=1}^n |f_i|^q}^\frac{1}{q}}_{\wlp^{[q]}(\mu)} \leq  \intoo[3]{ \sum_{i=1}^n \|f\|_{\wlp^{[q]}(\mu)}^q}^\frac{1}{q}.\]
    Since $q$-convexity is preserved under taking $\ell_\infty$-sums, sublattices and isomorphisms, we conclude that $X$ is $q$-convex with constant $K^{(q)}(X)\leq \left(\frac{p}{p-q}\right)^{\frac{1}{q}} \gamma_p K^{(\uparrow p)}(X)$.
\end{proof}

\begin{remark}\label{rem: interpretation renorming}
    Observe that the previous argument shows that for every $1\leq q<p$, $\wlp$ endowed with the norm $\|\cdot\|_{\wlp^{[q]}}$ is $q$-convex with constant 1, while it simultaneously satisfies an upper $p$-estimate with constant 1. In general, by the above results, obtaining information on the structure of $\left(\bigoplus_{\mu\in \Gamma}L_{p,\infty}(\mu)\right)_{\infty}$ will have implications for arbitrary Banach lattices with upper $p$-estimates. In particular, a direct proof of the fact that $L_\pinf$ is $(p,q)$-convex for every $q>p$ suffices to show that every $(p,\infty)$-convex operator is $(p,q)$-convex following the same scheme of proof as in \Cref{cor: upe implies q convex}. We note that this strategy is part of a more general line of reasoning; for example, understanding the ideals of universal spaces was used in \cite{AT} to deduce information about separably injective Banach lattices.
\end{remark}

The following example illustrates that the reverse implication in \Cref{cor: upe implies q convex} is not true:

\begin{ex}\label{ex: q convex but not upe}
    Let $1<p<\infty$. The space $X=\ell_\pinf(\ell_p)$, formed by the sequences of vectors $\overline{x}=(x_k)_{k\in\N}$ such that $x_k\in \ell_p$ for every ${k\in\N}$ and $\|\overline{x}\|_X:=\|(\|x_k\|_{\ell_p})_k\|_{\ell_\pinf}$ is finite, does not satisfy an upper $p$-estimate, but it is $q$-convex with constant $\left(\frac{p}{p-q}\right)^{\frac{1}{q}}$ for every $1\leq q<p$.
\end{ex}

\begin{proof}
    We start by showing that $X$ is $q$-convex for any $1\leq q<p$. Recall that both $\ell_p$ and $\ell_\pinf$ are $q$-convex with constants 1 and $\left(\frac{p}{p-q}\right)^{\frac{1}{q}}$, respectively. Given $\overline{x}^{(i)}=(x_k^{(i)})_{k\in\N}\in X$, $i=1,\ldots,n$, we have
    \begin{align*}
        \norm[3]{\intoo[3]{\sum_{i=1}^n |\overline{x}^{(i)}|^q}^\frac{1}{q}}_X & = \norm[4]{\intoo[4]{\intoo[3]{\sum_{i=1}^n |x_k^{(i)}|^q}^\frac{1}{q}}_k}_X = \norm[4]{\intoo[4]{\norm[3]{\intoo[3]{\sum_{i=1}^n |x_k^{(i)}|^q}^\frac{1}{q}}_{\ell_p}}_k}_{\ell_\pinf}\\
        & \leq \norm[4]{\intoo[4]{\intoo[3]{\sum_{i=1}^n \|x_k^{(i)}\|_{\ell_p}^q}^\frac{1}{q}}_k}_{\ell_\pinf} =  \norm[4]{\intoo[3]{\sum_{i=1}^n \abs[2]{\intoo[2]{\|x_k^{(i)}\|_{\ell_p}}_k}^q}^\frac{1}{q}}_{\ell_\pinf} \\
        & \leq \left(\frac{p}{p-q}\right)^{\frac{1}{q}}  \intoo[3]{\sum_{i=1}^n \norm[2]{\intoo[2]{\|x_k^{(i)}\|_{\ell_p}}_k}_{\ell_\pinf} ^q}^\frac{1}{q} = \left(\frac{p}{p-q}\right)^{\frac{1}{q}}  \intoo[3]{\sum_{i=1}^n \|\overline{x}^{(i)}\|_X^q}^\frac{1}{q},
    \end{align*}
    where the first and second inequalities follow from the $q$-convexity of $\ell_p$ and $\ell_\pinf$, respectively.\medskip

    To show that $X$ does not satisfy an upper $p$-estimate (i.e., it is not $(\pinf)$-convex), let us assume that the coordinates in $\ell_p$ and $\ell_\pinf$ are indexed starting at $0$. Let us fix $n\geq 1$ and denote by $(e_j)_{j=0}^\infty$ the canonical basis of $\ell_p$. For $k=0,\ldots,n-1$, we write $a^{(k)}=\sum_{j=0}^{n-1} \alpha_{(k+j)_n} e_j$ for the cyclic permutations of the coefficients $\alpha_k=(k+1)^\frac{1}{p^*}-k^\frac{1}{p^*}$ over the set $\{0,\ldots,n-1\}$. Here, by $i_n$ we mean $i \mod n$. We consider for $i=0,\ldots,n-1$ the (finite) sequences $\overline{x}^{(i)}=(x_k^{(i)})_{k=0}^\infty\in X$ given by $x_k^{(i)}=\alpha_{(k+i)_n} e_i$ for $k=0,\ldots,n-1$ and $x_k^{(i)}=0$ for $k\geq n$. It is clear that for every $i$, 
    \[\|\overline{x}^{(i)}\|_X=\|(\alpha_{(k+i)_n})_{k=0}^{n-1}\|_{\ell_\pinf}=\|(\alpha_{k})_{k=0}^{n-1}\|_{\ell_\pinf}=1,\quad \text{so  }\intoo[3]{\sum_{i=0}^{n-1} \|\overline{x}^{(i)}\|_X^p}^\frac{1}{p}=n^\frac{1}{p}.\]
    On the other hand, for every $k=0,\ldots, n-1$,
    \[\bigvee_{i=0}^{n-1}|x_k^{(i)}|=\sum_{i=0}^{n-1} \alpha_{(k+i)_n} e_i=a^{(k)},\]
    so
    \[\norm[3]{\bigvee_{i=0}^{n-1} |\overline{x}^{(i)}|}_X = \norm[4]{\intoo[4]{\norm[3]{\bigvee_{i=0}^{n-1} |x_k^{(i)}|}_{\ell_p}}_{k=0}^{n-1}}_{\ell_\pinf}= \norm[2]{\intoo[2]{\|a^{(k)}\|_{\ell_p}}_{k=0}^{n-1}}_{\ell_\pinf}= A_n n^\frac{1}{p},\]
    where $A_n=\intoo{\sum_{j=0}^{n-1} \alpha_j^p}^\frac{1}{p}$ is the norm of $a^{(k)}$ in ${\ell_p}$ for every $k$, and $n^\frac{1}{p}$ is the norm in $\ell_\pinf$ of the constant 1 vector on the set $\{0,\ldots,n-1\}$. It is not difficult to check that $A_n$ diverges as $n \rightarrow \infty$. More specifically, $\alpha_k \approx (k+1)^{-\frac{1}{p}}$ uniformly in $k$, so $A_n\approx \intoo{\sum_{k=1}^n k^{-1}}^\frac{1}{p}$, which diverges. To conclude the argument, observe that the quotients $\norm{\bigvee_{i=0}^{n-1} |\overline{x}^{(i)}|}_X /\intoo{\sum_{i=0}^{n-1} \|\overline{x}^{(i)}\|_X^p}^\frac{1}{p}=A_n$ are not uniformly bounded, so $X$ cannot be $(\pinf)$-convex.
\end{proof}

\begin{remark}\label{rem: counterexample with finite factors}
    Observe that for the previous argument to hold, it suffices to consider $\ell_\pinf(\ell_p^n)_{n\in \N}$, an $\ell_\pinf$-sum of $\ell_p^n$ with $n\in \N$. However, it is essential that the dimensions of the inner terms in the sum increase.
\end{remark}

\section{Banach lattices with upper \texorpdfstring{$p$}{}-estimates vs.~sublattices of \texorpdfstring{$\ell_\infty$-sums of $L_\pinf$}{}}\label{sec: isomorphic upe and weak Lp}

\Cref{prop: representation convexity infty sums} points out a fundamental difference between $p$-convexity and upper $p$-estimates: While the class of $p$-convex Banach lattices (with constant 1) and the class of closed sublattices of $\ell_\infty$-sums of $L_p$-spaces can be identified isometrically, the class of Banach lattices with upper $p$-estimates (with constant 1) and the class of closed sublattices of $\ell_\infty$-sums of $\wlp$-spaces do not coincide isometrically, but only isomorphically. Indeed, in this section, we provide an example of a Banach lattice that satisfies an upper $p$-estimate with constant 1 but does not embed isometrically into any $\ell_\infty$-sum of $\wlp$-spaces. Moreover, we show that $\wlp$ endowed with the different renormings from \eqref{eq: renorming weak Lp} can always be isometrically embedded into an $\ell_\infty$-sum of $L_\pinf^{[1]}$, concluding that the choice of norm for $L_\pinf$ is not the cause of the obstruction.

\subsection{A finite dimensional counterexample}\label{sec: fd example}

A natural question that arises from \Cref{thm: factorization Pisier operator} and \Cref{prop: representation convexity infty sums} is whether the constant $\gamma_p$ can be improved to $1$. In this section, we will show that the answer is generally negative.
\medskip

Let $X, Y$ be Banach lattices and let $1\leq C<\infty$.  We say that $X$ $C$-lattice embeds into $Y$ if there exists a lattice isomorphic injection $T:X\to Y$ such that $C^{-1}\|x\| \leq \|Tx\| \leq \|x\|$ for all $x\in X$.
In this case, we call $T$ a $C$-lattice embedding.  If $C =1$, then $T$ is a lattice isometric embedding.
A Banach space with a normalized $1$-unconditional basis can be regarded as a Banach lattice with respect to the coordinate-wise order.

\begin{thm}\label{t4.1}
Let $X$ be a Banach lattice with the order given by a normalized $1$-unconditional basis $(e_i)$.  
Denote the biorthogonal functionals by $(e_i^*)$.  For any $C\geq 1$, the following statements are equivalent:
\begin{enumerate}
\item For each normalized finitely supported $a = \sum^n_{i=1}a_ie_i\in X_+$ and every $\ep >0$, there exist $b = \sum^n_{i=1}b_ie^*_i\in X^*_+$ and $d = (d_i)^n_{i=1}\in \R^n_+$ so that $\sum^n_{i=1}a_ib_i > 1-\ep$,
$\sum^n_{i=1}d_i=1$  
and  $\|\sum_{i\in I}b_ie^*_i\|^{p^*} \leq C^{p^*}\sum_{i\in I}d_i$ for any $I \subseteq \{1,\dots,n\}$.
\item There is a set $\Gamma$ of (probability) measures such that $X$ $C$-lattice  embeds into the space $\intoo[1]{\bigoplus_{\mu\in \Gamma}L_{p,\infty}(\mu)}_\infty$.
\end{enumerate}
\end{thm}

\begin{proof}
(1)$\implies$(2) 
Let $a\in X_+$ be normalized and finitely supported, and let $\ep >0$ be given.  Obtain  $b$ and $d$ from (1). Let $(U_i)^n_{i=1}$ be a measurable partition of $(0,1)$ such that $\lambda(U_i) = d_i$, $1\leq i\leq n$, where $\lambda$ denotes the Lebesgue measure.
Define 
\[ S_{a,\ep}: X\to L_{p,\infty}(0,1) \text{  by } S_{a,\ep}\intoo[3]{\sum^\infty_{i=1}c_ie_i} = \sum^n_{i=1}\frac{c_ib_i\chi_{U_i}}{d_i}.\]
We have
\[ \|S_{a,\ep} a\| \geq \lambda(0,1)^{-\frac{1}{p^*}}\int^1_0|S_{a,\ep} a|\,d\lambda = \sum^n_{i=1}a_ib_i > 1-\ep.\]
On the other hand, if $\|\sum^\infty_{i=1}c_ie_i\|\leq1$ then there exists $I \subseteq \{1,\dots, n\}$ so that 
\begin{align*}
 \|S_{a,\ep} c\| &= \lambda \intoo[3]{\bigcup_{i\in I}U_i}^{-\frac{1}{p^*}} \int_{\bigcup_{i\in I}U_i}|S_{a,\ep}c|\,d\lambda=\intoo[3]{\sum_{i\in I}d_i}^{-\frac{1}{p^*}}\sum_{i\in I}|c_i|b_i\\
 & = \intoo[3]{\sum_{i\in I}d_i}^{-\frac{1}{p^*}}\left \la \sum^\infty_{i=1}|c_i|e_i, \sum_{i\in I}b_ie^*_i\right \ra
\leq \intoo[3]{\sum_{i\in I}d_i}^{-\frac{1}{p^*}}\norm[3]{\sum_{i\in I}b_ie^*_i} \leq C.
\end{align*}
Let $\Gamma$ be the Cartesian product of the set of normalized finitely supported elements in $X_+$ and the interval $(0,\frac{1}{2})$.
It is now clear that the map $T:X\to \intoo[1]{\bigoplus_{(a,\ep)\in \Gamma}L_{p,\infty}(0,1)}_\infty$ 
given by 
\[ Tc = (C^{-1}S_{a,\ep}c)_{(a,\ep)\in \Gamma}\]
is a $C$-lattice embedding.

\bigskip

\noindent (2)$\implies$(1)  Let $T:X\to  \intoo[1]{\bigoplus_{\mu\in \Gamma}L_{p,\infty}(\mu)}_\infty$ be a $C$-lattice  embedding.  
Suppose that $a = \sum^n_{i=1}a_ie_i$ is normalized  and $\ep > 0$.
For each $\mu_0\in \Gamma$, let $P_{\mu_0}$ be the projection of  $\intoo[1]{\bigoplus_{\mu\in \Gamma}L_{p,\infty}(\mu)}_\infty$ onto the $\mu_0$-th component.
There exists $\mu_0$ such that $C\|P_{\mu_0}Ta\| > 1-\ep$.
Let  $f_i = P_{\mu_0}Te_i.$ 
Choose a $\mu_0$-measurable set $U$ such that $0 < \mu_0(U) < \infty$ and 
\[ C\int_U P_{\mu_0}Ta\,d\mu_0 > (1-\ep)\mu_0(U)^{\frac{1}{p^*}}.\]
Set $U_i = U \cap \supp f_i$ and $m = (\sum^n_{i=1}\mu_0(U_i))^{\frac{1}{p^*}}$. Define
\[ b_i = \frac{C\int_{U_i}f_i\,d\mu_0}{m}, \quad d_i = \frac{\mu_0(U_i)}{m^{p^*}}, \quad 1\leq i\leq n.\]
We have $d_i \geq 0$, $\sum^n_{i=1}d_i =1$ and 
\[
 \sum^n_{i=1}a_ib_i  = \frac{C}{m}\int_{\bigcup^n_{i=1}U_i}\sum^n_{i=1}a_if_i\,d\mu_0 = 
 \frac{C}{m}\int_{U}P_{\mu_0}Ta\, d\mu_0 > 1-\ep.
\]
On the other hand, suppose that $I\subseteq \{1,\dots, n\}$. If $\|\sum^\infty_{i=1}c_ie_i\| \leq 1$, then
\begin{align*}
\left \la \sum^\infty_{i=1}c_ie_i, \sum_{i\in I}b_ie^*_i\right \ra & = \sum_{i\in I}c_ib_i = \frac{C}{m}\int_{\bigcup_{i\in I}U_i}\sum_{i\in I}c_if_i\,d\mu_0\\
& \leq  \frac{C}{m}\norm[3]{P_{\mu_0}T\intoo[3]{\sum^\infty_{i=1}c_ie_i}}\,\mu_0\intoo[3]{\bigcup_{i\in I}U_i}^{\frac{1}{p^*}}
\leq C\intoo[3]{\sum_{i\in I}d_i}^{\frac{1}{p^*}}.
\end{align*}
This proves that $\|\sum_{i\in I}b_ie_i^*\|^{p^*} \leq C^{p^*}\sum_{i\in I}d_i$.
\end{proof}

\begin{remark}
If $X$ is a two-dimensional Banach lattice that satisfies an upper $p$-estimate with constant $1$, then both conditions (1)  and (2) of \Cref{t4.1} hold with $C=1$.  
Indeed, let $(e_1,e_2)$ be a normalized $1$-unconditional basis for $X$.  Suppose that $a= a_1e_1 + a_2e_2$ is normalized and non-negative.   Choose $b = b_1e_1^* + b_2e^*_2\in X^*_+$ such that $\la a,b\ra =1 = \|b\|$. Define $d_i = b_i^{p^*}/(b_1^{p^*} + b_2^{p^*})$, $i=1,2$.
Note that $X^*$ satisfies a lower $p^*$-estimate with constant $1$ and $e_1^*$ and $e_2^*$ are disjoint, so
\[ b_1^{p^*} + b_2^{p^*} = \|b_1e^*_1\|^{p^*} + \|b_2e^*_2\|^{p^*} \leq \|b\|^{p^*} =1.\]
It is easy to check that $\|\sum_{i\in I}b_ie^*_i\|^{p^*} \leq \sum_{i\in I}d_i$ for any $I \subseteq \{1,2\}$ and $d_1 + d_2 = 1$.
\end{remark}

\begin{cor}\label{c4.2}
Suppose that $X$ is a  Banach lattice with the order given by a normalized  $1$-unconditional basis $(e_i)$ that $C$-lattice embeds into $\intoo[1]{\bigoplus_{\mu\in \Gamma}L_{p,\infty}(\mu)}_\infty$.
For any normalized $b = \sum^n_{i=1}b_ie^*_i \in X^*_+$ and any  $I_j\subseteq \{1,\dots, n\}$, $1\leq j\leq m$, such  that $\sum^m_{j=1}\chi_{I_j} = l\chi_{\{1,\dots,n\}}$ for some $l\in \N$, we have
\[ \sum_{j=1}^m\norm[3]{\sum_{i\in I_j}b_ie^*_i}^{p^*} \leq C^{p^*}l.\]
\end{cor}

\begin{proof}
Choose a normalized $a = \sum^n_{i=1}a_ie_i\in X_+$ so that $\sum^n_{i=1}a_ib_i =1$.  First, assume that $[(e_i)^n_{i=1}]$ is uniformly convex.
Let $\ep > 0$ be given. There exists $\delta >0$ so that if $b' = \sum^n_{i=1}b'_ie^*_i\in X^*_+$, $\|b'\|=1$
and $\la a, b'\ra > 1-\delta$, then $\|b-b'\| < \ep$.
By \Cref{t4.1}, there exists $b'\in X^*_+$ and $d = (d_i)^n_{i=1}\in \R^n_+$ so that $\sum_{i=1}^nd_i=1$, $\la a,b'\ra > 1-\delta$ and  $\|\sum_{i\in I}b_i'e_i^*\|^{p^*} \leq C^{p^*}\sum_{i\in I}d_i$ for any $I\subseteq \{1,\dots,n\}$.
It follows that for any $I\subseteq \{1,\dots, n\}$, 
\[ \norm[3]{\sum_{i\in I}b_ie_i^*}  \leq  \norm[3]{\sum_{i\in I}|b_i-b_i'|e_i^*}+ \norm[3]{\sum_{i\in I}b'_ie_i^*} \leq \ep |I| + C\intoo[3]{\sum_{i\in I}d_i}^{\frac{1}{p^*}}.
\]
Now assume that $I_j\subseteq \{1,\dots, n\}$, $1\leq j\leq m$, are such that $\sum^m_{j=1}\chi_{I_j} = l\chi_{\{1,\dots,n\}}$ for some $l\in \N$.
Set $A_j = \ep|I_j|$ and $B_j = C(\sum_{i\in I_j}d_i)^{\frac{1}{p^*}}$.  Then
\begin{align*}
\intoo[3]{\sum_{j=1}^m\norm[3]{\sum_{i\in I_j}b_ie^*_i}^{p^*}}^{\frac{1}{p^*}}& \leq \intoo[3]{\sum^m_{j=1}(A_j+ B_j)^{p^*}}^{\frac{1}{p^*}}
\leq \intoo[3]{\sum^m_{j=1}A_j^{p^*}}^{\frac{1}{p^*}} + \intoo[3]{\sum^m_{j=1}B_j^{p^*}}^{\frac{1}{p^*}} \\
& \leq \sum^m_{j=1}A_j + \intoo[3]{\sum^m_{j=1}B_j^{p^*}}^{\frac{1}{p^*}}
= \ep\sum^m_{j=1}|I_j|  + C\intoo[3]{\sum^m_{j=1}\sum_{i\in I_j}d_i}^{\frac{1}{p^*}}
\leq \ep nl + Cl^{\frac{1}{p^*}}.
\end{align*}
Taking $\ep \downarrow 0$ gives the desired result.
\medskip

We now return to a general $X$ with a normalized $1$-unconditional basis $(e_i)$. Given a normalized $b = \sum^n_{i=1}b_ie^*_i \in X^*_+$ and  $\ep > 0$, there is a uniformly convex lattice norm $\vvvert\cdot\vvvert$ on $[(e_i)^n_{i=1}]$ such that 
$\|x\| \leq \vvvert x\vvvert  \leq (1+\ep)\|x\|$ for any $x\in [(e_i)^n_{i=1}]$ (take, for instance, $\vvvert\cdot \vvvert =\|\cdot\|+\delta \|\cdot\|_{\ell_2^n}$ for $\delta>0$ small enough, which is strictly convex and hence uniformly convex, since $[(e_i)^n_{i=1}]$ is finite dimensional).
Set $u_i = e_i/\vvvert e_i\vvvert $, $u_i^* = e_i^*/\vvvert e_i^*\vvvert $ and $b' = \sum^n_{i=1}b_i'u^*_i = b/\vvvert b\vvvert $.
Note that $(u_i)^n_{i=1}$ is a $1$-unconditional basis for $[(e_i)^n_{i=1}] =   [(u_i)^n_{i=1}]$ and $b'$ is $\vvvert \cdot\vvvert $-normalized.   Moreover, if $T: [(e_i)^n_{i=1}]\to \intoo[1]{\bigoplus_{\mu\in \Gamma}L_{p,\infty}(\mu)}_\infty$ is a $C$-lattice embedding, then 
\[ \ti{T}x = Tx: [(u_i)^n_{i=1}]\to \intoo[3]{\bigoplus_{\mu\in \Gamma}L_{p,\infty}(\mu)}_\infty\]
is a $(1+\ep)C$-lattice embedding.
Assume that $I_j\subseteq \{1,\dots, n\}$, $1\leq j\leq m$, are such  that $\sum^m_{j=1}\chi_{I_j} = l\chi_{\{1,\dots,n\}}$ for some $l\in \N$.  
By the first part of the proof,
\[  \sum_{j=1}^m\Bigg\vvvert \sum_{i\in I_j}b_i'u^*_i\Bigg\vvvert ^{p^*} \leq (1+\ep)^{p^*}C^{p^*}l.
\]
Since $\vvvert b\vvvert \leq 1+\ep$, for any $I\subseteq \{1,\dots,n\}$ we have
\[ \Bigg\vvvert\sum_{i\in I}b'_iu^*_i\Bigg\vvvert =\frac{1}{\vvvert b\vvvert } \Bigg\vvvert \sum_{i\in I}b_ie^*_i\Bigg\vvvert  \geq \frac{1}{1+\ep}\norm[3]{\sum_{i\in I}b_ie^*_i}.
\]
Thus,
\[ \sum^m_{j=1}\norm[3]{\sum_{i\in I}b_ie_i^*}^{p^*} \leq (1+\ep)^{2p^*}C^{p^*}l.\]
Taking $\ep\downarrow 0$ gives the desired result.
\end{proof}

The next example shows that $\gamma_p$ cannot be made to be $1$ in Proposition \ref{prop: representation convexity infty sums} or Theorem \ref{thm: factorization Pisier operator}.

\begin{ex}\label{ex: finite dim BL} Suppose that $1< p < \infty$. Denote by $C_p$ the infimum of the set of constants $C$ so that every Banach lattice satisfying an upper $p$-estimate with constant $1$ $C$-lattice embeds into $\intoo[1]{\bigoplus_{\gamma\in \Gamma}L_{p,\infty}(\mu)}_\infty$ for some set of measures $\Gamma$.
Then $C_p^{p^*} \geq  \frac{3 \cdot 2^{p^*}}{2(1+2^{p^*})}> 1$.
\end{ex}

\begin{proof}
Let $X = \R^3$ so that the norm on $X^*$  is 
\[ \|(b_1,b_2,b_3)\|_{X^*} = \max\bigl\{\bigl(|b_i|^{p^*} + (|b_j|+|b_k|)^{p^*}\bigr)^{\frac{1}{p^*}}: \{i,j,k\} = \{1,2,3\}\bigr\}.\]
It is clear that the coordinate unit vectors form a normalized $1$-unconditional basis for $X$.
Suppose that $b=(b_1,b_2,b_3)\in X^*$ and $I,J$ are disjoint non-empty  subsets of $\{1,2,3\}$. Without loss of generality, we may assume that $I= \{i\}$ and $J = \{j,k\}$, where $i,j,k$ are distinct. Then
\[ \|b\chi_I\|^{p^*}_{X^*} + \|b\chi_J\|^{p^*}_{X^*} =  |b_i|^{p^*} + (|b_j| + |b_k|)^{p^*} \leq \|b\|_{X^*}^{p^*}.\]
Thus, $X^*$ satisfies a lower $p^*$-estimate with constant $1$.  Hence, $X$ satisfies an upper $p$-estimate with constant $1$.
Suppose that $X$ $C$-lattice embeds into some $\intoo[1]{\bigoplus_{\gamma\in \Gamma}L_{p,\infty}(\mu)}_\infty$.
The vector $(1+2^{p^*})^{-\frac{1}{p^*}}(1,1,1)$ is normalized in $X^*_+$.
By \Cref{c4.2}, 
\[ \frac{1}{1+2^{p^*} }\,(\|(1,1,0)\|^{p^*} + \|(1,0,1)\|^{p^*} + \|(0,1,1)\|^{p^*}) \leq 2C^{p^*}.\]
Thus,
\[ \frac{3 \cdot 2^{p^*}}{2(1+2^{p^*})} \leq C^{p^*}.
\]
Taking infimum over the  eligible $C$'s gives the desired result. Direct verification shows that  $\frac{3 \cdot 2^{p^*}}{2(1+2^{p^*})}> 1$ if $1< p < \infty$.
\end{proof}

\begin{question}
    Is $\gamma_p=(p^*)^\frac{1}{p^*}$ the optimal value of the constant $C_p$ defined in \Cref{ex: finite dim BL}?
\end{question}

As a consequence of the previous discussion, we find that the free Banach lattices associated to the classes $\cC_\pinf$ of Banach lattices satisfying an upper $p$-estimate with constant 1, and $\cX_\pinf$ of $L_\pinf(\mu)$ spaces, are not lattice isometric in general, even though they are lattice isomorphic (see \cite{GLTT}). Recall that, given a class $\cC$ of Banach lattices and a Banach space $E$, there exists an (essentially unique) Banach lattice $\fbl^\cC[E]$ together with a linear isometric embedding $\phi:E\rightarrow \fbl^\cC[E]$ with the property that for every Banach lattice $X$ in the class $\cC$ and every bounded operator $T:E\rightarrow X$, there exists a unique extension $\widehat{T}:\fbl^\cC[E]\rightarrow X$ as a lattice homomorphism such that $\widehat{T}\phi=T$ and $\|\widehat{T}\|=\|T\|$. This space is called the \emph{free Banach lattice generated by $E$ associated to the class $\cC$}. In \cite[Theorem 3.6]{GLTT} it was established that for every Banach space $E$, the norm of $\fbl^{\cC_\pinf}[E]$ is $\gamma_p$-equivalent to the norm of $\fbl^{\cX_\pinf}[E]$. However, we will see that this isomorphism constant cannot be made 1.\medskip

In order to show this claim, we need to recall the following property. A Banach lattice $X$ is said to have the \emph{isometric lattice-lifting property} if there exists a lattice isometric embedding $\alpha: X \rightarrow \fbl[X]$ such that $\wh{id_X}\alpha=id_X$. This property was introduced in the article \cite{AMRT}, inspired by previous work of Godefroy and Kalton \cite{GK} on Lipschitz-free spaces. Notably, Banach lattices ordered by a $1$-unconditional basis have this property (see \cite[Theorem 4.1]{AMRT}). In \cite[Theorem 8.3]{OTTT} an alternative proof of this fact was given, which also worked when $\fbl[E]$ was replaced by $\fbp[E]$. We now show how to generalize these results to any class $\cC$.

\begin{prop}\label{prop: embedding and FBL constant}
    Let $\cC$ be a class of Banach lattices, $X$ a Banach space with a normalized 1-unconditional basis $(e_i)_i$, viewed as a Banach lattice with the order given by the basis, and $C\geq1$. The following are equivalent.
    \begin{enumerate}
        \item $X$ $C$-lattice embeds into some $\intoo[1]{\bigoplus_{\gamma\in \Gamma} X_\gamma}_\infty$ with $X_\gamma \in \cC$ for every $\gamma\in \Gamma$.
        \item For every operator $T:X\rightarrow X$, there exists a unique extension $\widehat T:\fbl^{\cC}[X]\rightarrow X$ satisfying $\|\widehat T\|\leq C\|T\|.$
        \item There exists a unique extension $\widehat{id_X}:\fbl^{\cC}[X]\rightarrow X$ of the identity $id_X:X\rightarrow X$ satisfying $\|\widehat{id_X}\|\leq C.$
        \item $X$ $C$-lattice embeds into $\fbl^\cC[X]$.
    \end{enumerate}
    Moreover, the $C$-lattice embedding $\alpha: X\rightarrow \fbl^\cC[X]$ in $(4)$ can be chosen so that $\wh{id_X}\alpha=id_X$.
\end{prop}

\begin{proof}
    \noindent $(1)\Rightarrow(2)$ is an adaptation of \cite[Proposition 2.5]{GLTT} and $(2)\Rightarrow (3)$ is clear.\medskip

    \noindent $(3)\Rightarrow (4)$ Recall that by either \cite[Theorem 4.1]{AMRT} or \cite[Theorem 8.3]{OTTT} with $p=1$ there is a lattice isometric embedding $\alpha_0:X\rightarrow \fbl[X]$, the free Banach lattice (which can be seen as the free Banach lattice associated to the class of all Banach lattices), such that $\widehat{id_X} \alpha_0=id_X$ (here, $\widehat{id_X}:\fbl[X]\rightarrow X$). Moreover, since $\cC$ is a class of Banach lattices, the formal identity $j:\fbl[X]\rightarrow \fbl^{\cC}[X]$ is a continuous norm one injection such that $\widehat{id_X} j \alpha_0=id_X$ (now, we are abusing the notation by also denoting $\widehat{id_X}:\fbl^\cC[X]\rightarrow X$ the extension of $id_X$ to $\fbl^\cC[X]$ given in $(3)$). We claim that $\alpha:=j \alpha_0:X\rightarrow \fbl^{\cC}[X]$ defines a $C$-lattice embedding. To prove the claim, simply note that for every $x\in X$
    \[\frac{1}{C}\|x\|\leq  \frac{1}{\|\widehat{id_X}\|}\|\widehat{id_X} \alpha x\| \leq \|\alpha x\|_{\fbl^\cC[X]}\leq \|\alpha_0 x\|_{\fbl[X]} \leq \|x\|.\]
    This also proves the additional claim.\medskip
    
    \noindent $(4)\Rightarrow (1)$ follows from \cite[Proposition 2.7]{GLTT}, as $\fbl^\cC[X]$ embeds lattice isometrically into an $\ell_\infty$-sum of spaces in $\cC$.
\end{proof}

We can apply \Cref{prop: embedding and FBL constant} to the Banach lattice $X$ constructed in Example \ref{ex: finite dim BL} to conclude that, since $X$ does not 1-lattice embed into $\intoo[1]{\bigoplus_{\mu\in \Gamma}L_{p,\infty}(\mu)}_\infty$ for any choice of measures $\Gamma$, the extension $\widehat{id_X}$ of the identity on $X$ to $\fbl^{\cX_{p,\infty}}[X]$ cannot be contractive. However, if we instead consider $\fbl^{\cC_\pinf}[X]$ (which by \cite[Theorem 3.6]{GLTT} coincides as a set with $\fbl^{\cX_{p,\infty}}[X]$, with an equivalent renorming), we observe that in this setting the norm of the extension $\widehat{id_X}$ becomes $\|id_X\|=1$, since $X$ belongs to the class $\cC_\pinf$. Therefore, the norms $\|\cdot \|_{\fbl^{\cX_{p,\infty}}[X]}$ and $\|\cdot \|_{\fbl^{\cC_{p,\infty}}[X]}$ cannot agree.

\subsection{The renormings of weak-\texorpdfstring{$L_p$}{}}\label{sec: renormings}

In this paper (and also in \cite{GLTT}), we have chosen to equip $L_{p,\infty}(\mu)$ with the norm
\[\|f\|_{L_{p,\infty}^{[1]}}:=\sup_{0<\mu(A)<\infty} \mu(A)^{\frac{1}{p}-1} \int_A |f|\,d\mu. \]
However, for each $1\leq r<p$, we could have equally chosen to equip $L_{p,\infty}(\mu)$ with the norm 
\[\|f\|_{L_{p,\infty}^{[r]}}:=\sup_{0<\mu(A)<\infty} \mu(A)^{\frac{1}{p}-\frac{1}{r}} \left(\int_A |f|^r\,d\mu\right)^{\frac{1}{r}}, \]
as it also defines a lattice norm on $L_{p,\infty}(\mu)$ with upper $p$-estimate constant $1$. For this reason, it is natural to ask whether the ambiguity in choosing a norm on $L_{p,\infty}(\mu)$ is the reason why $\gamma_p$ cannot be chosen to be $1$ in \cite[Proposition 3.3 and Theorem 3.6]{GLTT}.
\medskip

The aim of this subsection is to show that $(L_{p,\infty}(\mu),\|\cdot\|_{L_{p,\infty}^{[r]}})$ embeds lattice isometrically into a Banach lattice of the form
$\intoo[1]{\bigoplus_{\mu\in \Gamma}L_{p,\infty}^{[1]}(\mu)}_\infty$. As a consequence, we deduce that the above renormings cannot eliminate the constant $\gamma_p$ in  \cite[Proposition 3.3 and Theorem 3.6]{GLTT}. 
\medskip

Recall that every Banach lattice with upper $p$-estimates is $r$-convex for every $1\leq r<p$, with a constant that might explode like $(p-r)^{-\frac{1}{p}}$ as $r$ approaches $p$. Moreover, every $r$-convex Banach lattice has an equivalent lattice norm whose $r$-convexity constant becomes 1 \cite[Proposition 1.d.8]{LT2}. As we saw in \Cref{rem: interpretation renorming}, it turns out that in the case of $L_\pinf$, we can use the $L_{p,\infty}^{[r]}$-renorming to actually renorm $\wlp$ so that the upper $p$-estimates constant and the $r$-convexity constant are simultaneously 1. Now that we have clarified the role of the renormings, we can proceed with the main goal of this section. We begin with some notation and technical lemmata.
\medskip

Let $e = (1,\dots, 1)\in \R^n$ and let $(e_i)^n_{i=1}$ be the unit vector basis for $\R^n$. Let $\beta, b\in \R^n_+$, $\|\beta\|_1, \|b\|_1 \leq 1$, and let $0 < s < 1$. Define $\al : \R^n_+ \to \R$ by
\begin{equation}\label{eq: alpha}
    \al(x) = (\beta\cdot x)^{1-s}\, (b\cdot x)^s,
\end{equation} 
where $\cdot$ is the dot product on $\R^n$. Clearly, $\al$ is positively homogeneous.

\begin{lem}\label{lem A2}
There exists $d= (d_i)^n_{i=1} \in \R^n_+$ with $\|d\|_1 \leq 1$ such that $\al(x) \leq d\cdot x$ for any $x\in \R^n_+$.
\end{lem}

\begin{proof}
Let $d=(1-s)\beta+sb$, so that $d\in \R^n_+$ and $\|d\|_1 \leq 1$. Given $x\in \R^n_+$, by Young's inequality for products, we have that
\[\al(x) = (\beta\cdot x)^{1-s}\, (b\cdot x)^s \leq (1-s) (\beta\cdot x)+s(b\cdot x)=((1-s)\beta+sb)\cdot x =d\cdot x.\]
\end{proof}

\begin{lem}\label{lem weakpr into weakp1}
Let $(\Om,\Sigma,\mu)$ be a measure space. Suppose that $(a_i)^n_{i=1} \in \R^n_+$ and $(U_i)^n_{i=1}$ is a disjoint sequence in $\Sigma$ with $0< \mu(U_i)<\infty$ so that $C: =(\sum^n_{i=1}\mu(U_i))^{\frac{1}{p}-\frac{1}{r}} \|\sum^n_{i=1}a_i\chi_{U_i}\|_{L_r(\mu)} \leq 1$.
There is a measure $\nu$ on $\Sigma$ and a lattice homomorphism 
$S: L_{p,\infty}^{[r]}(\mu) \to L_{p,\infty}^{[1]}(\nu)$ so that $\|S\| \leq 1$ and $\|S(\sum^n_{i=1}a_i\chi_{U_i})\| \geq C^r$.
\end{lem}

\begin{proof}
Define $M = \sum^n_{i=1}\mu(U_i)$, $b_i = \frac{\mu(U_i)}{M}$, $b= (b_i)^n_{i=1}$ and 
$\beta_i = M^{\frac{r}{p}-1}\mu(U_i)a_i^r$, $\beta = (\beta_i)^n_{i=1}$.  Clearly, $\beta, b\in \R^n_+$ and $\|b\|_1 =1$. 
Moreover, 
\[ \|\beta\|_1 =M^{\frac{r}{p}-1}\,\norm[3]{\sum^n_{i=1}a_i\chi_{U_i}}^r_{L_r(\mu)} = C^r \leq 1.
\]
Take $s = p^*(\frac{1}{r}-\frac{1}{p})$ and define $\al$ as in \eqref{eq: alpha}. Note that $\al(e) = \|\beta\|_1^{1-s}\,\|b\|_1^s = C^{r(1-s)}\leq 1$. Let $d$ be given by \Cref{lem A2}.
Define a measure $\nu$ on $\Sigma$ by $\nu(A) = \sum^n_{i=1}\frac{d_i\,\mu(A\cap U_i)}{\mu(U_i)}$ and a linear operator 
\[ S: L_{p,\infty}^{[r]}(\mu) \to L_{p,\infty}^{[1]}(\nu) \quad \text{by}\quad Sf = M^{\frac{r}{p}}\sum^n_{i=1}\frac{a_i^{r-1}b_i}{d_i}f\chi_{U_i}.\]
Clearly, $S$ is a lattice homomorphism.
Note that $\nu(\bigcup^n_{i=1}U_i) = \|d\|_1 \leq 1$. Hence,
\begin{align*} 
\norm[3]{S\intoo[3]{\sum^n_{i=1}a_i\chi_{U_i}}} &\geq   M^{\frac{r}{p}}\norm[3]{\sum^n_{i=1}\frac{a_i^rb_i}{d_i}\chi_{U_i}}_{L_1(\nu)} = M^{\frac{r}{p}}\sum^n_{i=1}a^r_ib_i \\&= M^{\frac{r}{p}-1}\norm[3]{\sum^n_{i=1}a_i\chi_{U_i}}^r_{L_r(\mu)} = C^r.
\end{align*}
On the other hand, let $f\in L_{p,\infty}^{[r]}(\mu)$ with $\|f\| \leq 1$.  If $A\in \Sigma$ with $0 < \nu(A) <\infty$ then
\begin{align*}
 \int_A|Sf| \,d\nu & = M^{\frac{r}{p}}\sum^n_{i=1}\frac{a_i^{r-1}b_i}{d_i}\int_{A\cap U_i} |f| \,d\nu
=  M^{\frac{r}{p}-1}\sum^n_{i=1} a_i^{r-1}\int_{A\cap U_i} |f| \,d\mu\\
&\leq M^{\frac{r}{p}-1}\norm[3]{\sum^n_{i=1}a^{r-1}_i\chi_{A\cap U_i}}_{L_{r^*}(\mu)}\, \norm[3]{f\chi_{\bigcup^n_{i=1}(A\cap U_i)}}_{L_r(\mu)}\\
&\leq  \left(\sum^n_{i=1}\beta_i\frac{\,\mu(A\cap{U_i})}{\mu(U_i)}\right)^{\frac{1}{r^*}}\, \left(\sum^n_{i=1} b_i\frac{\,\mu(A\cap{U_i})}{\mu(U_i)}\right)^{\frac{1}{r}-\frac{1}{p}}.
\end{align*}
Let $x = \bigl(\frac{\mu(A\cap U_i)}{\mu(U_i)}\bigr)^n_{i=1} \in \R^n_+$.
Since $1-s =\frac{p^*}{r^*}$, it follows from \Cref{lem A2} that
\[ \int_A|Sf| \,d\nu \leq (\beta\cdot x)^{\frac{1-s}{p^*}}\, (b\cdot x)^{\frac{s}{p^*}} = \al(x)^{\frac{1}{p^*}} \leq (x\cdot d)^{\frac{1}{p^*}} = \nu(A)^{\frac{1}{p^*}}.
\]
Hence, $\|Sf\| \leq 1$.
\end{proof}

We are now in disposition of proving the claim that we made at the beginning of the subsection.

\begin{thm}
Suppose that $(\Om,\Sigma,\mu)$ is a measure space and  $1< r < p$.
There is a set $\Gamma$ of measures on $\Sigma$ so that $L_{p,\infty}^{[r]}(\mu)$ embeds lattice isometrically into $\intoo[1]{\bigoplus_{\nu \in \Gamma} L_{p,\infty}^{[1]}(\nu)}_{\infty}.$
\end{thm}

\begin{proof}
Let $\Gamma$ be the set of simple functions $f = \sum^n_{i=1}a_i\chi_{U_i}$, where $n\in \N$, $a_i >0$, $0< \mu(U_i) <\infty$ and $(U_i)^n_{i=1}$ is a disjoint sequence in $\Sigma$ with   
\[ C_f = \intoo[3]{\sum^n_{i=1}\mu(U_i)}^{\frac{1}{p}-\frac{1}{r}}\norm[3]{\sum^n_{i=1}a_i\chi_{U_i}}_r \leq 1.\]
For each $f\in \Gamma$, according to Lemma \ref{lem weakpr into weakp1}, there is a lattice homomorphism $S_f: L_{p,\infty}^{[r]}(\mu) \to L_{p,\infty}^{[1]}(\nu_f)$ such that $\|S_f\| \leq 1$ and $\|S_ff\| \geq C^r_f$. 
Define \[T:  L_{p,\infty}^{[r]}(\mu) \to \intoo[3]{\bigoplus_{f\in \Gamma}L_{p,\infty}^{[1]}(\nu_f)}_\infty \quad\text{by} \quad Tg = (S_fg)_{f\in \Gamma}.\]
Clearly, $T$ is a lattice homomorphism and $\|T\|\leq 1$.
Suppose that $g\in L_{p,\infty}^{[r]}(\mu)_+$ with $\|g\|_{L_{p,\infty}^{[r]}}=1$.
For any $\ep >0$, there exists a non-negative simple function $h$ such that $g \geq h$ and $\|h\|_{L_{p,\infty}^{[r]}} \geq 1-\ep$.
Write $h = \sum^n_{i=1}a_i\chi_{U_i}$.  There exists $I\subseteq \{1,\dots,n\}$ so that 
\[\intoo[3]{\sum_{i\in I}\mu(U_i)}^{\frac{1}{p}-\frac{1}{r}}\norm[3]{\sum_{i\in I}a_i\chi_{U_i}}_r = \|h\|_{L_{p,\infty}^{[r]}}  \geq 1-\ep.\]
Then $f : = \sum_{i\in I}a_i\chi_{U_i}\in \Gamma$ with $C_f \geq 1-\ep.$
Thus,
\[ \|Tg\|\geq \|S_fg\| \geq \|S_ff\| \geq C_f^r \geq (1-\ep)^r,
\]
so that $\|Tg\|\geq 1$.  This shows that $T$ is an isometric embedding.
\end{proof}

\section{Factorization of convex and concave operators}\label{sec: abstract factorizations}

There are several results in the literature that provide methods of factoring $(p,q)$-convex/concave operators through $(p,q)$-convex/concave Banach lattices. Namely, in \cite[Proposition 5]{Reisner} it is shown that every $q$-concave operator factors through a $q$-concave Banach lattice, and the bi-adjoint of every $p$-convex operator factors through a $p$-convex Banach lattice. In \cite[Theorem 2.8.7]{MN} one can find analogous statements for cone $q$-absolutely summing operators and $q$-superadditive Banach lattices (i.e., $(q,1)$-concave operators and Banach lattices with lower $q$-estimates) and for $p$-majorizing operators and $p$-subadditive Banach lattices (i.e., $(\pinf)$-convex operators and Banach lattices with upper $p$-estimates). Finally, in \cite{RT}, the approach of \cite{MN} was further explored to cover the setting of $q$-concavity and $p$-convexity, establishing the existence of canonical factorizations for $q$-concave and $p$-convex operators that satisfy certain maximality/minimality properties, as well as a duality theory for such factorizations.\medskip 

The aim of this section is to extend these results to an abstract setting of operators satisfying general convexity and concavity conditions, recovering as a particular case the analogues of \cite{RT} for $(\pinf)$-convex and $(q,1)$-concave operators. We begin by considering operators whose domain or target space is a Banach lattice, which include concave and convex operators as particular cases, respectively. Namely, in \Cref{sec: geometry factorization} we identify suitable classes of factorizations for these operators in terms of the sets that they induce in the corresponding Banach lattice, and we devote \Cref{sec: duality convexity} to introducing two types of convex sets associated to such operators and investigating their duality relations (see \Cref{thm: polarity sets C and D}). Next, in \Cref{sec: concavity and convexity}, we define $(\tau,\sigma)$-convex and concave operators and Banach lattices as a general version of \Cref{def: convexity and concavity} and use the sets introduced in the previous subsection to establish that such operators admit factorizations through such Banach lattices. These factorizations also satisfy certain minimality and maximality properties with respect to the classes of factorizations introduced in \Cref{sec: geometry factorization}, as was the case in \cite[Section 2]{RT} (see Theorems \ref{thm: minimal factorization tau sigma convex operator} and \ref{thm: maximal factorization tau sigma concave operator}). In particular, we improve \cite[Theorem 2.8.7]{MN} in the following way:

\begin{thm}\label{thm: optimal factorization upe lpe case}
    Let $E$ be a Banach space and $X$ a Banach lattice.
    \begin{enumerate}
        \item Let $T:E\rightarrow X$ be a $(p,\infty)$-convex operator. There exist a $(p,\infty)$-convex Banach lattice $Y_0$ with constant 1, and operators $U_0:E\to Y_0$ and $V_0:Y_0\to X$, so that $V_0$ is an injective, interval preserving  lattice homomorphism such that $V_0(B_{Y_0})$ is closed in $X$ (i.e., $V_0$ is of Class $\cC$) and $T = V_0U_0$. Moreover, if there exist a $(\pinf)$-convex Banach lattice $Y$ and operators $U:E\to Y$ and $V:Y\to X$ such that $V$ is of Class $\cC$ and $T=VU$, then there is a Class $\cC$ operator $\vp:Y_0\to Y$ so that $U = \vp U_0$ and $V_0 = V\vp$.\\
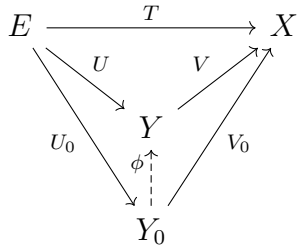
        \begin{center}
        \begin{tikzcd}
        E  \ar[dr, "U"] \ar[ddr, "U_0"'] \ar[rr, "T"]  &    & X  \\
        &  Y \ar[ur, "V"] & \\
        &  Y_0 \ar[uur, "V_0"']  \ar[u, dashed, "\vp" near end] & 
        \end{tikzcd} 
        \end{center}
        \item Let $T:X\rightarrow E$ be a $(p,1)$-concave operator. There exist a $(p,1)$-concave Banach lattice $Y^0$ with constant 1, and operators $U^0:X\to Y^0$ and $V^0:Y^0\to E$, so that $U^0$ is a lattice homomorphism with dense range (i.e., $U^0$ is of Class $\cD$) and $T = V^0U^0$. Moreover, if there exist a $(p,1)$-concave Banach lattice $Y$ and operators $U:X\to Y$ and $V:Y\to E$ such that $U$ is of Class $\cD$ and $T=VU$, then there is a Class $\cD$ operator $\vp:Y\to Y^0$ so that $U^0 = \vp U$ and $V^0\vp = V$.\\
        \begin{center}
        \begin{tikzcd}
        X  \ar[dr, "U"] \ar[ddr, "U^0"'] \ar[rr, "T"]  &    & E  \\
        &  Y \ar[ur, "V"] \ar[d, dashed, "\vp" near start]  & \\
        &  Y^0 \ar[uur, "V^0"']   & 
        \end{tikzcd} 
        \end{center}
    \end{enumerate}
\end{thm}

Moreover, we show in \Cref{sec: duality factorizations} that these canonical factorizations satisfy a duality relation similar to the one presented in \cite[Theorem 5]{RT}. Finally, in \Cref{sec: factorization of simultaneous conditions}, we apply the abstract theory developed in the previous sections to factor operators which are simultaneously $(p,p_2)$-convex and $(q,q_2)$-concave, for $p_2\in \{p,\infty\}$ and $q_2\in\{1,q\}$. The case $p_2=p$ and $q_2=q$ was already covered in \cite[Theorem 15]{RT}, whereas Theorems \ref{thm: factorization of pp2 convex and q1 concave operator} and \ref{thm: factorization of pinfty convex and qq2 concave operator} cover the remaining three cases.

\subsection{Geometry of factorization}\label{sec: geometry factorization}

We start by establishing suitable classes of factorizations for convex and concave operators. 
The definitions are motivated by results in \cite{RT}. Let $X,Y$ be Banach lattices and let $E$ be a Banach space. Recall that a positive linear operator $T:X\to Y$ is {\em interval preserving}, respectively {\em almost interval preserving}, if $T[0,x] = [0,Tx]$, respectively, $T[0,x]$ is norm dense in $[0,Tx]$, for any $x\in X_+$ \cite[Definition 1.4.18]{MN}.  A linear operator $T:X\to Y$ is of 
\begin{enumerate}
\item {\em Class $\cC$} if $T$ is an injective, interval preserving  lattice homomorphism such that $T(B_X)$ is closed in $Y$. Note that $T$ is necessarily bounded, and hence $T(B_X)$ is bounded in $Y$.
\item {\em Class $\cD$} if $T$ is a lattice homomorphism with dense range. Note that in particular $T$ is almost interval preserving.
\end{enumerate}

Let $T:E\to X$ be a bounded linear operator. A triple $(Y,U,V)$ is a {\em Class $\cC$ factorization} of $T$ if $Y$ is a Banach lattice, $U:E\to Y$ and $V:Y\to X$ are bounded linear operators so that $V$ is of class $\cC$ and $T = VU$. 
A triple $(Y,U,V)$ is a {\em Class $\cD$ factorization} of a bounded linear operator $T:X\to E$ if $Y$ is a Banach lattice, $U:X\to Y$ and $V:Y\to E$ are bounded linear operators so that $U$ is of class $\cD$ and $T = VU$. 
Two Class $\cC$, respectively Class $\cD$, factorizations $(Y_i,U_i,V_i)$, $i=1,2$, of $T$ are {\em equivalent} if there is a Banach lattice isomorphism $\iota:Y_1\to Y_2$ so that $U_2 = \iota U_1$ and $V_1 = V_2 \iota$. Two closed bounded subsets $B_1, B_2$ of $E$ are {\em equivalent} if there is a  positive  constant $c$ so that $c^{-1}B_1 \subseteq B_2 \subseteq cB_1$. As we will see in this section, there is a correspondence between Class $\cC$ factorizations of an operator $T:E\to X$ and subsets of $X$ satisfying certain properties, and a similar situation operates for Class $\cD$ factorizations. Before proceeding with the precise statements, let us recall the following well-known fact:

\begin{lem}\label{lem: minkowski functional}
Let $E$ be a Banach space and let $B$ be a closed bounded convex circled subset of $E$.
Then the Minkowski functional $\rho$ of $B$ defines a complete norm on $\spn B$ so that $B$ is the closed ball with respect to $\rho$.
\end{lem}

\begin{proof}
Let $C$ be a finite constant so that $B\subseteq CB_E$. Recall that the Minkowski functional of $B$ is given by $\rho(y)=\inf\{\lambda>0: y\in \lambda B\}$ for every $y\in E$ ($y\notin \spn B$ if and only if $\rho(y)=\infty)$.
If $y\in \spn B$ and $\rho(y) = 0$, then 
\[ y \in \bigcap_{\lambda >0} \lambda B\subseteq \bigcap_{\lambda>0}\lambda C B_E = \{0\}.\]
Hence $\rho$ is a norm on $\spn B$.\medskip

Suppose $(y_n)$ is a sequence in $\spn B$ so that $\sum \rho(y_n) <1$.
Then there exists a sequence $(\lambda_n)$ so that $y_n \in \lambda_n B$  for all $n$ and $\sum \lambda_n <1$.
In particular, $\|y_n\| \leq \lambda_nC$ and hence $\sum \|y_n\| <C$.  So $\sum y_n$ converges in $E$ to some $y$.
Since 
\[\sum^m_{n=1}y_n \in \sum^m_{n=1}\lambda_n\cdot B \subseteq B\] for all $m$ and $B$ is closed, $y\in B$.
If $k >m$,
\[ \sum^k_{n=m+1}y_n  \in \sum^k_{n=m+1}\lambda_n\cdot B \subseteq \sum^\infty_{n=m+1}\lambda_n\cdot B.
\]
We take the limit as  $k\to \infty $ in the space $E$. Then  $y - \sum^m_{n=1}y_n\in \sum^\infty_{n=m+1}\lambda_n\cdot B.$  Hence, $\rho(y - \sum^m_{n=1}y_n)\to 0$ as $m\to \infty$.
This proves that $\rho$ is complete.\medskip

If $y\in B$, then obviously $\rho(y) \leq 1$.
Conversely, suppose that $\rho(y) \leq 1$.  For any $\lambda>1$, $\frac{y}{\lambda}\in B$.
Since $B$ is closed in $E$, $y\in B$, as required.
\end{proof}

Regarding Class $\cC$ factorizations, we have the following correspondence:

\begin{prop}\label{prop: class C correspondence}
Let $T:E\to X$ be a nonzero bounded linear operator.
\begin{enumerate}
\item The map $(Y,U,V) \mapsto V(B_Y)$ is a surjective correspondence between the set of Class $\cC$ factorizations of $T$ and the set of  closed bounded convex solid subsets of $X$ that contain $cT(B_E)$ for some $c>0$. 
\item Let $(Y_i,U_i,V_i)$, $i=1,2$, be two Class $\cC$ factorizations of $T$. The following are equivalent:
\begin{enumerate}
\item There is a bounded linear map $\vp:Y_1\to Y_2$ such that $U_2 = \vp U_1$ and $V_1 = V_2 \vp $.  In this case, $\vp$ is necessarily of Class $\cC$.
\item There is a positive constant $c$ so that $V_1(B_{Y_1}) \subseteq cV_2(B_{Y_2})$.
\end{enumerate}
As a result, two Class $\cC$ factorizations $(Y_1,U_1,V_1)$ and $(Y_2,U_2,V_2)$ are equivalent if and only if $V_1(B_{Y_1})$ and $V_2(B_{Y_2})$ are equivalent.
\end{enumerate}
\end{prop}

\begin{proof}
(1)  If $(Y,U,V)$ is a Class $\cC$ factorization of $T$, then $V$ is of Class $\cC$ and hence $V(B_Y)$ is a closed bounded convex and solid subset of $X$. Since $T=VU$ is nonzero, $T(B_E) = VU(B_E) \subseteq \|U\|V(B_Y)$. Thus, $V(B_Y)$ contains $cT(B_E)$ with $c = \frac{1}{1+\|U\|}>0$.\medskip

Conversely, suppose that $B$ is a closed bounded convex solid subset of $X$ that contains $cT(B_E)$ for some $c>0$.
Let $Y = \spn B$ and let $\rho$ be the Minkowski functional of $B$ on $Y$.
Then $(Y,\rho)$ is a Banach space by Lemma \ref{lem: minkowski functional}. Since $B$ is solid, $Y$ is a Banach lattice in the order inherited from $X$. Define $U:E\to Y$ by $Ux =Tx$.  Since $T(E)\subseteq Y$, $U$ is a well-defined linear operator.
Moreover, $U(B_E) =T(B_E) \subseteq \frac{1}{c}B$. Hence $U$ is bounded.
Let  $V:Y\to X$ be the formal inclusion. 
Then $V$ is of class $\cC$.  Clearly, $T=VU$. Thus, $(Y,U,V)$ is a Class $\cC$ factorization of $T$.\medskip

We have shown that the correspondence $(Y,U,V)\mapsto V(B_Y)$ is well defined and surjective between the stated sets.
Let us verify injectivity. Assume that $(Y_i,U_i,V_i)$ are Class $\cC$ factorizations so that $V_1(B_{Y_1}) = V_2(B_{Y_2})$. Then, since $V_2$ is an injective lattice homomorphism, the operator $\iota= V_2^{-1}V_1:Y_1\rightarrow Y_2$ is a well-defined lattice homomorphism and $\iota B_{Y_1}=B_{Y_2}$, so it is contractive. Similarly, $\iota^{-1}= V_1^{-1}V_2: Y_2\rightarrow Y_1$ is a well-defined and contractive lattice homomorphism. It follows that $Y_1$ and $Y_2$ are isometric Banach lattices, and moreover $U_2 = \iota U_1$ and $V_1 = V_2 \iota$, so both Class $\cC$ factorizations are equivalent.

\medskip

\noindent (2) (a) $\implies$ (b)  From the assumption, we have $V_1(B_{Y_1}) = V_2\vp(B_{Y_1}) \subseteq \|\vp\|V_2(B_{Y_2})$. 
Let us verify that $\vp$ is of Class $\cC$.
Recall that $V_2$ is injective. For any $v\in Y_1$, 
\[V_2|\vp v| = |V_2\vp v| = |V_1v| = V_1|v| = V_2\vp|v|.\]
Hence $|\vp v| = \vp|v|$.
Similarly, if $v\in (Y_1)_+$, then
\[ V_2[0,\vp v] = [0,V_2\vp v] = [0,V_1v]= V_1[0,v] = V_2(\vp[0,v]).\]
Thus $[0,\vp v] = \vp[0,v]$.  This shows that $\vp$ is an interval preserving lattice homomorphism. It is injective since $V_1 = V_2\vp$ is injective.
Finally, $\vp(B_{Y_1}) = V_2^{-1}(V_1(B_{Y_1}))$ is closed since $V_1(B_{Y_1})$ is closed. Therefore, $\vp$ is of Class $\cC$.\medskip

\noindent (b) $\implies$ (a)  Since $V_1(Y_1) \subseteq V_2(Y_2)$ and $V_2$ is injective, $\vp=V^{-1}_2V_1:Y_1\to Y_2$ is a well-defined linear operator.
If $u\in B_{Y_1}$, then $\vp(u) = V_2^{-1}V_1u \in V_2^{-1}(V_1(B_{Y_1})) \subseteq cB_{Y_2}$.
Thus, $\vp$ is bounded.
By the definition of $\vp$, $V_1 = V_2\vp$.
If $x\in E$, then 
\[\vp U_1x = V^{-1}_2V_1U_1x = V^{-1}_2Tx = V^{-1}_2V_2 U_2x = U_2x.\]
Thus $\vp U_1 = U_2$.
The final statement follows by applying (2) in both directions.
\end{proof}

On the other hand, Class $\cD$ factorizations can be characterized as follows:

\begin{prop}\label{prop: class D correspondence}
Let $T:X\to E$ be a nonzero bounded linear operator.
\begin{enumerate}
\item The map $(Y,U,V)\mapsto U^{-1}(B_Y)$ is a surjective correspondence between the set of Class $\cD$ factorizations of $T$ and the set of closed convex solid subsets $S$ of $X$ so that  $c^{-1}B_X\subseteq S\subseteq cT^{-1}(B_E)$ for some $c\geq1$.
\item Let $(Y_i,U_i,V_i)$, $i=1,2$, be two Class $\cD$ factorizations of $T$. The following are equivalent:
\begin{enumerate}
\item There is a bounded linear map $\vp:Y_1\to Y_2$ such that $U_2 = \vp U_1$ and $V_1 = V_2 \vp $. In this case, $\vp$ is necessarily of Class $\cD$.
\item There is a positive constant $c$ so that $U^{-1}_1(B_{Y_1}) \subseteq cU^{-1}_2(B_{Y_2})$.
\end{enumerate}
As a result, two Class $\cD$ factorizations $(Y_1,U_1,V_1)$ and $(Y_2,U_2,V_2)$ are equivalent if and only if $U_1^{-1}(B_{Y_1})$ and $U_2^{-1}(B_{Y_2})$ are equivalent.
\end{enumerate}
\end{prop}

\begin{proof}
(1) Let $(Y,U,V)$ be a Class $\cD$ factorization of $T$.  
Since $U$ is a lattice homomorphism, $U^{-1}(B_Y)$ is a closed convex solid subset of $X$.
By the boundedness of $U$, $U(B_X) \subseteq \|U\|B_Y$. Hence $B_X \subseteq \|U\|U^{-1}(B_Y)$.
Note that $T\neq 0$ implies that $U\neq 0$, so $\|U\| >0$. \medskip
Also, $T(U^{-1}(B_Y)) = VU(U^{-1}(B_Y))= V(B_Y) \subseteq \|V\|B_E$.  Thus $U^{-1}(B_Y)\subseteq \|V\|T^{-1}(B_E)$.  Again, $\|V\| >0$ since $V\neq 0$.
\medskip

Conversely, suppose that $B$ is a closed convex solid subset of $X$ so that $c^{-1}B_X\subseteq B \subseteq T^{-1}(B_E)$ for some $c\geq 1$.
The Minkowski functional $\rho$ of $B$ is a lattice seminorm on $X$. The kernel $I:= \ker \rho$ is a vector lattice ideal in $X$; $\rho$ induces a lattice norm $\wh{\rho}$ on the vector lattice $X/I$ by
$\wh{\rho}(Qu) = \rho(u)$, where $Q:X\to X/I$ is the quotient map.
Let $Y$ be the completion of $X/I$ with respect to $\wh{\rho}$.
Then $Y$ is a Banach lattice.  
Let $U:= Q:X\to X/I \subseteq Y$. Then $U$ is a lattice homomorphism; in particular, $U$ is bounded.
By definition, $U$ has dense range. Hence $U$ is of Class $\cD$.
Define $V:X/I \to E$ by $V(Qu) = Tu$.
Note that $u\in I$ implies that $u\in \lambda B$ for all $\lambda >0$. Since $T(B)\subseteq B_E$, $Tu = 0$.
Thus, $V$ is well defined.
If $\wh{\rho}(Qu) <1$, then $u\in B$.  Hence, $\|Tu\| \leq 1$. So, $V$ is bounded.
Therefore, $V$ extends to a bounded linear operator from $Y$ into $E$, still denoted by $V$.
By definition, $VU = T$. Thus, $(Y,U,V)$ is a Class $\cD$ factorization of $T$. \medskip

To check the injectivity of the correspondence, assume that $(Y_i,U_i,V_i)$, $i=1,2$, are two Class $\cD$ factorizations of $T$ such that $U_1^{-1}(B_{Y_1})=U_2^{-1}(B_{Y_2})$, and let us define $\iota: U_1(X)\rightarrow Y_2$ by $\iota y=U_2x$ for every $y=U_1x\in U_1(X)$. Observe that
\[\ker U_1=\bigcap_{\lambda>0}\lambda U_1^{-1}(B_{Y_1})= \bigcap_{\lambda>0}\lambda U_2^{-1}(B_{Y_2})=\ker U_2.\]
Therefore, whenever $U_1x=U_1x'$ it follows that $U_2x=U_2x'$, and hence $\iota$ is well defined. Moreover, if $y=U_1x\in B_{Y_1}$, then $x\in U_1^{-1}(B_{Y_1})=U_2^{-1}(B_{Y_2})$ and $\iota y=U_2 x\in B_{Y_2}$, so $\iota$ is contractive. Since $U_1$ and $U_2$ are lattice homomorphisms, $U_1(X)$ is a sublattice, and for every $y=U_1x\in U_1(X)$, $|y|=U_1|x|$ and $\iota |y|=U_2|x|=|U_2x|=|\iota y|$, so $\iota$ is a lattice homomorphism. Finally, $U_1$ has dense range, so $\iota$ extends to $Y_1$ as a contractive lattice homomorphism. Interchanging the roles of $U_1$ and $U_2$, we can show that $\iota$ has an inverse operator, so it is an onto isometry from $Y_1$ into $Y_2$. Clearly, $U_2=\iota U_1$ and $V_1=V_2\iota$, so the factorizations are equivalent.

\medskip

\noindent (2) (a) $\implies$ (b) If $u\in X$ and $U_1u\in B_{Y_1}$, then $U_2u = \vp U_1u \in \vp(B_{Y_1}) \subseteq \|\vp\|B_{Y_2}$. Hence $U^{-1}_1B_{Y_1} \subseteq \|\vp\|U^{-1}_2(B_{Y_2})$.
Let us verify that $\vp$ is of Class $\cD$.
For any $u\in X$, 
\[ |\vp U_1 v| = |U_2v| = U_2|v| = \vp U_1|v| = \vp|U_1 v|.\]
Since $U_1(X)$ is dense in $Y_1$, this proves that $\vp$ is a lattice homomorphism.
As $\vp (Y_1) \supseteq  U_2 (X)$, $\vp$ has dense range in $Y_2$.  Therefore, $\vp$ is of Class $\cD$.\medskip

\noindent(b) $\implies$ (a)  If $u\in X$ and $U_1u = 0$, then $u\in \lambda U^{-1}_1(B_{Y_1})$ for all $\lambda >0$. Hence $u\in \lambda U^{-1}_2(B_{Y_2})$ for all $\lambda>0$. Thus, $U_2u =0$.
This shows that the map $\vp:U_1(X)\to Y_2$ given  by $\vp(U_1u) = U_2u$ is well defined.
Moreover, if $U_1u\in B_{Y_1}$, then $u\in U_1^{-1}(B_{Y_1}) \subseteq cU_2^{-1}(B_{Y_2})$. Hence
$\vp(U_1u) = U_2u \in cB_{Y_2}$.  Thus, $\vp:U_1(X) \to Y_2$ is bounded.  Denote its bounded extension to $Y_1$ by $\vp$ again.  
By definition, $\vp U_1 = U_2$.
If $u\in X$, then $V_2\vp U_1 u = V_2 U_2 u = Tu = V_1U_1 u$.
As $U_1(X)$ is dense in $Y_1$, $V_2\vp = V_1$. This completes the proof of the proposition.
\end{proof}

\subsection{Duality of convex sets}\label{sec: duality convexity}

Krivine's functional calculus (cf.~\cite[Section 1.d]{LT2}) allows us to define expressions such as $\intoo{\sum_{i=1}^n |x_i|^{p}}^\frac{1}{p}$ for a finite quantity of elements $x_1,\ldots,x_n$ in an arbitrary Banach lattice $X$. However, we could replace the $p$-sum of $n$ elements for an arbitrary $n\in \N$, by any norm $\sigma$ on $c_{00}$, as the only requirements for functional calculus to exist is that the function on $\R^n$ that we want to represent is positively homogeneous and continuous. Therefore, we will use this fact to study convex and concave operators and Banach lattices in a more general sense. \medskip

Throughout the following subsections, let $\sigma:c_{00}\to \R$ be a lattice norm on $c_{00}$ that is
\begin{enumerate}
\item {\em Symmetric}: $\sigma((s_i)_i) = \sigma((s_{\pi(i)})_i)$ for any permutation $\pi$ of $\N$.
\item {\em Normalized}: $\sigma(1,0,\dots) =1$.
\end{enumerate}
Denote the norm on $c_{00}$ dual to $\sigma$ by $\sigma^*$. That is, for any $(s_i)_i\in c_{00}$,
\[ \sigma^*((s_i)_i) = \sup\cbr[3]{\sum_i s_it_i: (t_i)_i\in c_{00}, \sigma((t_i)_i) \leq 1}.\]
It is clear that $\sigma^*$ is a symmetric normalized lattice norm on $c_{00}$.
Moreover, taking $\sigma^{**} = (\sigma^*)^*$, we have  $\sigma^{**} = \sigma$.
Let $(e_i)$ be the unit vector basis of $c_{00}$.
For each $n\in\N$, define $\sigma_n:\R^n\to \R$ by $\sigma_n(t_1,\dots,t_n) = \sigma(\sum^n_{i=1}t_ie_i)$. Then $\sigma_n$ is a monotone positively homogeneous continuous function. As a particular case, we can take $\sigma$ to be the norm of $\ell_p$ for some $1\leq p< \infty$ or the norm of $c_0$.\medskip

In what follows, let $X$ be a Banach lattice and $E$ be a Banach space.
Let $\tau$ be another symmetric normalized lattice norm on $c_{00}$. The following sets are the main objects of our considerations.

\begin{defn}\label{def: convex sets of T}
\begin{enumerate}
\item 
Let $T:E\to X$ be a bounded linear operator. Define
\[ C_T^{\tau,\sigma} = \cbr[3]{u\in X: \exists (x_i)^n_{i=1}\subseteq  E \text{ such that } |u| \leq \sigma_n(|Tx_1|,\dots,|Tx_n|), \tau\intoo[3]{\sum^n_{i=1}\|x_i\|e_i} \leq 1}.\]
\end{enumerate}
\item Let $T:X\to E$ be a bounded linear operator. Define 
\[ D^T_{\tau,\sigma} = \cbr[3]{u\in X: |u| \geq \sigma_n(|u_1|,\dots,|u_n|),\, (u_i)^n_{i=1}\subseteq  X \implies \tau \intoo[3]{\sum^n_{i=1}\|Tu_i\|e_i} \leq 1}.\]
\end{defn}

In the rest of the section, we consider the duality $\la X,X^*\ra$.  In particular, all polars are taken with respect to this dual pair. Our aim is to show \Cref{thm: polarity sets C and D}, which establishes that the sets defined above satisfy a duality relation. First, we need a linearization lemma. Let $K$ be a compact Hausdorff space, and let $\mu$ be a regular Borel measure on $K$.
Suppose that $(h_i)^n_{i=1}\subseteq C(K)$ and $(g_i)^n_{i=1}\subseteq L_1(\mu)$.
It follows from the uniqueness of function calculus that
\begin{align*} \sigma(h_1,\dots,h_n)(t) &= \sigma(h_1(t),\dots, h_ n(t)) \text{ for all $t\in K$, and }\\ \sigma^*(g_1,\dots, g_n)(t) &= \sigma^*(g_1(t),\dots,g_n(t)) \text{ for $\mu$-almost all $t\in K$.}
\end{align*}

\begin{lem}\label{lem: bound sigma* in C(K)}
Let $\mu$ be a regular Borel measure on a compact Hausdorff space $K$.
Given $(g_i)^n_{i=1}\subseteq L_1(\mu)$ and $\ep >0$, there is a sequence $(h_i)^n_{i=1}$ in $C(K)$ such that 
$\sigma_n(h_1(t),\dots,h_n(t)) \leq 1$ for all $t\in K$, and 
\[ 
\int \sigma_n^*(g_1,\dots,g_n)\,d\mu \leq \sum^n_{i=1}\int h_i g_i\,d \mu + \ep.
\]
\end{lem}

\begin{proof}
First, assume that the functions $(g_i)^n_{i=1}\subseteq L_1(\mu)$ are simple. Without loss of generality, we can assume that there is a Borel measurable partition $(A_j)^k_{j=1}$ so that each $g_i = \sum^k_{j=1}a_{ij}\chi_{A_j}$ for some $a_{ij}$ (not necessarily positive). In this case, given $\ep >0$, by the regularity of the measure $\mu$ there exist closed sets $(C_j)^k_{j=1}$ and open sets $(U_j)^k_{j=1}$ such that $C_j\subseteq A_j\subseteq U_j$ and $\mu(U_j\setminus C_j)<\frac{\ep}{4ks}$, where $s=\sup_j \sigma^*_n(a_{1j},\dots, a_{nj}) >0$ (recall that $\sigma_n^*$ is a lattice norm, so $\sigma^*_n(a_1,\dots, a_n)=\sigma^*_n(|a_1|,\dots, |a_n|)$). Since the sets $(C_j)^k_{j=1}$ are pairwise disjoint and $K$ is compact and Hausdorff, and thus, normal, there exist pairwise disjoint open neighborhoods $(V_j)^k_{j=1}$ with $C_j\subseteq V_j$ for each $j$. Let $W_j=U_j\cap V_j$ and $\widetilde g_i=\sum^k_{j=1}a_{ij}\chi_{W_j}$. Then $\mu (W_j\Delta A_j)\leq \mu(U_j\setminus C_j)$, where $A\Delta B$ denotes the symmetric difference of $A$ and $B$, so
\[\int |\sigma_n^*(g_1,\dots,g_n)-\sigma_n^*(\widetilde g_1,\dots,\widetilde g_n) |\,d\mu \leq \sum_{j=1}^k \sigma^*_n(a_{1j},\dots, a_{nj}) \mu (W_j\Delta A_j) \leq \frac{\ep}{4} . \]

Now, choose $b_{ij}$ so that $b_{ij}a_{ij} \geq 0$, $\sigma_n(b_{1j},\dots,b_{nj}) \leq 1$ and 
\[\sigma^*_n(a_{1j},\dots, a_{nj}) \leq \sum^n_{i=1}a_{ij}b_{ij}+\frac{\ep}{4\mu(K)}\quad \text{ for $1\leq j\leq k$.}\]
Since $K$ is normal, there are functions $(u_j)_{j=1}^k\subseteq C(K)$ such that $0\leq u_j\leq \chi_{W_j}$ and $u_j=1$ on $C_j$. The functions $h_i=\sum^k_{j=1}b_{ij}u_j$ satisfy $|h_i|\leq \sum^k_{j=1}|b_{ij}|\chi_{W_{j}}$ and, moreover,
\begin{align*}
    0\leq \sum^k_{j=1}a_{ij}b_{ij}\mu(W_j) & =\int \intoo[3]{\sum^k_{j=1}b_{ij}\chi_{W_{j}}} \widetilde g_i\,d\mu   = \int h_i\widetilde g_i\,d\mu + \int \intoo[3]{\sum^k_{j=1}b_{ij}(\chi_{W_{j}}-u_j)}\widetilde g_i\,d\mu \\
    & \leq \int h_i g_i\,d\mu +\int |h_i||g_i-\widetilde g_i|\,d\mu + \sum^k_{j=1}\int a_{ij}b_{ij}(\chi_{W_{j}}-u_j)\chi_{W_j\setminus C_j} \,d\mu \\
    & \leq \int h_i g_i\,d\mu + \sum^k_{j=1}a_{ij}b_{ij}(\mu(W_j\Delta A_j)+\mu(W_j\setminus C_j))  \\
    & \leq \int h_i g_i\,d\mu +\frac{\ep}{2ks}\sum^k_{j=1}a_{ij}b_{ij}  \quad\text{for $1\leq i\leq n$.}
\end{align*}
Observe that, if $t\in W_j$, 
\[ \sigma_n(h_1(t),\dots, h_n(t)) \leq \sigma_n(b_{1j},\dots, b_{nj})\leq 1,
\]
so $\sigma(h_1(t),\dots,h_n(t)) \leq 1$ for all $t\in K$.
Finally,
\begin{align*}
\int \sigma^*(g_1,\dots,g_n)\,d\mu & \leq \int \sigma^*(\widetilde g_1,\dots,\widetilde g_n)\,d\mu +\frac{\ep}{4}  = \sum^k_{j=1}\int \sigma^*_n(a_{1j},\dots,a_{nj})\chi_{W_j}\,d\mu +\frac{\ep}{4}\\
&\leq \sum^k_{j=1}\left(\sum^n_{i=1}a_{ij}b_{ij}+\frac{\ep}{4\mu(K)}\right)\mu(W_j)+\frac{\ep}{4} \leq  \sum^n_{i=1}\sum^k_{j=1}a_{ij}b_{ij}\mu(W_j) +\frac{\ep}{2}\\
&\leq \sum^n_{i=1}\int h_ig_i \,d\mu+ \frac{\ep}{2ks}\sum^k_{j=1} \sum^n_{i=1}a_{ij}b_{ij} +\frac{\ep}{2} \\
& \leq \sum^n_{i=1}\int h_ig_i \,d\mu+ \frac{\ep}{2ks}\sum^k_{j=1}  \sigma^*_n(a_{1j},\dots, a_{nj}) +\frac{\ep}{2} \leq  \sum^n_{i=1}\int h_ig_i \,d\mu+ \ep.
\end{align*}
This completes the proof for simple functions. Now, given general functions $(g_i)^n_{i=1}\subseteq L_1(\mu)$, for every $i=1,\dots,n$ there exists a sequence $(s_{im})_{m=1}^\infty$ of simple functions such that $s_{im}^{+}\uparrow g_i^{+}$ and $s_{im}^{-}\uparrow g_i^{-}$ almost everywhere, so in particular $\sigma^*(s_{1m},\dots,s_{nm})\uparrow \sigma^*(g_1,\dots,g_n)$ almost everywhere. By the dominated convergence theorem, for every $\ep>0$ there is a natural number $M\in\N$ such that 
\begin{align*}
    & \int |\sigma^*(g_1,\dots,g_n)-\sigma^*(s_{1M},\dots,s_{nM})| \,d\mu \leq \frac{\ep}{n+2}, \,\,\text{and}\\
    & \int |g_i-s_{iM}|\,d\mu \leq \frac{\ep}{n+2} \quad\text{for $1\leq i\leq n$.}
\end{align*}
By the first part of the proof, for this $M$ there exist some $(h_i)_{i=1}^n\subseteq C(K)$ such that we have $\sigma_n(h_1(t),\dots,h_n(t)) \leq 1$ for all $t\in K$ and
\[\int \sigma_n^*(s_{1M},\dots,s_{nM})\,d\mu \leq \sum^n_{i=1}\int h_i s_{iM}\,d \mu + \frac{\ep}{n+2}.\]
In particular, $h_i\in B_{C(K)}$ and we have
\[\int |h_i||g_i-s_{iM}|\,d\mu \leq \int |g_i-s_{iM}|\,d\mu \leq \frac{\ep}{n+2}\]
for every $i$. In order to obtain the result, we just use the previous inequalities:
\begin{align*}
    \int \sigma^*(g_1,\dots,g_n)\, d\mu & \leq \int \sigma^*(s_{1M},\dots,s_{nM})\, d\mu + \frac{\ep}{n+2}\\
    & \leq \sum^n_{i=1}\int h_i s_{iM}\,d \mu + \frac{2\ep}{n+2}  \leq \sum^n_{i=1}\int h_i g_i\,d \mu + \ep.
\end{align*}
\end{proof}

We will use the previous lemma to linearize expressions of the form $ \la |u|, \sigma^*_n(|u^*_1|,\dots,|u^*_n|) \ra $ for $u\in X$ and $u_1^*,\ldots, u_n^*\in X^*$.

\begin{prop}\label{prop: duality sigma} The following statements hold:
\begin{enumerate}
\item Suppose that  $u^*\in X^*$ and $(u_i)^n_{i=1}\subseteq X$. For any $\ep >0$,
there are $(u_i^*)^n_{i=1} \subseteq X^*$ such that $\sigma_n^*(|u_1^*|,\dots, |u_n^*|) \leq |u^*|$ and 
\[ \la \sigma_n(|u_1|,\dots,|u_n|), |u^*|\ra \leq \sum_{i=1}^n\la u_i,u^*_i\ra +\ep.\]
\item Suppose that  $u\in X$ and $(u_i^*)^n_{i=1}\subseteq X^*$. For any $\ep >0$,
there are $(u_i)^n_{i=1} \subseteq X$ such that $\sigma_n(|u_1|,\dots, |u_n|) \leq |u|$ and 
\[ \la |u|, \sigma^*_n(|u^*_1|,\dots,|u^*_n|) \ra \leq \sum_{i=1}^n\la u_i,u^*_i\ra +\ep.\]
\end{enumerate}
\end{prop}

\begin{proof}
It suffices to prove (2). (1) follows by applying (2) with $X$ taken to be $X^*$ and $\sigma$ taken to be $\sigma^*$. \medskip

Let $u\in X$, $(u_i^*)^n_{i=1}\subseteq X^*$, and $\ep >0$ be given.
There exists a compact Hausdorff space $K$ and an interval preserving injective  lattice homomorphism $\iota:C(K) \to X$ so that $\iota1 = |u|$, where $1$ denotes the constant $1$ function on $K$.
Set $u^* = \sigma^*_n(|u_1^*|,\dots, |u^*_n|)$.  Since $|u^*_i| \leq u^*$ and $\iota^*$ is a lattice homomorphism, 
$|\mu_i| :=\iota^*|u^*_i| \leq \iota^*u^* =: \mu$.
Hence, there are functions $g_i\in L_1(\mu)$ such that $\mu_i = g_i\cdot\mu$.
Apply  Lemma \ref{lem: bound sigma* in C(K)} to obtain the corresponding sequence $(h_i)^n_{i=1}$ in $C(K)$
and set $u_i = \iota h_i$.  Then $u_i\in X$ and 
\[ \sigma_n(|u_1|,\dots,|u_n|) = \sigma_n(\iota |h_1|,\dots,\iota |h_n|) = \iota\sigma_n(|h_1|,\dots,|h_n|) \leq \iota1 = |u|. \]
Moreover, 
\begin{align*}
 \la |u|,\sigma^*_n (|u^*_1|,\dots,|u^*_n|)\ra & = \la 1, \sigma^*_n(\iota^*|u^*_1|,\dots, \iota^*|u^*_n|)\ra = \la 1, \sigma^*_n(\iota^*u^*_1,\dots, \iota^*u^*_n)\ra\\ & = \int \sigma_n^*(g_1,\dots,g_n)\,d\mu \leq \sum^n_{i=1}\int h_ig_i\,d\mu + \ep\\
 & = \sum^n_{i=1}\la h_i, \iota^*u^*_i\ra + \ep = \sum^n_{i=1}\la u_i,u^*_i\ra +\ep.
\end{align*}
This completes the proof.
\end{proof}

Next, we show that $D^{T^*}_{\tau^*,\sigma^*} \subseteq (C^{\tau,\sigma}_T)^\circ$ (respectively, $D^{T}_{\tau,\sigma} \subseteq (C^{\tau^*,\sigma^*}_{T^*})^\circ$) whenever $T:E\rightarrow X$ (respectively, $T:X\rightarrow E$) is bounded.

\begin{prop}\label{prop: duality sets C and D} The following statements hold:
\begin{enumerate}
\item Let $T:E\to X$ be a bounded linear operator. Then $u^*(u) \leq 1$ for all $u\in C^{\tau,\sigma}_T$ and all $u^* \in D^{T^*}_{\tau^*,\sigma^*}$.
\item Let $T:X\to E$ be a bounded linear operator. Then $u^*(u) \leq 1$ for all $u\in D^T_{\tau,\sigma}$ and all $u^* \in C_{T^*}^{\tau^*,\sigma^*}$.
\end{enumerate}
\end{prop}

\begin{proof}
(1) Suppose that $u\in C^{\tau,\sigma}_T$ and $u^* \in D^{T^*}_{\tau^*,\sigma^*}$.
There exists $(x_i)^n_{i=1}$ in $E$ so that $|u| \leq \sigma_n(|Tx_1|,\dots,|Tx_n|)$ and $\tau(\sum^n_{i=1}\|x_i\|e_i) \leq 1$.
By Proposition \ref{prop: duality sigma}(1), for any $\ep > 0$, there exists $(u^*_i)^n_{i=1}\subseteq X^*$ so that $\sigma^*_n(|u_1^*|,\dots,|u^*_n|)\leq |u^*|$ and 
\begin{align*}
     \la |u|, |u^*|\ra & \leq  \la \sigma_n(|Tx_1|,\dots,|Tx_n|), |u^*|\ra\leq \sum^n_{i=1}\la Tx_i, u^*_i\ra + \ep\\
     & = \sum^n_{i=1}\la x_i, T^*u^*_i\ra + \ep \leq \tau\intoo[3]{\sum^n_{i=1}\|x_i\|e_i}\cdot \tau^*\intoo[3]{\sum^n_{i=1}\|T^*u^*_i\|e^*_i}+ \ep.
\end{align*}
Since $u^* \in D^{T^*}_{\tau^*,\sigma^*}$ and $|u^*| \geq \sigma_n^*(|u^*_1|,\dots,|u^*_n|)$, $\tau^*(\sum^n_{i=1}\|T^*v_i^*\|e_i) \leq 1$.
Therefore, $\la |u|,|u^*|\ra \leq 1$, as claimed.\medskip
 
 The proof of (2) is similar. 
\end{proof}

The next lemma will help us show the reverse inclusions.

\begin{lem}\label{lem: sum of evaluations bounded by sigma}
Let $(u_i)^n_{i=1}$ and $(u^*_i)^n_{i=1}$ be sequences in $X$ and $X^*$, respectively.
Then 
\[\sum^n_{i=1}\la u_i,u^*_i\ra \leq  \la \sigma_n(|u_1|,\dots,|u_n|),\sigma_n^*(|u^*_1|,\dots, |u^*_n|)\ra.\]
\end{lem}

\begin{proof}
Take a compact Hausdorff space $K$ and an injective interval preserving lattice homomorphism $\iota:C(K)\to X$ so that $\iota 1 = \sigma_n(|u_1|,\dots,|u_n|)$.
Since $|u_i| \leq \sigma_n(|u_1|,\dots, |u_n|)$, there exists $f_i\in C(K)_+$ so that  $\iota f_i = |u_i|$.  Set $\mu = \iota^*\sigma^*_n(|u^*_1),\dots,|u^*_n|)$ and $\mu_i = \iota^*(|u^*_i|)$.
Then $0\leq f_i\leq 1$ and $0\leq \mu_i\leq \mu$.  There are functions $g_i\in L_1(\mu)_+$ so that $\mu_i = g_i\cdot \mu$.
Then
\begin{align*}
\sum^n_{i=1}\la u_i,u^*_i\ra & \leq  \sum^n_{i=1}\la |u_i|,|u^*_i|\ra = \sum^n_{i=1}\int f_i\,d\mu_i = \int \sum^n_{i=1} f_i(t)g_i(t)\,d\mu(t)\\
&\leq \int \sigma_n(f_1(t),\dots,f_n(t))\cdot \sigma^*_n(g_1(t),\dots, g_n(t))\,d\mu\\
&= \la \sigma_n(f_1,\dots, f_n), \sigma^*_n(\mu_1,\dots,\mu_n)\ra\\
&= \la \sigma_n(|u_1|,\dots, |u_n|), \sigma^*_n(|u^*_1|,\dots,|u^*_n|)\ra.
\end{align*}
\end{proof}

We are now ready to establish the main result of this subsection.

\begin{thm}\label{thm: polarity sets C and D}  The following statements hold:
\begin{enumerate}
\item Let $T:E\to X$ be a bounded linear operator.
Then $(C^{\tau,\sigma}_T)^\circ = D^{T^*}_{\tau^*,\sigma^*}$.
\item Let $T:X\to E$ be a bounded linear operator.
Then $(C^{\tau^*,\sigma^*}_{T^*})^\circ = D^T_{\tau,\sigma}$.
\end{enumerate}
\end{thm}

\begin{proof}
(1) It follows from Proposition \ref{prop: duality sets C and D}(1) that $D^{T^*}_{\tau^*,\sigma^*} \subseteq (C^{\tau,\sigma}_T)^\circ$.
Conversely, suppose that $u^*\notin D^{T^*}_{\tau^*,\sigma^*}$.
There is a sequence $(u^*_i)^n_{i=1}$ so that $|u^*| \geq \sigma_n^*(|u^*_1|,\dots,|u^*_n|)$ and $\tau^*(\sum^n_{i=1}\|T^*u^*_i\|e_i)> 1$.
Choose $x_i\in E$ with  $\|x_i\|\leq 1$ and a real sequence $(\al_i)^n_{i=1}$ with $\tau(\sum^n_{i=1}\al_ie_i)\leq 1$ so that $\sum^n_{i=1}\al_i\la x_i,T^*u_i^*\ra >1$.
Set $u_i = \al_iTx_i$. 
Using Lemma \ref{lem: sum of evaluations bounded by sigma}, we have
\[1 < \sum^n_{i=1}\la u_i, u^*_i\ra \leq \la \sigma_n(|u_1|,\dots,|u_n|),\sigma_n^*(|u^*_1|,\dots, |u^*_n|)\ra \leq \la \sigma_n(|u_1|,\dots,|u_n|), |u^*|\ra.\]
However, $\tau(\sum^n_{i=1}\|\al_ix_i\|e_i) \leq \tau(\sum^n_{i=1}\al_ie_i)\leq 1$.
Hence, $\sigma_n(|u_1|,\dots,|u_n|) \in C^{\tau,\sigma}_T$.
This shows that $u^* \notin (C^{\tau,\sigma}_T)^\circ$.\medskip

(2) can be proved in the same manner.
\end{proof}

\subsection{Convexity and concavity}\label{sec: concavity and convexity}
Let $\sigma$ and $\tau$ be symmetric and normalized lattice norms on $c_{00}$, $X$ a Banach lattice and $E$ a Banach space. As we advanced earlier, the following definition generalizes \Cref{def: convexity and concavity}.
\begin{defn}\label{def: tau sigma convex and concave}
\begin{enumerate}
\item[]
\item A bounded linear operator $T: E\to X$ is $(\tau,\sigma)$-convex with constant $K$ if for any sequence $(x_i)^n_{i=1}$ in $E$, 
\[ \|\sigma_n(|Tx_1|,\dots, |Tx_n|)\| \leq  K\tau\intoo[3]{\sum^n_{i=1}\|x_i\|e_i}.
\]
A Banach lattice $X$ is $(\tau,\sigma)$-convex with constant $K$ if the identity operator on $X$ is $(\tau,\sigma)$-convex with constant $K$.
\item A bounded linear operator $T: X\to E$ is $(\tau,\sigma)$-concave with constant $K$ if for any sequence $(u_i)^n_{i=1}$ in $X$, 
\[ \|\sigma_n(|u_1|,\dots, |u_n|)\| \geq  K^{-1}\tau\intoo[3]{\sum^n_{i=1}\|Tu_i\|e_i}.
\]
A Banach lattice $X$ is $(\tau,\sigma)$-concave with constant $K$ if the identity operator on $X$ is $(\tau,\sigma)$-concave with constant $K$.
\end{enumerate}
\end{defn}

It is a well-known fact (see, for instance, \cite[Chapter 16]{DJT}) that an operator $T$ is $(p,q)$-convex if and only if its adjoint $T^*$ is $(p^*,q^*)$-concave, and vice versa. Similarly, Theorem \ref{thm: polarity sets C and D} leads to a duality theory for the convex and concave operators introduced in \Cref{def: tau sigma convex and concave}.

\begin{thm}\label{thm: adjoint of convex is concave}
Let $T:E\to X$ be a $(\tau,\sigma)$-convex operator with constant $K$.
Then $T^*$ is a $(\tau^*,\sigma^*)$-concave operator with constant $K$.
\end{thm}

\begin{proof}
Let $(u_i^*)^n_{i=1}$ be a sequence in $X^*$ and let $\ep >0$ be given.
Choose a sequence $(x^*_i)^n_{i=1}$ in $B_E$ so that 
\[ \la Tx_i,u^*_i\ra  = \la x_i,T^*u^*_i\ra \geq \|T^*u^*_i\|- \frac{K\ep}{n} \text{ for any $i$}.
\]
By definition of $\tau^*$, there exists a positive real sequence $(a_i)^n_{i=1}$ so that $\tau(\sum^n_{i=1}a_ie_i) \leq 1$
and that 
\[\tau^*\intoo[3]{\sum^n_{i=1}\|T^*u^*_i\|e_i} \leq \sum^n_{i=1}a_i\|T^*u^*_i\| +K\ep.\]
Since $T$ is $(\tau,\sigma)$-convex with constant $K$, 
\[ \|\sigma_n(a_1|Tx_1|,\dots, a_n|Tx_n|)\| \leq K\tau\intoo[3]{\sum^n_{i=1}a_i\|x_i\|e_i} \leq K\tau\intoo[3]{\sum^n_{i=1}a_ie_i}\leq K.\]
Therefore,
\begin{align*}
    \|\sigma^*_n(|u^*_1|,\dots, |u^*_n|)\| & \geq K^{-1}\la \sigma_n(a_1|Tx_1|,\dots, a_n|Tx_n|), \sigma^*_n(|u^*_1|,\dots, |u^*_n|)\ra\\
    & \geq K^{-1}\sum^n_{i=1}a_i\la Tx_i,u^*_i\ra \quad \text{by Lemma \ref{lem: sum of evaluations bounded by sigma}}\\
    & \geq K^{-1}\sum^n_{i=1}a_i\intoo[2]{\|T^*u^*_i\|-\frac{K\ep}{n}} \geq K^{-1}\sum^n_{i=1}a_i\|T^*u^*_i\| -\ep\\
    &\geq K^{-1}\tau^*\intoo[3]{\sum^n_{i=1}\|T^*u^*_i\|e_i} - 2\ep.
\end{align*}
Taking $\ep\downarrow 0$ give the required inequality for $(\tau^*,\sigma^*)$-concavity of $T^*$ with constant $K$.
\end{proof}

\begin{thm}\label{thm: adjoint of concave is convex}
Let $T:X\to E$ be a $(\tau,\sigma)$-concave operator with constant $K$. 
Then $T^*$ is a $(\tau^*,\sigma^*)$-convex operator with constant $K$.
\end{thm}

\begin{proof}
    Let $(x_i^*)_{i=1}^n\subseteq E^*$, $\ep>0$ and $u\in B_X$ be given. By Proposition \ref{prop: duality sigma}(2), there is a sequence $(u_i)_{i=1}^n\subseteq X$ such that $\sigma_n(|u_1|,\dots,|u_n|)\leq |u|$ and
    \begin{align*}
        \la|u|,\sigma^*_n(|T^*x^*_1|,\dots,|T^*x^*_n|)\ra & \leq \sum_{i=1}^n \la u_i,T^*x^*_i\ra +\ep = \sum_{i=1}^n \la Tu_i,x^*_i\ra +\ep  \\
        & \leq \tau\intoo[3]{\sum_{i=1}^n \| Tu_i\|e_i} \tau^*\intoo[3]{\sum_{i=1}^n \|x^*_i\|e_i} +\ep \leq K\tau^*\intoo[3]{\sum_{i=1}^n \|x^*_i\|e_i} +\ep,
    \end{align*}
    where the last inequality follows from the concavity of $T$, since
    \[\tau\intoo[3]{\sum_{i=1}^n \| Tu_i\|e_i}\leq K \|\sigma_n(|u_1|,\dots,|u_n|)\| \leq K\|u\| \leq K.\]
    Taking $\ep \downarrow 0$ we conclude that
    \[\|\sigma^*_n(|T^*x^*_1|,\dots,|T^*x^*_n|)\|=\sup_{u\in B_X} \la|u|,\sigma^*_n(|T^*x^*_1|,\dots,|T^*x^*_n|)\ra\leq K\tau^*\intoo[3]{\sum_{i=1}^n \|x^*_i\|e_i},\]
    so $T^*$ is a $(\tau^*,\sigma^*)$-convex operator with constant $K$.
\end{proof}

Combining both results, we get:

\begin{cor}\label{cor: duality convex concave operators}  
    \begin{enumerate}
    \item[]
        \item $T:E\to X$ is $(\tau,\sigma)$-convex with constant $K$ if and only if $T^*:X^*\rightarrow E^*$ is $(\tau^*,\sigma^*)$-concave with constant~$K$.
        \item $T:X\to E$ is $(\tau,\sigma)$-concave with constant $K$ if and only if $T^*:E^*\rightarrow X^*$ is $(\tau^*,\sigma^*)$-convex with constant~$K$.
    \end{enumerate}
\end{cor}

\begin{proof}
    (1) One implication was established in \Cref{thm: adjoint of convex is concave}. For the reverse implication, we can apply \Cref{thm: adjoint of concave is convex} to obtain that $T^{**}$ is $(\tau,\sigma)$-convex with constant $K$. If we denote by $J_X$ and $J_E$ the canonical embeddings into the corresponding bidual spaces, we get that the mapping $J_XT=T^{**}J_E:E\rightarrow J_X(X)\subseteq X^{**}$ is also $(\tau,\sigma)$-convex with the same constant. Since $J_X$ is a lattice isometry, we get the claimed result by composing $J_XT$ with $J_X^{-1}$.\medskip

    \noindent (2) Again, one implication was established in \Cref{thm: adjoint of concave is convex}, and the converse follows by applying \Cref{thm: adjoint of convex is concave} to $T^*$ to obtain that $T=J_E^{-1}J_E T=J_E^{-1}T^{**}J_X$ is $(\tau,\sigma)$-concave with constant $K$.
\end{proof}

We say that a symmetric lattice norm $\rho$ on $c_{00}$ is {\em block convex with constant $A$}, respectively {\em block concave with constant $A$},  if $\rho(\sum_i y_i) \leq A\rho(\sum_i \rho(y_i)e_i)$, respectively $\rho(\sum_i y_i) \geq A^{-1}\rho(\sum_i \rho(y_i)e_i)$, for any finite disjoint sequence $(y_i)_i$ in $c_{00}$.
It is easy to see that if a symmetric lattice norm $\rho$ is block convex with constant $A$, respectively block concave with constant $A$, then its dual norm $\rho^*$ is block concave, respectively block convex with the same constant. Observe that, for every $1\leq p< \infty$, the norm in $\ell_p$ is both block concave and block convex with constant 1, and the same can be said about the norm in $c_0$.\medskip

The following proposition shows that, given an operator $T:E\rightarrow X$, the Banach lattice generated by the set $(C^{\tau,\sigma}_T)^{\circ\circ}$ with the procedure described in \Cref{lem: minkowski functional} is $(\tau,\sigma)$-convex.

\begin{prop}\label{prop: double polar of C is tau-sigma convex}
Let $T:E\rightarrow X$ be a bounded linear operator, and assume that $\sigma$ is block concave with constant $A_\sigma$ and $\tau$ is block convex with constant $A_\tau$. Let $(u_i)^n_{i=1}$ be a sequence in $(C^{\tau,\sigma}_T)^{\circ\circ}$ and let $(\al_i)^n_{i=1}$ be a real sequence so that $\tau(\sum^n_{i=1}\al_ie_i) \leq 1$.
Then  
$\sigma_n(|\al_1u_1|,\dots,|\al_nu_n|) \in A_\sigma A_\tau\cdot (C^{\tau,\sigma}_T)^{\circ\circ}$.
\end{prop}

\begin{proof}
Let $u^*\in D^{T^*}_{\tau^*,\sigma^*}$ and let $\ep >0$ be given.
By Proposition \ref{prop: duality sigma}(1), there exists $(u^*_i)^n_{i=1}\subseteq X^*$ such that 
$\sigma_n^*(|u_1^*|,\dots,|u^*_n|) \leq |u^*|$ and 
\[ \la \sigma_n(|\al_1u_1|,\dots,|\al_nu_n|),|u^*|\ra \leq\sum^n_{i=1}\la |\al_iu_i|,|u^*_i|\ra + \ep.\]
Let $\lambda_i = \inf\{\lambda>0: |u^*_i|\in  \lambda D^{T^*}_{\tau^*,\sigma^*}\}$.
For each $i$, there exists $(u^*_{ij})^{m_i}_{j=1}\subseteq X^*$ such that
$|u_i^*| \geq \sigma_{m_i}^*(|u^*_{i1}|,\dots, |u^*_{im_i}|)$ and  
\[ \tau^*(y_i) \geq (1-\ep)\lambda_i, \text{  where } y_i = \sum^{m_i}_{j=1}\|T^*u^*_{ij}\|e_{j+ \sum^{i-1}_{k=1}m_k}.\]
Using the fact that $\sigma^*$ is block convex with constant $A_\sigma$, we find that 
\begin{align*}
 |u^*|&\geq \sigma^*_n(|u^*_1|,\dots,|u^*_n|) \\
 &\geq \sigma^*_n(\sigma^*_{m_1}(|u^*_{11}|,\dots,|u^*_{1m_1}|),\dots, \sigma^*_{m_n}(|u^*_{n1}|,\dots,|u^*_{nm_n}|)\\
 &\geq (A_\sigma)^{-1}\sigma_r^*(|u^*_{11}|,\dots, |u^*_{1m_1}|,\dots, |u^*_{n1}|,\dots,|u^*_{nm_n}|), \quad r = \sum^n_{i=1}m_i.
 \end{align*}
Since $u^* \in D^{T^*}_{\tau^*,\sigma^*}$ and $\tau^*$ is block concave with constant $A_\tau$,
\begin{align*}1 &\geq \tau^*\intoo[3]{\sum^n_{i=1}\sum^{m_i}_{j=1}(A_\sigma)^{-1}\|T^*u_{ij}^*\|e_{j+ \sum^{i-1}_{k=1}m_k}}\\&
 = (A_\sigma)^{-1}\tau^*\intoo[3]{\sum^n_{i=1}y_i} \geq (A_\sigma A_\tau)^{-1}\tau^*\intoo[3]{\sum^n_{i=1}\tau^*(y_i)e_i}\\
 & \geq  (1-\ep)(A_\sigma A_\tau)^{-1}\tau^*\intoo[3]{\sum^n_{i=1}\lambda_ie_i}.
\end{align*}
As $|u_i|\in  (C^{\tau,\sigma}_T)^{\circ\circ}$ and $|u^*_i| \in \lambda_iD^{T^*}_{\tau^*,\sigma^*} = \lambda_i(C^{\tau,\sigma}_T)^{\circ}$ by Theorem \ref{thm: polarity sets C and D}(1), $\la |u_i|,|u_i^*|\ra \leq \lambda_i$.  
Therefore,
\begin{align*}
\la \sigma_n&(|\al_1u_1|,\dots,|\al_nu_n|),|u^*|\ra \leq\sum^n_{i=1}\la |\al_iu_i|,|u^*_i|\ra + \ep
\\ &\leq \sum^n_{i=1}|\al_i\lambda_i|+\ep
\leq \tau\intoo[3]{\sum^n_{i=1}\al_ie_i}\cdot \tau^*\intoo[3]{\sum^n_{i=1}\lambda_ie_i}+\ep\\& \leq \frac{A_\sigma A_\tau}{1-\ep}+\ep.
\end{align*}
Taking $\ep \downarrow0$, we see that $\la \sigma_n(|\al_1u_1|,\dots,|\al_nu_n|),|u^*|\ra\leq A_\sigma A_\tau$.
As this holds for all $u^*\in D^{T^*}_{\tau^*,\sigma^*}$, 
we conclude that $(A_\sigma A_\tau)^{-1}\sigma_n(|\al_1u_1|,\dots,|\al_nu_n|) \in (D^{T^*}_{\tau^*,\sigma^*})^\circ$.  Since the last set coincides with 
$(C^{\tau,\sigma}_T)^{\circ\circ}$ by Theorem \ref{thm: polarity sets C and D}(1), the proof is complete.
\end{proof}

We will use the previous lemma to construct optimal factorizations for $(\tau,\sigma)$-convex and $(\tau,\sigma)$-concave operators, extending the results of \cite[Section 2]{RT} for $p$-convex and $p$-concave operators. In particular, applying \Cref{thm: minimal factorization tau sigma convex operator} with $\tau$ and $\sigma$ being the norms on $\ell_p$ and $c_0$, respectively, we will obtain the first statement of \Cref{thm: optimal factorization upe lpe case}. If we apply instead \Cref{thm: maximal factorization tau sigma concave operator} with the norms on $\ell_q$ and $\ell_1$ acting as $\tau$ and $\sigma$, we will recover the second part of \Cref{thm: optimal factorization upe lpe case}.

\subsubsection{ Minimal convex factorization}
In this subsection, let $T:E\to X$ be a $(\tau,\sigma)$-convex operator with constant $K$ and recall the sets $C^{\tau,\sigma}_T$ and $D^{T^*}_{\tau^*,\sigma^*}$ from  Definition \ref{def: convex sets of T}.

\begin{prop}\label{prop: Coo is closed convex solid}
The set $(C^{\tau,\sigma}_T)^{\circ\circ}$ is a closed convex solid set in $X$ so that
$T(B_E) \subseteq (C^{\tau,\sigma}_T)^{\circ\circ} \subseteq KB_X$.  
\end{prop}

\begin{proof}
The set $(C^{\tau,\sigma}_T)^{\circ\circ}$ is the closed convex hull of $C^{\tau,\sigma}_T$, where the polars are taken with respect to the duality $\la X,X^*\ra$. Since $C^{\tau,\sigma}_T$ is solid, so is $(C^{\tau,\sigma}_T)^{\circ\circ}$.
Suppose that $u\in C_T^{\tau,
\sigma}$.
There exists $(x_i)^n_{i=1}$ in $E$ so that $|u| \leq \sigma_n(|Tx_1|,\dots,|Tx_n|)$ and $\tau(\sum^n_{i=1}\|x_i\|e_i) \leq 1$. By the convexity of $T$, $\|u\| \leq K$. Thus $C^{\tau,\sigma}_T \subseteq KB_X$ and hence $(C^{\tau,\sigma}_T)^{\circ \circ} \subseteq KB_X$.
Finally, if  $x\in B_E$, then $|Tx|\leq \sigma_1(|Tx|)$ and $\tau(\|x\|e_1) = \|x\| \leq 1$. Hence, $Tx\in C^{\tau,\sigma}_T\subseteq (C^{\tau,\sigma}_T)^{\circ\circ}$.
\end{proof}

By Proposition \ref{prop: class C correspondence}, there is a Class $\cC$ factorization $(Y_0,U_0,V_0)$ of $T$ so that $V_0B_{Y_0} = (C^{\tau,\sigma}_T)^{\circ\circ}$. The Banach lattice $Y_0$ induced by this factorization is $(\tau,\sigma)$-convex.

\begin{thm}\label{thm: minimal factorization tau sigma convex operator}
Assume that $\sigma$ is block concave with constant $A_\sigma$ and $\tau$ is block convex with constant $A_\tau$. Then, the Banach lattice $Y_0$ is $(\tau,\sigma)$-convex with constant $A_\sigma A_\tau$. The operators $U_0, V_0$ satisfy $\|U_0\| \leq 1$, $\|V_0\|\leq K$.
If $(Y,U,V)$ is a Class $\cC$ factorization of $T$ so that $Y$ is $(\tau,\sigma)$-convex, then there is a 
linear operator $\vp:Y_0\to Y$ of Class $\cC$ so that $U = \vp U_0$ and $V_0 = V\vp$.
\end{thm}

\begin{proof}
By definition, $Y_0 = \spn (C^{\tau,\sigma}_T)^{\circ\circ}$ and $(C^{\tau,\sigma}_T)^{\circ\circ}$ is the closed ball of $Y_0$ (Lemma \ref{lem: minkowski functional}). Since $U_0$ agrees with $T$ formally and $V_0$ is the inclusion map from $Y$ to $X$, 
by Proposition \ref{prop: Coo is closed convex solid}, $\|U_0\| \leq 1$ and $\|V_0\|\leq K$.\medskip

Let $(u_i)^n_{i=1}$ be a sequence in $B_{Y_0} = (C^{\tau,\sigma}_T)^{\circ\circ}$ and let $(\al_i)^n_{i=1}$ be a scalar sequence satisfying $\tau(\sum^n_{i=1}\al_ie_i) \leq 1$.
By Proposition \ref{prop: double polar of C is tau-sigma convex}, $\sigma_n(|\al_1u_1|,\dots, |\al_nu_n|) \in A_\sigma A_\tau\cdot B_{Y_0}$.  Hence $Y_0$ is $(\tau,\sigma)$-convex with constant $A_\sigma A_\tau$. \medskip

Now, suppose that $(Y,U,V)$ is a Class $\cC$ factorization of $T$ so that $Y$ is $(\tau,\sigma)$-convex with constant $K$. Without loss of generality, we can assume that $U$ is contractive.
We claim that $\frac{1}{K}(V(B_Y))^\circ \subseteq D^{T^*}_{\tau^*,\sigma^*}$. 
Indeed, if  $u^*\notin D^{T^*}_{\tau^*,\sigma^*}$, then there exists $(u^*_i)^n_{i=1}$ in $X^*$ so that $|u^*|\geq \sigma_n^*(|u^*_1|,\dots,|u^*_n|)$ and $\tau^*(\sum^n_{i=1}\|T^*u^*_i\|e_i) > 1$.
Choose $(x_i)^n_{i=1}\subseteq  B_E$ and a scalar sequence $(\al_i)^n_{i=1}$ so that $\tau(\sum^n_{i=1}\al_ie_i)\leq1$ and  
\begin{align*}
 1 &< \sum^n_{i=1}\al_i\la x_i,T^*u^*_i\ra  = \sum^n_{i=1}\la T(\al_ix_i),u^*_i\ra\quad \text{(Lemma \ref{lem: sum of evaluations bounded by sigma})}\\
 &
 \leq \la \sigma_n(|T(\al_1x_1)|,\dots, |T(\al_nx_n)|), \sigma^*_n(|u^*_1|,\dots,|u^*_n|)\ra \\
 &\leq \la \sigma_n(|T(\al_1x_1)|,\dots, |T(\al_nx_n)|), |u^*|\ra\\
 & =  \la V\sigma_n(|U(\al_1x_1)|,\dots, |U(\al_nx_n)|), |u^*|\ra.
\end{align*}
However, since $Y$ is $(\tau,\sigma)$-convex with constant $K$ and $U$ was assumed to be contractive, 
\[\|\sigma_n(|U(\al_1x_1)|,\dots, |U(\al_nx_n)|)\| \leq K\tau\intoo[3]{\sum^n_{i=1}\al_ie_i} \leq K.\]
Thus $V\sigma_n(|U(\al_1x_1)|,\dots, |U(\al_nx_n)|)\in KV(B_Y)$.
This shows that $|u^*| \notin (KV(B_Y))^\circ = \frac{1}{K}(V(B_Y))^\circ$, completing the proof of the claim. \medskip

From the claim, Theorem \ref{thm: polarity sets C and D}(1), and the fact that $V_0$ is the formal inclusion, we obtain that
\[ V_0B_{Y_0} = V_0((C^{\tau,\sigma}_T)^{\circ \circ})= 
(C^{\tau,\sigma}_T)^{\circ \circ}= (D^{T^*}_{\tau^*,\sigma^*})^\circ \subseteq K(V(B_Y))^{\circ\circ} = KV(B_Y),\]
where the last equality holds since $V(B_Y)$ is closed and convex.
The existence of the desired map $\vp:Y_0\to Y$ now follows from Proposition \ref{prop: class C correspondence}(2).
\end{proof}

Due to Theorem \ref{thm: minimal factorization tau sigma convex operator}, we call $(Y_0,U_0,V_0)$ the {\em minimal $(\tau,\sigma)$-convex factorization} of $T:E\to X$.

\subsubsection{Maximal concave factorization}
Similar results for concave operators can be obtained via duality.
Let $T: X\to E$ be a $(\tau,\sigma)$-concave operator with constant $K$.

\begin{prop}\label{prop: D is closed convex solid}
The set $D^T_{\tau,\sigma}$ is a closed convex solid set in $X$ 
so that $\frac{1}{K}B_X \subseteq D^T_{\tau,\sigma} \subseteq T^{-1}(B_E)$.
\end{prop}

\begin{proof}
By Theorem \ref{thm: adjoint of concave is convex}, $T^*:E^*\to X^*$  is a $(\tau^*,\sigma^*)$-convex operator.  By Proposition \ref{prop: Coo is closed convex solid}, 
\[T^*(B_{E^*}) \subseteq (C^{\tau^*,\sigma^*}_{T^*})^{\circ\circ} \subseteq KB_{X^*}.
\]
Easy calculations show that $(T^*(B_{E^*}))^\circ = T^{-1}(B_E)$
 and $(KB_{X^*})^\circ = K^{-1}B_X$.  Thus, taking polars in the string of inclusions above gives
 \[ K^{-1}B_X \subseteq (C^{\tau^*,\sigma^*}_{T^*})^{\circ} \subseteq T^{-1}(B_E).
 \]
By Theorem \ref{thm: polarity sets C and D}(2),     $(C^{\tau^*,\sigma^*}_{T^*})^{\circ} = D^T_{\tau,\sigma}$. This concludes the proof of the proposition.
\end{proof}

By Proposition \ref{prop: class D correspondence}, there is a Class $\cD$ factorization $(Y^0,U^0,V^0)$ of $T$ such that $(U^0)^{-1}(B_{Y^0}) = D^T_{\tau,\sigma}$.
The next result follows from Theorem \ref{thm: minimal factorization tau sigma convex operator} by duality.  We omit the details.

\begin{thm}\label{thm: maximal factorization tau sigma concave operator}
Assume that $\sigma$ is block convex with constant $A_\sigma$ and $\tau$ is block concave with constant $A_\tau$. Then, the Banach lattice $Y^0$ is $(\tau,\sigma$)-concave with constant $A_\sigma A_\tau$. The operators $U^0$, $V^0$ satisfy $\|U^0\| \leq K$, $\|V^0\|\leq 1$. 
If $(Y,U,V)$ is a Class $\cD$ factorization of $T$ so that $Y$ is $(\tau,\sigma)$-concave, then there is a linear operator $\vp:Y\to Y^0$ of Class $\cD$ so that $U^0 = \vp U$ and $V^0\vp = V$.
\end{thm}

We call $(Y^0,U^0,V^0)$ the {\em maximal $(\tau,\sigma)$-concave factorization} of $T:X\to E$.\medskip

As a consequence, we have the following generalized version of a factorization result given in \cite{Reisner}.

\begin{thm}\label{thm: double factorization Reisner}
Let $E,F$ be Banach spaces. Let $\tau,\sigma, \tau',\sigma'$ be normalized symmetric norms on $c_{00}$ such that $\sigma$ and $\tau'$ are block concave with constants $A_\sigma$ and $A'_\tau$, and $\tau$ and $\sigma'$ are block convex with constants $A_\tau$ and $A'_\sigma$, respectively. Assume that a bounded linear operator $T:E\to F$  has a factorization $T = VU$, where $X$ is a Banach lattice, $U:E\to X$ is $(\tau,\sigma)$-convex with constant $K$ and $V:X\to F$ is $(\tau',\sigma')$-concave with constant $K'$.
Then there are a $(\tau,\sigma)$-convex Banach lattice $Y$ with constant $A_\sigma A_\tau$ and a $(\tau',\sigma')$-concave Banach lattice $Z$ with constant $A_\sigma' A_\tau'$ and bounded linear operators $A:E\to Y$, $B:Y\to Z$, $C:Z\to F$ so that $T = CBA$, $\|A\| \leq K$, $\|C\|\leq K'$ and $B$ is a contraction of Class $\cD$.
\end{thm}

\begin{proof}
We may replace $X$ by the closed ideal $I$ generated by $U(E)$ in $X$ and $V$ by its restriction $V|_I$. In this way, we may assume without loss of generality that the ideal generated by $U(E)$ in $X$ is dense.
Apply Theorem \ref{thm: minimal factorization tau sigma convex operator} to $U$ and Theorem \ref{thm: maximal factorization tau sigma concave operator} to $V$ to find the following:
\begin{enumerate}
\item  A $(\tau, \sigma)$-convex Banach lattice $Y$ with constant $A_\sigma A_\tau$ and bounded linear operators $A:E\to Y$, $R: Y\to X$ so that $R$ is of Class $\cC$ and $U = RA$. Furthermore, $\|A\| \leq K$ and $\|R\|\leq 1$.
\item  A $(\tau', \sigma')$-concave Banach lattice $Z$ with constant $A_\sigma' A_\tau'$ and bounded linear operators ${S}:X\to Z$, ${C}: Z\to F$ so that ${S}$ is of Class $\cD$, $V = {C}{S}$,  $\|{S}\| \leq K'$ and $\|{C}\|\leq 1$.
We may replace $S$ and $C$ by $S/K'$ and $K'C$, respectively. Thus, we may instead assume that $\|S\|\leq 1$ and $\|C\|\leq K'$.
\end{enumerate}
Since $U = RA$, $U(E)\subseteq R(Y)$. Also, since $R$ is of Class $\cC$, $R(Y)$ is an ideal in $X$. By the comment at the beginning of the proof, $\ol{R(Y)} = X$. Let $B = SR: Y\to Z$. 
Then $T = CBA$, $\|A\| \leq K$, $\|C\|\leq K'$ and $\|B\|\leq 1$.
Since both $R$ and $S$ are lattice homomorphisms with dense range, $B$ is a lattice homomorphism with dense range. Thus, $B$ is of Class $\cD$.
\end{proof}

\subsection{Duality of factorization}\label{sec: duality factorizations}

Our next goal is to generalize \cite[Theorem 5]{RT}. Let us start by studying the relations between the Banach lattices induced by a closed bounded convex solid set $B$ and its polar $B^\circ$, and their corresponding duals.

\begin{prop}\label{prop: pseudo duality of the constructions}
Let $B$ be a closed bounded convex solid subset of $X$.
Denote the Minkowski functional of $B$ on $X$ by $\rho$ and the Minkowski functional of $B^\circ$ on $X^*$ by $\rho^*$.
Let $Y= (\spn B, \rho)$.  Let $Q: X^* \mapsto X^*/\ker \rho^*$ be the quotient map, and let $Z$ be the completion of  $(X^*/\ker \rho^*, \ol{\rho^*})$, where $\ol{\rho^*}(Qu^*) = \rho^*(u^*)$.
Then $Y$ and $Z$ are Banach lattices, and there are unique bounded linear maps $\vp$ and $\psi$ determined as follows:
\begin{align*}
\vp: Y \to Z^*, \quad & \la Qu^*, \vp(u)\ra = u^*(u),\\
\psi: Z\to Y^*, \quad  & \la u, \psi(Qu^*) \ra = u^*(u).
\end{align*}
Moreover, $\vp$ and $\psi$ are lattice isometric embeddings, the map $\psi$ is of Class $\cC$, and, if $X$ has order continuous norm, then  the map $\vp$ is of Class $\cC$.  
\end{prop}

\begin{proof}
The fact that $(Y,\rho)$ is a Banach lattice follows from Lemma \ref{lem: minkowski functional}. It is clear that $Z$ is a Banach lattice.
If $u\in Y$ and $u^*\in X^*$, then 
$|u^*(u)| \leq \rho(u)\cdot \rho^*(u^*) = \rho(u)\cdot \ol{\rho^*}(Qu^*)$.
It follows that $\vp(u)$ determines a bounded linear functional on $(X^*/\ker \rho^*, \ol{\rho^*})$ and hence $\vp(u) \in Z^*$. Similarly, we have that $\psi$ is a bounded linear map on $(X^*/\ker \rho^*, \ol{\rho^*})$ and hence extends uniquely to a bounded linear map on $Z$.
Furthermore, 
\[ \|\vp(u)\|_{Z^*} = \sup_{\ol{\rho^*}(Qu^*)\leq 1}\la Qu^*,\vp(u)\ra = \sup_{\rho^*(u^*)\leq 1}u^*(u) = \rho(u).
\]
Thus, $\vp$ is an isometric embedding. Similarly, it can be seen that $\psi$ is an isometric embedding.  
\medskip
If $u\in Y_+$ and $u^*\in X^*$, then
\[
\la u, |\psi(Qu^*)|\ra  = \sup_{|v|\leq u}|\la v, \psi(Qu^*) \ra| = \sup_{|v| \leq u}|u^*(v)| = |u^*|(u) = \la u,\psi(|Qu^*|)\ra.\]
Hence, $|\psi(Qu^*)| = \psi(|Qu^*|)$. Since $Q(X^*)$ is dense in $Z$, $\psi$ is a lattice homomorphism. Similarly, $\vp$ is a lattice homomorphism.
\medskip

Next, we claim that for any $u^*\in X^*_+$, $\psi[0,Qu^*] = [0,\psi Qu^*]$.   
Indeed, the inclusion $\subseteq$ is clear.
Suppose that $x^* \in [0,\psi Qu^*]$.
For any $u\in Y$,
\[ |\la u,x^*\ra|\leq \la |u|, x^*\ra  \leq \la |u|,\psi Qu^*\ra = u^*(|u|).
\]
By the Hahn-Banach Theorem for positive functionals (cf.~\cite[Theorem 1.26]{AB}), there exists a functional $v^*$ on $X$ such that $v^*|_Y = x^*$ and $0\leq v^*(u)\leq u^*(u)$ for all $u\in X_+$. Thus, $Qv^*\in[0, Qu^*]$ and for any $u\in Y$,
\[ \la u,\psi Qv^*\ra = v^*(u)= \la u, x^*\ra.\]
Hence, $\psi Qv^* = x^*$. 
This completes the proof of the claim.
\medskip

We extend the claim to show that $\psi$ is interval preserving and thus $\psi$ is of Class $\cC$.
Suppose that $z\in Z_+$ and $0 \leq x^*\leq \psi z$.
For each $n\in \N$, choose $v^*_n\in X^*_+$ so that  $\| z- Qv^*_n\|_Z < \frac{1}{n}$.
Set 
$ y^*_n = (x^* + \psi Qv^*_n - \psi z)^+$.
Then $0\leq y^*_n \leq \psi Q v^*_n$ and
\[ \|y^*_n-x^*\| \leq \|\psi Qv^*_n- \psi z\|  = \|Qv^*_n - z\|_Z < \frac{1}{n}.\]
By the claim, $y^*_n \in [0,\psi Qv^*_n] = \psi[0,Qv^*_n]$.
Hence, there exists $w^*_n\in X^*$, $Qw^*_n \in [0,Qv^*_n]$, so that $y^*_n = \psi Qw^*_n$.
Since $(y^*_n)$ converges to $x^*$ and $\psi$ is an isometry, $(Qw^*_n)$ converges in $Z$ to some $z_0\in Z$.  Clearly, $0\leq z_0 \leq \lim Qv^*_n = z$ and 
\[\psi z_0 = \lim \psi Qw_n^* = \lim y^*_n = x^*.\]
This shows that $[0,\psi z] \subseteq \psi[0,z]$.  The reverse inclusion is clear.
\medskip

Finally, we show that $\vp$ is interval preserving if $X$ has order continuous norm.
Let $J_X: X \to X^{**}$ be the canonical inclusion.
Then $J_X[-u,u] = [-J_Xu,J_Xu]$ for all $u\in X_+$.
Suppose that $u \in Y_+$.  Then $\vp[0,u] \subseteq [0,\vp u]$.
Conversely, let $z^* \in Z^*$, $z^*\in [0,\vp u]$.
Since $Q:X^*\to (X^*/\ker\rho^*, \ol{\rho^*})\subseteq Z$ is bounded, $Q^*z^* \in X^{**}$.
For any $u^* \in X^*$,
\[ |\la u^*,Q^*z^*\ra| \leq  \la |Qu^*|,z^*\ra \leq \la Q|u^*|, \vp u\ra = \la u, |u^*|\ra = \la |u^*|,J_Xu\ra.
\] 
Hence, $|Q^*z^*| \leq J_Xu$.
Thus $Q^*z^* \in [-J_Xu,J_Xu] = J_X[-u,u]$.
Let $v \in X$, $v\in [-u,u]$ be such that $J_Xv = Q^*z^*$.  Then $v\in Y$ and 
\[ \la Qu^*,\vp v\ra = u^*(v) = \la u^*,Q^*z^*\ra = \la Qu^*,z^*\ra
\]
for all $u^*\in X^*$.
Therefore, $\vp v= z^*$. This completes the proof that $\vp[0,u] = [0,\vp u]$.
\end{proof}

Now, let $T:E\to X$ be a $(\tau,\sigma)$-convex operator, and assume that $\sigma$ is block concave and $\tau$ is block convex. By Theorem \ref{thm: minimal factorization tau sigma convex operator}, there is a minimal Class $\cC$ factorization $(Y_0,U_0,V_0)$ of $T$ so that $Y_0 = \spn (C^{\tau,\sigma}_T)^{\circ\circ}$ with norm given by the Minkowski functional of $(C^{\tau,\sigma}_T)^{\circ\circ}$.
By Theorem \ref{thm: adjoint of convex is concave}, $T^*:X^*\to E^*$ is $(\tau^*,\sigma^*)$-concave, $\sigma^*$ is block convex and $\tau^*$ is block concave.
By Theorem \ref{thm: maximal factorization tau sigma concave operator}, there is a maximal Class $\cD$ factorization $(Y^0,U^0,V^0)$ of $T^*$. Let $\rho$ be the Minkowski functional of $D^{T^*}_{\tau^*,\sigma^*}$ and let $\ol{\rho}$ be the norm on $X^*/\ker \rho$ induced by $D^{T^*}_{\tau^*,\sigma^*}$. Then
$Y^0$ is the completion of $(X^*/\ker\rho,\ol{\rho}).$ 
From  Theorem \ref{thm: polarity sets C and D}(1), we have $(C^{\tau,\sigma}_T)^\circ = D^{T^*}_{\tau^*,\sigma^*}$.
As a result, using Proposition \ref{prop: pseudo duality of the constructions}, we obtain  partial duality of the Banach lattices $Y_0$ and $Y^0$.
Note that below we use the fact that $Y_0$ is a subset of $X$ by construction.

\begin{thm}\label{thm: pseudo duality for factorization of convex operator}
Let $T:E\to X$ be a $(\tau,\sigma)$-convex operator and assume that $\sigma$ is block concave and $\tau$ is block convex.
Let $(Y_0,U_0,V_0)$ be the Class $\cC$ factorization of $T$ induced by the set $(C^{\tau,\sigma}_T)^{\circ\circ}$ and let $(Y^0,U^0,V^0)$ be the Class $\cD$ factorization of $T^*$ induced by the set $D^{T*}_{\tau^*,\sigma^*}$. Let the maps $\vp$ and $\psi$ be determined by 
\begin{align*}
\vp: Y_0 \to (Y^0)^*, \quad & \la U^0u^*, \vp(u)\ra = u^*(u),\\
\psi: Y^0\to (Y_0)^*, \quad  & \la u, \psi(U^0u^*) \ra = u^*(u).
\end{align*}
Then $\vp$ and $\psi$ are lattice isometric embeddings.  The map $\psi$ is of Class $\cC$. If $X$ has order continuous norm, then  the map $\vp$ is of Class $\cC$.  
\end{thm}

\begin{proof}
The proof is essentially immediate. Let us just check that $\vp$ is an into isometry.
Denote the norm on $(Y^0)^*$ by $\ol{\rho}^*$. Note that $U^0$ is the quotient map $X^*\to X^*/\ker \rho$.
Thus,
\[ \ol{\rho}^*(\vp(u)) = \sup_{\ol{\rho}(U^0u^*)\leq 1} \la U^0u^*,\vp(u)\ra
= \sup_{\rho(u^*)\leq 1}u^*(u) = \sup_{u^*\in D^{T^*}_{\tau^*,\sigma^*}}u^*(u) = \|u\|_{Y_0},
\]
where the last equality holds because  $\|\cdot\|_{Y_0}$ is the Minkowski functional of $(C^{\tau,\sigma}_{T})^{\circ\circ} = (D^{T^*}_{\tau^*,\sigma^*})^\circ$.
\end{proof}

A similar result holds for concave operators and their adjoints.
Let $T:X\to E$ be a $(\tau,\sigma)$-concave operator, where $\sigma$ is block convex and $\tau$ is block concave.
Let $\theta$ be the Minkowski functional of the set $D^T_{\tau,\sigma}$ and let $\ol{\theta}$ be the induced norm on $X/\ker \theta$.
Let $Z^0$ be the completion of $(X/\ker\theta, \ol{\theta})$.
By Theorem \ref{thm: maximal factorization tau sigma concave operator}, there is a maximal Class $\cD$ factorization $(Z^0,A^0,B^0)$ of $T$.
Also, there is a minimal Class $\cC$ factorization $(Z_0,A_0,B_0)$ of $T^*$, where $Z_0 = \spn (C^{\tau^*,\sigma^*}_{T^*})^{\circ\circ}$, normed by the Minkowski functional of $\spn (C^{\tau^*,\sigma^*}_{T^*})^{\circ\circ}$.
As a result of the duality $(C^{\tau^*,\sigma^*}_{T^*})^\circ  =D^T_{\tau,\sigma}$ (Theorem \ref{thm: polarity sets C and D}(2)), we obtain the following duality theorem.  The proof is similar to that of Theorem \ref{thm: pseudo duality for factorization of convex operator} and is omitted.

\begin{thm}\label{thm: pseudo duality for factorization of concave operator}
Let $T:X\to E$ be a $(\tau,\sigma)$-concave operator and assume that $\sigma$ is block convex and $\tau$ is block concave.
Let $(Z^0,A^0,B^0)$ be the Class $\cD$ factorization of $T$ induced by the set $D^T_{\tau,\sigma}$ and let $(Z_0,A_0,B_0)$ be the Class $\cC$ factorization of $T^*$ induced by the set $(C^{T^*}_{\tau^*,\sigma^*})^{\circ\circ}$.
Let the maps $\ti{\vp}$ and $\ti{\psi}$ be determined by 
\begin{align*}
\ti{\vp}: Z^0 \to (Z_0)^*, \quad & \la u^*, \ti{\vp}(A^0u)\ra = u^*(u),\\
\ti{\psi}: Z_0\to (Z^0)^*, \quad  & \la A^0u, \ti{\psi}(u^*) \ra = u^*(u).
\end{align*}
Then $\ti{\vp}$ and $\ti{\psi}$ are lattice isometric embeddings.  The map $\ti{\psi}$ is of Class $\cC$. If $X$ has order continuous norm, then  the map $\ti{\vp}$ is of Class $\cC$.  
\end{thm}

\subsection{Factorization of a simultaneously convex and concave operator}\label{sec: factorization of simultaneous conditions}

Let $X$ and $Y$ be Banach lattices. In \cite{RT}, a factorization for an operator $T:X\to Y$ that is both $p$-convex and $q$-concave was given. In this section, we adapt an argument of \cite{Reisner} to prove an analogous factorization result that allows us to replace the $p$-convexity or $q$-concavity assumptions by $(\pinf)$-convexity and $(q,1)$-concavity, respectively.  We start with the following:

\begin{prop}\label{prop: Calderon interpolation of sets}
Let $Z$ be a Banach lattice and let $C_0,C_1$ be bounded convex sets in $Z$.
For any $0< \theta <1$, let
\begin{equation}\label{eq: Calderon interpolation of sets} 
C_\theta = \{v\in Z: |x| \leq |v_0|^\theta |v_1|^{1-\theta} \text{ for some } v_0\in C_0, v_1\in C_1\}.
\end{equation}
Then ${C_\theta}$ is a  
 bounded convex solid set in $Z$. 
\end{prop}

\begin{proof}
Suppose that  $a_0,a_1,b_0,b_1$ are nonnegative real numbers.  Let $0 < \al < 1$.  Then
\begin{align*}
(1-\al)a_0^\theta\, a_1^{1-\theta} &+ \al b_0^\theta\, b_1^{1-\theta} \leq \|((1-\al)^\theta a_0^\theta, \al^\theta b_0^{\theta})\|_{\frac{1}{\theta}}\cdot \|((1-\al)^{1-\theta}a_1^{1-\theta}, \al^{1-\theta}b_1^{1-\theta})\|_{\frac{1}{1-\theta}}\\& = ((1-\al)a_0+ \al b_0)^{\theta}\cdot ((1-\al)a_1+\al b_1)^{1-\theta}.
\end{align*}
Now, for $v,w\in C_\theta$, there are $v_0,w_0\in C_0$ and $v_1,w_1\in C_1$ so that  $|v| \leq |v_0|^\theta\,|v_1|^{1-\theta}$ and $|w| \leq |w_0|^\theta\,|w_1|^{1-\theta}$.  Hence,
\begin{align*}
|(1-\al)v+\al w| &\leq (1-\al)|v| + \al |w| \leq (1-\al)|v_0|^\theta\,|v_1|^{1-\theta} + \al|w_0|^\theta\,|w_1|^{1-\theta}\\&
\leq  ({(1-\al)|v_0|+ \al |w_0|})^{\theta}\cdot ((1-\al)|v_1|+\al |w_1|)^{1-\theta}.
\end{align*}
Therefore, $(1-\al)v+\al w\in C_\theta$.  It follows that ${C_\theta}$ is convex.  Clearly, $C_\theta$ is also solid.
By \cite[Proposition 1.d.2]{LT2}, if $|v| \leq |v_0|^\theta\,|v_1|^{1-\theta}$, then
$\|v\| \leq \|v_0\|^\theta\, \|v_1\|^{1-\theta}.$ Thus, the boundedness of $C_\theta$ follows from that of $C_0$ and $C_1$.  
\end{proof}

\subsubsection{Factoring a \texorpdfstring{$(p,p_2)$}{}-convex and \texorpdfstring{$(q,1)$}{}-concave operator}
In this subsection, let $X,Y$ be Banach lattices, $1\leq p,q\leq \infty$ and $p_2 \in \{p,\infty\}$.  Suppose that $T:X\to Y$ is a bounded linear operator that is $(p,p_2)$-convex with constant $K_p$ and $(q,1)$-concave with constant $K_q$.
As in Definition \ref{def: convex sets of T}, let 
\[ C_T^{p,p_2} = \cbr[3]{v\in Y: \exists (u_i)^n_{i=1}\subseteq X \text{ such that } |v|\leq \intoo[3]{\sum^n_{i=1}|Tu_i|^{p_2}}^{\frac{1}{p_2}}, \sum^n_{i=1}\|u_i\|^p\leq 1},\]
which corresponds to the set $C_T^{\tau,\sigma}$ taking the norms $\sigma=\|\cdot\|_{p_2}$ and $\tau=\|\cdot\|_{p}$ on $c_{00}$, both of them block convex and block concave with constant 1.
By Proposition \ref{prop: Coo is closed convex solid}, $C_0 := (C^{p,p_2}_T)^{\circ\circ}$ is a closed  convex solid set in $Y$ so that $T(B_X) \subseteq C_0 \subseteq K_pB_Y$.
By Proposition \ref{prop: double polar of C is tau-sigma convex}, if $\rho_0$ is the Minkowski functional of $C_0$, then
\[ 
\rho_0\intoo[3]{\intoo[3]{\sum^n_{i=1}|u_i|^{p_2}}^{\frac{1}{p_2}}} \leq \intoo[3]{\sum^n_{i=1}\rho_0(u_i)^p}^{\frac{1}{p}},
\]
for $(u_i)^n_{i=1} \subseteq \spn C_0$.

\begin{prop}\label{prop: interpolation pp2 convex operator}
Let $C_0 = (C^{p,p_2}_T)^{\circ\circ}$ and $C_1 = B_Y$.
For any $0<\theta <1$, let $p_\theta$ be determined by the equation $\frac{1}{p_\theta} = \frac{\theta}{p} + (1-\theta)$ and set $\ol{p_2} = p_\theta$ if $p_2 = p$, and $\ol{p_2} = \infty$ if $p_2 = \infty$. Define $C_\theta$ by \eqref{eq: Calderon interpolation of sets} and let $\rho_\theta$ be the Minkowski functional of $C_\theta$ on $Z_\theta: =\spn {C_\theta}$.
Then $T(B_X) \subseteq K_p^{1-\theta}C_\theta$ and $(Z_\theta,\rho_\theta)$ is $(p_\theta,\ol{p_2})$-convex with constant $1$.
\end{prop}

\begin{proof}
First, observe that if $K_p=0$, then $T$ is the null operator, $Z_\theta=\{0\}$, and the result becomes trivial. Therefore, we can assume that $K_p>0$. Now, suppose that $u\in B_X$. Then $Tu \in C_0 \subseteq K_pC_1$, and hence 
\[ \frac{|Tu|}{ K_p^{1-\theta}} \leq  |Tu|^\theta\, \intoo[3]{\frac{|Tu|}{K_p}}^{1-\theta}\in C_\theta. \]
Thus, $T(B_X) \subseteq K_p^{1-\theta}C_\theta$.
\medskip

Now, suppose that $v_i \in \al_iC_\theta$, $1\leq i\leq n$.
There are $v_0^i\in C_0, v_1^i \in C_1$ so that $|v_i| \leq \al_i|v_0^i|^\theta\,|v_1^i|^{1-\theta}$.
First, consider the case where $p_2 = p$. 
Note that $\frac{p}{\theta p_\theta}$ and $\frac{1}{(1-\theta)p_\theta}$ 
are conjugate indices.  Thus, using H\"older's inequality in the second step below, we have
\begin{align}
 \intoo[3]{\sum^n_{i=1}|v_i|^{p_\theta}}^{\frac{1}{p_\theta}} & \leq \intoo[3]{\sum^n_{i=1}(\al_i^{\frac{p_\theta}{p}}|v^i_0|)^{\theta p_\theta} (\al_i^{p_\theta}|v^i_1|)^{(1-\theta)p_\theta}}^{\frac{1}{p_\theta}}\\ \notag
 & \leq \intoo[3]{\intoo[3]{\sum^n_{i=1}(\al_i^{\frac{p_\theta}{p}}|v^i_0|)^{p}}^{\frac{1}{p}}}^\theta \intoo[3]{\sum^n_{i=1}\al_i^{p_\theta}|v^i_1|}^{(1-\theta)}.\label{e5.2}
 \end{align}
 Hence,
\begin{align*} 
\rho_\theta\intoo[3]{\intoo[3]{\sum^n_{i=1}|v_i|^{p_\theta}}^{\frac{1}{p_\theta}}} &\leq \rho_0\intoo[3]{\intoo[3]{\sum^n_{i=1}(\al_i^{\frac{p_\theta}{p}}|v^i_0|)^{p}}^{\frac{1}{p}}}^\theta \norm[3]{\sum^n_{i=1}\al_i^{p_\theta}|v^i_1|}_Y^{1-\theta}
\\ &\leq \intoo[3]{\sum^n_{i=1}\al_i^{p_\theta}}^{\frac{\theta}{p}} \intoo[3]{\sum^n_{i=1}\al_i^{p_\theta}}^{1-\theta}
= \intoo[3]{\sum^n_{i=1}\al_i^{p_\theta}}^{\frac{1}{p_\theta}}.
\end{align*}
For the case $p_2 = \infty$, we use the inequality 
\begin{equation}\label{e5.3} 
\bigvee^n_{i=1}|v_i| \leq \intoo[3]{\bigvee^n_{i=1}(\al_i^{\frac{p_\theta}{p}}|v^i_0|)}^\theta \intoo[3]{\bigvee^n_{i=1}(\al_i^{p_\theta}|v^i_1|)}^{1-\theta}
\end{equation}
 to obtain
\begin{align*} 
\rho_\theta\intoo[3]{\bigvee^n_{i=1}|v_i|} &\leq \rho_0\intoo[3]{\bigvee^n_{i=1}(\al_i^{\frac{p_\theta}{p}}|v^i_0|)}^\theta \norm[3]{\sum^n_{i=1}\al_i^{p_\theta}|v^i_1|}_Y^{1-\theta}
\\ &\leq \intoo[3]{\sum^n_{i=1}\al_i^{p_\theta}}^{\frac{\theta}{p}} \intoo[3]{\sum^n_{i=1}\al_i^{p_\theta}}^{1-\theta}= \intoo[3]{\sum^n_{i=1}\al_i^{p_\theta}}^{\frac{1}{p_\theta}}.
\end{align*}
In both cases, we may take $\al_i$ to be arbitrarily close to $\rho_\theta(v_i)$ to complete the proof of the proposition.
\end{proof}

In \Cref{prop: interpolation pp2 convex operator} we have exploited the convexity of the operator $T$ to build a $(p_\theta,\ol{p_2})$-convex normed vector lattice $Z_\theta$. Next, we exploit the $(q,1)$-concavity of $T$:

\begin{prop}\label{prop: qtheta1 convavity of the factor}
For $0<\theta < 1$, set $q_\theta = \frac{q}{1-\theta}$. 
Let  $S_\theta: X\to Z_\theta$ be the operator defined by $S_\theta u = Tu$.
Then $S_\theta$ is $(q_\theta,1)$-concave with constant $K_q^{1-\theta}$.
\end{prop}

\begin{proof}
First, note that $S_\theta$ maps into $Z_\theta$ by Proposition \ref{prop: interpolation pp2 convex operator}.
Let $(u_i)^n_{i=1}\subseteq X$. Write $\rho_1$ for the original norm on $Y$.
Note that 
\[
|S_\theta u_i| = |Tu_i| = \rho_0(Tu_i)^\theta \rho_1(Tu_i)^{1-\theta} \intoo[3]{\frac{|Tu_i|}{\rho_0(Tu_i)}}^\theta\intoo[3]{\frac{|Tu_i|}{\rho_1(Tu_i)}}^{1-\theta},
\]
which implies that $\rho_\theta(S_\theta u_i)\leq \rho_0(Tu_i)^\theta \rho_1(Tu_i)^{1-\theta}$.
Thus
\begin{align*}
\intoo[3]{\sum^n_{i=1}\rho_\theta(S_\theta u_i)^{q_\theta}}^{\frac{1}{q_\theta}} & \leq \intoo[3]{\sum^n_{i=1}\rho_0(Tu_i)^{\theta q_\theta} \rho_1(Tu_i)^{(1-\theta)q_\theta}}^{\frac{1}{q_\theta}} \\
&\leq \sup_i \,\rho_0(Tu_i)^{\theta}\intoo[3]{\sum^n_{i=1}\rho_1(Tu_i)^{q}}^{\frac{1}{q_\theta}}.
\end{align*}
Let $u = \sum^n_{i=1}|u_i|$.  Since $T$ is $(q,1)$-concave with constant $K_q$,
\[ \|u\|= \norm[3]{\sum^n_{i=1}|u_i|} \geq K_q^{-1}\intoo[3]{\sum^n_{i=1}\rho_1(Tu_i)^{q}}^{\frac{1}{q}}.\]
Since $T(B_X) \subseteq C_0$, $\rho_0(Tu_i) \leq \|u_i\|\leq \|u\|$.  Hence, 
\[ \intoo[3]{\sum^n_{i=1}\rho_\theta(S_\theta u_i)^{q_\theta}}^{\frac{1}{q_\theta}}  \leq\|u\|^\theta (K_q\|u\|)^{\frac{q}{q_\theta}}=K_q^{1-\theta} \|u\|.
\]
Therefore, $S_\theta$ is $(q_\theta,1)$-concave with constant $K_q^{1-\theta}$.
\end{proof}

The previous results yield an analogue of \cite[Theorem 15]{RT} for operators that are simultaneously $(p,p_2)$-convex and $(q,1)$-concave:

\begin{thm}\label{thm: factorization of pp2 convex and q1 concave operator}
Let $X,Y$ be Banach lattices and let $T:X\to Y$ be a bounded linear operator that is both $(p,p_2)$-convex ($p_2 \in \{p,\infty\}$) and $(q,1)$-concave. Let $0<\theta < 1$.  Determine $p_\theta, q_\theta$ by the equations $\frac{1}{p_\theta} = \frac{\theta}{p} + (1-\theta)$ and $q_\theta = \frac{q}{1-\theta}$. Set $\ol{p_2} = p_\theta$ if $p_2 = p$, and $\ol{p_2} = \infty$ if $p_2 = \infty$. 
There exists 
\begin{enumerate}
\item a $(p_\theta,\ol{p_2})$-convex Banach lattice $\mathcal{Z}_\theta$ and  a $(q_\theta,1)$-concave Banach lattice $W_\theta$; 
\item an interval preserving injective lattice homomorphism $\Phi: \mathcal{Z}_\theta \to Y$ and a Class $\cD$ operator $\Psi: X\to W_\theta$;
\item a bounded linear operator $T_\theta: W_\theta\to \mathcal{Z}_\theta$, 
\end{enumerate}
so that $T = \Phi T_\theta\, \Psi$.
\end{thm}

\begin{proof}
Let $(Z_\theta,\rho_\theta)$ be as in Proposition \ref{prop: interpolation pp2 convex operator} and let $S_\theta:X\to Z_\theta$ be defined by $S_\theta u = Tu$.
By Proposition \ref{prop: interpolation pp2 convex operator}, $(Z_\theta, \rho_\theta)$ is a $(p_\theta,\ol{p_2})$-convex normed vector lattice.
Hence, its norm completion $\mathcal{Z}_\theta$ is a $(p_\theta,\ol{p_2})$-convex Banach lattice.
Since $C_\theta$ is a bounded convex solid subset of $Y$, its Minkowski functional $\rho_\theta$ dominates the norm of $Y$.  Thus, every $\rho_\theta$-Cauchy sequence in $Z_\theta$ converges in $Y$.
Hence $\mathcal{Z}_\theta$ may be identified (as a vector lattice) with 
\[ \{v\in Y: \exists \text{ a $\rho_\theta$-Cauchy sequence } (v_n)^\infty_{n=1} \text{ in } Z_\theta \text{ so that } (v_n)^\infty_{n=1} \text{ converges to } v \text { in } Y\}.\]
In this way, $\mathcal{Z}_\theta$ is (identified with) a vector lattice ideal in $Y$.
Hence, the inclusion map $\Phi : \mathcal{Z}_\theta\to Y$ is an interval preserving injective lattice homomorphism.
Note that $\Phi S_\theta = T$ by the definition of $\Phi$ and $S_\theta$.
\medskip

By Proposition \ref{prop: qtheta1 convavity of the factor}, $S_\theta: X\to Z_\theta \subseteq \mathcal{Z_\theta}$ is $(q_\theta,1)$-concave. 
Hence, by \Cref{thm: optimal factorization upe lpe case}(2), there is a $(q_\theta,1)$-concave Banach lattice $W_\theta$, a Class $\cD$ operator $\Psi :X\to W_\theta$, and a bounded linear operator $T_\theta: W_\theta\to \mathcal{Z}_\theta$ so that $S_\theta = T_\theta\, \Psi$.
Therefore, $T = \Phi S_\theta = \Phi T_\theta\, \Psi$, as claimed.
\end{proof}

\subsubsection{Factoring a \texorpdfstring{$(p,\infty)$}{}-convex and \texorpdfstring{$(q,q_2)$}{}-concave operator}

In this subsection, let $X,Y$ be Banach lattices, $1\leq p,q\leq \infty$ and $q_2 \in \{q,1\}$. Suppose that  $T:X\to Y$ is a bounded linear operator that is $(p,\infty)$-convex with constant $K_p$ and $(q,q_2)$-concave with constant $K_q$. 
As in Definition \ref{def: convex sets of T}, let 
\[ C_{T^*}^{q^*,q_2^*} = \cbr[3]{u^*\in X^*: \exists (v_i^*)^n_{i=1} \subseteq Y^* \text{ such that } |u^*|\leq \intoo[3]{\sum^n_{i=1}|T^*v_i^*|^{q^*_2}}^{\frac{1}{q^*_2}}, \sum^n_{i=1}\|v_i\|^{q^*}\leq 1},\]
which corresponds to the set $C_{T^*}^{\tau,\sigma}$ taking the norms $\sigma=\|\cdot\|_{q_2^*}$ and $\tau=\|\cdot\|_{q^*}$ on $c_{00}$, both of them block convex and block concave with constant 1.
Since $T$ is $(q,q_2)$-concave with constant $K_q$, $T^*$ is $(q^*,q_2^*)$-convex with constant $K_q$ by Theorem \ref{thm: adjoint of concave is convex}.
By Proposition \ref{prop: Coo is closed convex solid}, $C_0:= (C_{T^*}^{q^*,q_2^*})^{\circ \circ}$ is a closed convex solid set in $X^*$ so that $T^*B_{Y^*}\subseteq C_0 \subseteq K_qB_{X^*}$. 
Fix $0<\theta<1$ and let 
\[ C_\theta = \{u^*\in X^*: |u^*| \leq |u^*_0|^\theta|u^*_1|^{1-\theta}, u^*_0 \in C_0, u^*_1 \in B_{X^*}\}.\]
Let $Z_\theta = \spn C_\theta$ and let $\rho_\theta$ be the Minkowski functional of $C_\theta$.
Then $(Z_\theta,\rho_\theta)$ is a normed vector lattice.
By Proposition \ref{prop: interpolation pp2 convex operator}, $T^*B_{Y^*} \subseteq K_q^{1-\theta}C_\theta$ and $(Z_\theta,\rho_\theta)$ is $(q_\theta^*, (\ol{q_2})^*)$-convex with constant $1$, where 
\[ \frac{1}{q^*_\theta} = \frac{\theta}{q^*} + (1-\theta) \quad\text{and}\quad 
\ol{q_2} = \begin{cases}
q_\theta &\text{if $q_2 = q$},\\
1& \text{if $q_2 =1$.}
\end{cases}\]
Note that the first condition is equivalent to taking $q_\theta = \frac{q}{\theta}$.
Let $D\subseteq X$ be the set $(C_\theta)^\circ$, where the polar is taken with respect to the duality $\la X,X^*\ra$.

\begin{prop}\label{prop: pseudo dual of the interpolation of q1 concave operator}
The set $D$ is a closed convex solid subset of $X$ that contains $K_p^{-\theta}B_X$.
Let $\pi$ be the Minkowski functional of $D$ on $X$ and let $Q:X \to X/\ker\pi$ be the quotient map. Define $\wh{\pi}$ on $X/\ker\pi$ by $\wh{\pi}(Qu) = \pi(u)$.
Then $(X/\ker\pi,\wh{\pi})$ is a normed vector lattice.
The map $j: (X/\ker\pi, \wh{\pi})\to (Z_\theta,\rho_\theta)^*$ given by $\la u^*, j(Qu)\ra = \la u,u^*\ra$ is a well-defined lattice isometric embedding.  In particular, $(X/\ker \pi,\wh{\pi})$ is $(q_\theta, \ol{q_2})$-concave.
\end{prop}

\begin{proof}
Since $C_\theta \subseteq K_p^\theta B_{X^*}$ is a bounded solid subset of $X^*$, $D$ is a closed convex solid subset of $X$ that contains $K_p^{-\theta}B_X$. 
In particular, $\ker\pi$ is a vector lattice ideal in $X$ and $\wh{\pi}$ is a well-defined lattice norm on the vector lattice $X/\ker \pi$.
If $Qu =0$, then $u \in \lambda D$ for all $\lambda >0$.
Thus, $|\la u, u^*\ra| \leq \lambda$ for any $u^*\in C_\theta$ and $\lambda >0$. Hence, $\la u,u^*\ra = 0$ for any $u^*\in C_\theta$. It follows that $\la u,u^*\ra =0$ for any $u^*\in Z_\theta$.
This shows that the map $j$ is well defined.
For any $u\in X$,
\[
\|j(Qu)\| = \sup_{\rho_\theta(u^*)\leq 1}|\la u^*,j(Qu)\ra| =
 \sup_{u^*\in C_\theta}|\la u,u^*\ra|.\]
If $\wh{\pi}(Qu) <1$, then $u\in D$ and hence $|\la u,u^*\ra| \leq 1$ for any $u^*\in C_\theta$, since $D = (C_\theta)^\circ$. 
On the other hand, if $\wh{\pi}(Qu) >1$, then $u\notin  D =  (C_\theta)^\circ$.
Hence, there exists $u^*\in C_\theta$ such that $\la u, u^*\ra > 1$.
Thus, $\|j(Qu)\| > 1$ by the above equation.  This completes the proof that $j$ is an isometric embedding.
\medskip

For any $u \in X$ and $u^*\in (Z_\theta)_+$,
\[ \la u^*, |j(Qu)|\ra = \sup_{|w^*|\leq |u^*|}|\la w^*, jQu\ra| = \sup_{|w^*|\leq |u^*|}|\la u, w^*\ra| = \la |u|, u^*\ra = \la u^*, jQ|u|\ra.
\] 
Thus $|jQu| = jQ|u|$.
Since $Q$ is a lattice homomorphism, $Q|u| = |Qu|$.  Therefore, $|jQu| = j|Qu|$.  This shows that $j$ is a lattice homomorphism.
\medskip

Finally, since $(Z_\theta, \rho_\theta)$ is $(q^*_\theta, (\ol{q_2})^*)$-convex, its dual is $(q_\theta, \ol{q_2})$-concave by Theorem \ref{thm: adjoint of convex is concave}. Thus, 
 $(X/\ker \pi, \wh{\pi})$, which is lattice isometric to a sublattice of $(Z_\theta, \rho_\theta)^*$, is $(q_\theta, \ol{q_2})$-concave.
\end{proof}

The next step is to show that $T$ factors through $(X/\ker \pi, \wh{\pi})$ as $T=RQ$, with $R$ a $(p_\theta,\infty)$-convex operator.

\begin{prop}\label{prop: definition of R}
Define $R: (X/\ker \pi,\wh{\pi}) \to Y$ by $R(Qu) = Tu$.
Then $R$ is a well-defined bounded linear operator.
\end{prop}

\begin{proof}
Suppose that $u\in D$.  Then $\la u, u^*\ra \leq 1$ for all $u^*\in C_\theta$.
Since $T^*B_{Y^*} \subseteq K_q^{1-\theta}C_\theta$, $|\la Tu, v^*\ra| \leq  K_q^{1-\theta}$ for all $v^*\in B_{Y^*}$. This proves that  $\|Tu\| \leq  K_q^{1-\theta}\pi(u)$.
In particular,  if $Qu= 0$, then $\pi(u) = 0$ and thus $Tu=0$.
Therefore, $R$ is well defined.
Also, the same inequality shows that 
\[ \|R(Qu)\| = \|Tu\| \leq  K_q^{1-\theta}\pi(u) =   K_q^{1-\theta}\wh{\pi}(Qu).\]
Thus, $R$ is bounded.
\end{proof}

Define $S_\theta:Y^*\to Z_\theta$ by $S_\theta v^* = T^*v^*$.  By Proposition \ref{prop: qtheta1 convavity of the factor}, $S_\theta$ is $(p^*_\theta,1)$-concave, where $p^*_\theta = \frac{p^*}{1-\theta}$ (equivalently, $\frac{1}{p_\theta} = \frac{1-\theta}{p}+\theta$).

\begin{prop}\label{prop: R is convex}
Define  $\psi: (Z_\theta, \rho_\theta)\to (X/\ker\pi, \wh{\pi})^*$ by $\la Qu,\psi u^*\ra = \la u,u^*\ra$.  Then $\psi$ is a well-defined bounded linear operator so that  $R^* = \psi S_\theta$. Consequently, $R$ is $(p_\theta,\infty)$-convex.
\end{prop}

\begin{proof}
Note that $|\la u,u^*\ra|\leq 1$ if $u\in D$ and $u^*\in C_\theta$.  Hence $\la u,u^*\ra = 0$ if $Qu = 0$ and $u^*\in Z_\theta$. This shows that $\psi u^*$ is a well-defined functional on $X/\ker\pi$. 
Also, if $\wh{\pi}(Qu)<1$, then $u\in D$ and hence $|\la Qu, \psi u^*\ra| = |\la u,u^*\ra| \leq \rho_\theta(u^*)$ for any $u^*\in Z_\theta$.  Hence, $\psi u^*\in (X/\ker\pi, \wh{\pi})^*$ and $\|\psi u*\| \leq \rho_\theta(u^*)$.
Therefore, $\psi$ is a bounded linear operator. 
\medskip

For any $v^*\in Y^*$ and any $u\in X$,
\[ \la Qu, \psi S_\theta v^*\ra = \la u, S_\theta v^*\ra = \la u, T^*v^*\ra = \la Tu,v^*\ra = \la RQu,v^*\ra = \la Qu, R^*v^*\ra.
\]
Thus, $R^* = \psi S_\theta$.  
Finally, since $S_\theta$ is $(p^*_\theta,1)$-concave, so is $R^*$.  Hence $R$ is $(p_\theta,\infty)$-convex.
\end{proof}

Finally, we can state the analogue of \cite[Theorem 15]{RT} for operators that are simultaneously $(p,\infty)$-convex and $(q,q_2)$-concave.

\begin{thm}\label{thm: factorization of pinfty convex and qq2 concave operator}
Let $X,Y$ be Banach lattices and let $T:X\to Y$ be a bounded linear operator that is both $(p,\infty)$-convex and $(q,q_2)$-concave ($q_2 \in \{q,1\}$). Let $0<\theta < 1$. Determine $p_\theta, q_\theta$ by the equations $q_\theta = \frac{q}{\theta}$ and $\frac{1}{p_\theta} = \frac{\theta}{p} +{1-\theta}$. Set $\ol{q_2} = q_\theta$ if $q_2 = q$, and $\ol{q_2} = 1$ if $q_2 = 1$. 
There exists 
\begin{enumerate}
\item a $(q_\theta,\ol{q_2})$-concave Banach lattice $\mathcal{Z}_\theta$ and  a $(p_\theta,\infty)$-convex Banach lattice $W_\theta$, 
\item a Class $\cC$ operator $\Phi: {W}_\theta \to Y$ and a Class $\cD$ operator $\Psi: X\to \mathcal{Z}_\theta$,
\item a bounded linear operator $T_\theta: \mathcal{Z}_\theta\to {Y}_\theta$, 
\end{enumerate}
so that $T = \Phi T_\theta\, \Psi$.
\end{thm}

\begin{proof}
Let $\mathcal{Z}_\theta$ be the completion of $(X/\ker \pi, \wh{\pi})$ and let $\mathcal{R}:\mathcal{Z}_\theta\to Y$ be the continuous extension to $\mathcal{Z}_\theta$ of the bounded linear operator $R:(X/\ker \pi, \wh{\pi})\to Y$. Also, let $\Psi:X\to \mathcal{Z}_\theta$ be given by $\Psi u = Qu$. Then $\Psi$ is a Class $\cD$ operator.
By Proposition \ref{prop: definition of R}, $T= \mathcal{R}\Psi$.
It follows from Proposition \ref{prop: R is convex} that $\mathcal{R}$ is $(p_\theta,\infty)$-convex.
By \Cref{thm: optimal factorization upe lpe case}(1), there is a $(p_\theta,\infty)$-convex Banach lattice $W_\theta$, a Class $\cC$ operator $\Phi:W_\theta\to Y$ and a bounded linear operator $T_\theta:\mathcal{Z}_\theta\to W_\theta$ so that $\mathcal{R} = \Phi T_\theta$. Therefore, $T= \Phi T_\theta\,\Psi$, as claimed.
\end{proof}

\section{Push-outs in classes of convex Banach lattices}\label{sec: PO}

Recall that in a given category, for objects $X_0,X_1,X_2$ and morphisms $T_i:X_0\rightarrow X_i$, $i=1,2$, a \textit{push-out diagram} is an object $PO=PO(T_1,T_2)$ together with morphisms $S_i:X_i\rightarrow PO$, $i=1,2$, making the following diagram commutative
\begin{equation*}
	\xymatrix{
		X_1 \ar@{->}[r]^{S_1}& PO  \\
		X_0 \ar@{->}[r]_{T_2} \ar@{->}[u]^{T_1} & X_2 \ar@{->}[u]_{S_2} 
	}
\end{equation*}
and with the universal property that if $S'_i:X_i\rightarrow Y$ are such that $S'_1T_1=S'_2T_2$, then there is a unique $\gamma:PO\rightarrow Y$ such that $\gamma S_i=S'_i$ for $i=1,2,$ as follows:
\begin{equation*}
	\xymatrix{
		&  &  Y \\
        X_1 \ar@{->}[r]^{S_1} \ar@/^1.0pc/[rru]^{S'_1}& PO \ar@{->}[ru]^{\gamma} & \\
		X_0 \ar@{->}[r]_{T_2} \ar@{->}[u]^{T_1} & X_2 \ar@{->}[u]_{S_2}  \ar@/_1.0pc/[uur]_{S'_2} &
	}
\end{equation*}
In the category $\mathcal{BL}$ of Banach lattices and lattice homomorphisms, \textit{isometric push-outs}, i.e., push-outs satisfying the additional conditions that $\max\{\|S_1\|,\|S_2\|\}\leq 1$ and $\|\gamma\|\leq \max\{ \|S'_1\|,\|S'_2\| \}$, were shown to exist in \cite{AT}. The construction of such isometric push-outs was extended in \cite{GLTT} to certain classes of Banach lattices $\cC$:

\begin{thm}\label{push-out}
Let $\mathcal C$ be a class of Banach lattices which is closed under the operations of taking quotients and sublattices of $\ell_\infty$-sums of spaces in $\cC$. Given Banach lattices $X_0,X_1,X_2$ in $\cC$ and lattice homomorphisms $T_i:X_0\rightarrow X_i$ for $i=1,2$, there is a Banach lattice $PO^{\cC}$ in $\cC$ and lattice homomorphisms $S_1,S_2$ so that the following is an isometric push-out diagram in $\cC$:
\begin{equation*}
	\xymatrix{
		X_1 \ar@{->}[r]^{S_1}& PO^{\ol{\cC}}  \\
		X_0 \ar@{->}[r]_{T_2} \ar@{->}[u]^{T_1} & X_2 \ar@{->}[u]_{S_2} 
	}
\end{equation*}
\end{thm}

\Cref{push-out} applies, in particular, when $\mathcal{C}$ is the class $\cC_p$ of $p$-convex Banach lattices or the class $\cC_\pinf$ of Banach lattices with upper $p$-estimates (both with constant 1). For these classes, the problem of extending \cite[Theorem 4.4]{AT} was also studied in \cite{GLTT}. This result establishes that isometric push-outs in the category $\mathcal{BL}$ preserve isometries. Namely, if in the definition of the push-out diagram above, $T_1:X_0\rightarrow X_1$ is contractive and $T_2:X_0\rightarrow X_2$ is an isometric embedding, it was shown in \cite{AT} that $S_1:X_1\rightarrow PO$ is an isometric embedding. For push-outs in the classes $\cC_p$ and $\cC_\pinf$, it was shown in \cite[Theorem 4.7]{GLTT} that $S_1$ is a $K_\cC$-isomorphic embedding, with $K_\cC=2^\frac{1}{p^*}$ when $\cC=\cC_p$ and $K_\cC=2^\frac{1}{p^*}\gamma_p$ when $\cC=\cC_\pinf$. The proof provided in \cite{GLTT} uses Maurey's and Pisier's Factorization Theorems. It turns out that both constants can be improved by using different techniques. \medskip

In the $p$-convex setting, we can obtain $K_{\cC_p}=1$, that is, a complete extension of \cite[Theorem 4.4]{AT}, using a $p$-concavification and $p$-convexification argument (see, for instance, \cite[Section 1.d]{LT2}):

\begin{thm}\label{thm: PO preserving isometries p-convex}
Let $1<p<\infty$ and let $\cC_p$ be the class of $p$-convex Banach lattices (with constant 1). Let
\begin{equation*}
	\xymatrix{
		X_1 \ar@{->}[r]^{S_1}& PO^{\cC_p}  \\
		X_0 \ar@{->}[r]_{T_2} \ar@{->}[u]^{T_1} & X_2 \ar@{->}[u]_{S_2} 
	}
\end{equation*}
be an isometric push-out diagram in $\cC_p$. If $T_1$ is contractive and $T_2$ is an isometric embedding, then $S_1$ is an isometric embedding.
\end{thm}

\begin{proof}
Since $X_i$, $i=0,1,2$, are $p$-convex Banach lattices, we can $p$-concavify the spaces. Let $X_{i(p)}$ be their $p$-concavifications. The maps $T_i:X_{0(p)}\rightarrow X_{i(p)}$, $i=1,2$, will still be contractive lattice homomorphisms, since lattice homomorphisms preserve functional calculus, and moreover $T_2$ will still be an isometry:
\[\|T_2z\|_{X_{2(p)}}:=\|T_2 z\|_{X_2}^p=\|z\|_{X_0}^p=:\|z\|_{X_{0(p)}}, \quad z\in X_{0(p)}.\]
Let $PO$ be a push-out in the category of all Banach lattices for this new diagram:
\begin{equation*}
	\xymatrix{
		X_{1(p)} \ar@{->}[r]^{R_1}& PO  \\
		X_{0(p)} \ar@{->}[r]_{T_2} \ar@{->}[u]^{T_1} & X_{2(p)} \ar@{->}[u]_{R_2} 
	}
\end{equation*}
By \cite[Theorem 4.4]{AT} we know that $R_1:X_{1(p)} \rightarrow PO$ is an isometric embedding. If we $p$-convexify this diagram, we find that for every $i=0,1,2$, $(X_{i(p)})^{(p)}=X_i$, as $X_i$ has $p$-convexity constant 1. Arguing as before, the lattice homomorphisms $R_i:X_i\rightarrow PO^{(p)}$, $i=1,2$, are contractive, and $R_1$ is an isometry. Since $PO^{(p)}$ is $p$-convex with constant 1, we can apply the minimality property of $PO^{\cC_p}$ to find a lattice homomorphism $\gamma: PO^{\cC_p}\rightarrow PO^{(p)}$ such that $R_i=\gamma S_i$, $i=1,2$, and $\|\gamma\|\leq \max\{\|R_1\|, \|R_2\|\}\leq 1$:
\begin{equation*}
	\xymatrix{
		&  &  PO^{(p)} \\
        X_1 \ar@{->}[r]^{S_1} \ar@/^1.0pc/[rru]^{R_1}& PO^{\cC_p} \ar@{->}[ru]^{\gamma} & \\
		X_0 \ar@{->}[r]_{T_2} \ar@{->}[u]^{T_1} & X_2 \ar@{->}[u]_{S_2}  \ar@/_1.0pc/[uur]_{R_2} &
	}
\end{equation*}
It follows that
\[\|y\|_{X_1}=\|R_1 y\|_{PO^{(p)}}=\|\gamma S_1 y\|_{PO^{(p)}}\leq \|S_1 y\|_{PO^{\cC_p}} \leq \|y\|_{X_1}, \]
so $S_1$ is an isometry.
\end{proof}

Unfortunately, this technique cannot be used in the upper $p$-estimates setting. However, we can adapt the abstract factorization techniques developed in \Cref{sec: abstract factorizations} to improve the constant $K_{\cC_\pinf}$ obtained in \cite{GLTT}.

\begin{thm}\label{parallelisometries}
Let $1<p<\infty$ and let $\cC_{p,\infty}$ be the class of Banach lattices with upper $p$-estimates (with constant 1). Let
\begin{equation*}
	\xymatrix{
		X_1 \ar@{->}[r]^{S_1}& PO^{\cC_\pinf}  \\
		X_0 \ar@{->}[r]_{T_2} \ar@{->}[u]^{T_1} & X_2 \ar@{->}[u]_{S_2} 
	}
\end{equation*}
be an isometric push-out diagram in $\cC_\pinf$. If $T_1$ is contractive and $T_2$ is an isometric embedding, then $S_1$ is a $2^\frac{1}{p^*}$-embedding.
\end{thm}

\begin{proof}
We begin the proof by considering the isometric push-out diagram in the category $\mathfrak{BL}$ from \cite[Theorem 4.4]{AT}, so that $R_1$ is an isometric embedding:
\[\xymatrix{
		X_1  \ar@{->}[r]^{R_1} & PO   \\
        X_0  \ar@{->}[u]^{T_1}  \ar@{->}[r]_{T_2} & X_2 \ar@{->}[u]_{R_2} 
}\]
Our aim is to construct a $(\pinf)$-convex Banach lattice $Z$ with constant 1 and a lattice homomorphism $\psi:Z\rightarrow PO$ with $\|\psi\|\leq 2^\frac{1}{p^*}=:K_p$ such that for $i=1,2$, $R_i$ factors through $Z$ as $\psi \wh R_i$, with $\wh R_i:X_i\rightarrow Z$ contractive lattice homomorphisms. Once this is established, it follows that $\wh R_1$ must be an isomorphism, since 
\[\|y\|_{X_1}=\|R_1y\|_{PO}\leq \|\psi\| \|\wh R_1 y\|_Z\leq K_p \|\wh R_1 y\|_Z\]
for every $y\in X_1$. Moreover, by the universal property of the push-out in $\cC_\pinf$ established in \Cref{push-out}, there exists a contractive lattice homomorphism $\gamma:PO^{\cC_\pinf}\rightarrow Z$ that makes the following diagram commute:
\begin{equation*}
	\xymatrix{
		&  &  Z \\
        X_1 \ar@{->}[r]^{S_1} \ar@/^1.0pc/[rru]^{\wh R_1}& PO^{\cC_\pinf} \ar@{->}[ru]^{\gamma} & \\
		X_0 \ar@{->}[r]_{T_2} \ar@{->}[u]^{T_1} & X_2 \ar@{->}[u]_{S_2}  \ar@/_1.0pc/[uur]_{\wh R_2} &
	}
\end{equation*}
It follows that
\[\|y\|_{X_1}\leq K_p\|\wh R_1 y\|_Z= K_p\|\gamma S_1 y\|_Z\leq  K_p\|S_1 y\|_{PO^{\cC_\pinf}},\]
as we wanted to show.\medskip

Let us construct the Banach lattice $Z$. We will take $Z=\spn B$, for $B\subseteq PO$ a certain closed, convex and solid set (so that $Z$ endowed with the Minkowski functional is a Banach lattice by \Cref{lem: minkowski functional}) such that $Z$ is $(\pinf)$-convex with constant 1 and $R_i (B_{X_i})\subseteq B\subseteq K_p B_{PO}$. The latter will imply that $R_i$ can be factored through $Z$ as $R_i=\psi \wh R_i$, with $\wh R_i:X_i\rightarrow Z$ given by $\wh R_i x=R_i x$ and $\psi:Z\rightarrow PO$ the formal identity, so that $\wh R_i$ is contractive and $\|\psi\|\leq K_p$. Using the ideas from \Cref{sec: abstract factorizations}, let us choose $B=C^{\circ\circ }$, where
\begin{align*}
    C= \Bigg \{u\in PO \,:\, & \exists (v_j)_{j=1}^m\subseteq X_1, (w_j)_{j=m+1}^n\subseteq X_2 \text{ such that } |u|\leq \intoo[3]{\bigvee_{j=1}^m |R_1 v_j|}\vee \intoo[3]{ \bigvee_{j=m+1}^n |R_2 w_j|} \\
    & \text{ and } \intoo[3]{\sum_{j=1}^m \|v_j\|^{p} + \sum_{j=m+1}^n \| w_j\|^{p}}^\frac{1}{p}\leq 1 \Bigg\},
\end{align*}
and the polars are taken with respect to $\la PO, PO^* \ra$. $B$ is the closed convex hull of the solid set $C$, so it is closed, convex and solid. Clearly $R_i (B_{X_i})\subseteq C\subseteq B$. Moreover, $C\subseteq K_p B_{PO}$. Indeed, if $u\in C$, there exist $(v_j)_{j=1}^m\subseteq X_1$ and $(w_j)_{j=m+1}^n\subseteq X_2$ such that
\[|u|\leq \intoo[3]{\bigvee_{j=1}^m |R_1 v_j|}\vee \intoo[3]{ \bigvee_{j=m+1}^n |R_2 w_j|} \text{ and } \intoo[3]{\sum_{j=1}^m \|v_j\|^{p} + \sum_{j=m+1}^n \| w_j\|^{p}}^\frac{1}{p}\leq 1.\]
Therefore, using that the $X_i$ satisfy upper $p$-estimates with constant 1, and hence $R_i$ are $(\pinf)$-convex with constant 1, we obtain that
\begin{align*}
    \|u\|& \leq \norm[3]{\intoo[3]{\bigvee_{j=1}^m |R_1 v_j|}\vee \intoo[3]{ \bigvee_{j=m+1}^n |R_2 w_j|}}\leq \norm[3]{\bigvee_{j=1}^m |R_1 v_j|} + \norm[3]{\bigvee_{j=m+1}^n |R_2 w_j|}\\
    & \leq \intoo[3]{\sum_{j=1}^m \|v_j\|^{p}}^\frac{1}{p} + \intoo[3]{\sum_{j=m+1}^n \|w_j\|^{p}}^\frac{1}{p}\leq 2^\frac{1}{p^*}.
\end{align*}
Since $B_{PO}$ is already closed and convex, it follows that $B\subseteq K_p B_{PO}$. It remains to show that $Z$ satisfies an upper $p$-estimate (equivalently, it is $(\pinf)$-convex) with constant 1. This can be done by reproducing the results from Sections \ref{sec: duality convexity} and \ref{sec: concavity and convexity} for the $(\pinf)$-convex case, that is, choosing $\sigma$ to be the $c_0$-norm and $\tau$ the $\ell_p$-norm, and replacing the roles of the sets $C_T^{\tau, \sigma}$ and $D^{T^*}_{\tau^*,\sigma^*}$ by $C$ and
\begin{align*}
    D= \Bigg \{u^*\in PO^* \,:\, & |u^*|\geq \intoo[3]{\sum_{j=1}^n |u^*_j|^{p_2^*}}^\frac{1}{p_2^*} \text{ implies that }  \\
    & \forall J\subseteq \{1,\ldots,n\} \intoo[3]{\sum_{j\in J} \|R_1^*u^*_j\|^{p^*} + \sum_{j\notin J}\|R_2^*u^*_j\|^{p^*}}^\frac{1}{p^*}\leq 1  \Bigg\},
\end{align*}
respectively. More specifically, one needs to show that $C^\circ =D$ following the proofs of \Cref{prop: duality sets C and D}(1) and \Cref{thm: polarity sets C and D}(1), and then adapt the argument of \Cref{prop: double polar of C is tau-sigma convex} to show that for every sequence $(u_i)^n_{i=1}\subseteq B=C^{\circ \circ}$ and $(\al_i)^n_{i=1}\subseteq \R$ such that $\sum^n_{i=1}|\al_i|^p \leq 1$, it follows that $\bigvee^n_{i=1} |\alpha_i u_i|\in  B$, i.e., that the Minkowski functional of $B$ induces a $(\pinf)$-convex norm with constant 1 over $Z$. We omit the details.
\end{proof}

The contrast between Theorems \ref{thm: PO preserving isometries p-convex} and \ref{parallelisometries} makes us wonder whether the constant $K_p=2^\frac{1}{p^*}$ in the latter result is just an artifact of the proof or, on the contrary, it is a fundamental difference between $p$-convexity and upper $p$-estimates. It should be noted that if $K_p$ could be taken to be 1 in \Cref{parallelisometries}, this would allow us to adapt other categorical constructions from \cite{AT}, such as Banach lattices of universal disposition or separably injective Banach lattices, to the category of Banach lattices with upper $p$-estimates.

\section{Searching for a universal separable Banach lattice with upper \texorpdfstring{$p$}{}-estimates}\label{sec: universal}

As we have seen in the previous section, certain universal constructions in the category of Banach lattices can be easily generalized to the $p$-convex setting via a $p$-concavification and $p$-convexification argument, but this process does not carry over to the class of Banach lattices with upper $p$-estimates. Another example concerns injectively universal Banach lattices. Recall that in \cite{LLOT} it has been shown that $C(\Delta, L_1[0,1])$ ($\Delta$ denotes the Cantor set) is injectively universal for the class of all separable Banach lattices, that is, every separable Banach lattice embeds lattice isometrically into $C(\Delta, L_1[0,1])$. Using $p$-concavification and $p$-convexification, we can show the following:

\begin{thm}\label{thm: universal p convex}
   $C(\Delta, L_p[0,1])$ is injectively universal for the class of separable $p$-convex Banach lattices (with constant 1). That is, if $X$ is a $p$-convex Banach lattice with constant 1, then there exists a lattice isometric embedding from $X$ into $C(\Delta, L_p[0,1])$.
\end{thm}

\begin{proof}
    The proof is composed of two claims: First, we show that the $p$-convexification of the Banach lattice $C(\Delta, L_1[0,1])$ is injectively universal for the class of separable $p$-convex Banach lattices, and then we show that this space can be identified with $C(\Delta, L_p[0,1])$.
\medskip

    To show the first claim, let $X$ be a $p$-convex Banach lattice with constant 1, and consider its $p$-concavification $X_{(p)}$, which is also a Banach lattice. By \cite[Theorem 1.1]{LLOT}, there exists a lattice isometric embedding $T:X_{(p)}\rightarrow C(\Delta, L_1[0,1])$. If we $p$-convexify both spaces, we find that $T$ is again a lattice isometry from $X=(X_{(p)})^{(p)}$ into $C(\Delta, L_1[0,1])^{(p)}$. 
    Therefore, $C(\Delta, L_1[0,1])^{(p)}$ is injectively universal for the class of separable $p$-convex Banach lattices.\medskip

    Next, we identify $C(\Delta, L_1[0,1])^{(p)}$ with $C(\Delta, L_p[0,1])$. Let $\Phi:C(\Delta, L_p[0,1])\rightarrow C(\Delta, L_1[0,1])^{(p)}$ be the map that sends any function $f\in C(\Delta, L_p[0,1])$ to $f^p$. This map is well defined: Given any $t\in \Delta$, $f(t)^p\in L_1[0,1]$, and clearly $f^p$ is continuous as a function from $\Delta$ into $L_1[0,1]$, so $f^p$ belongs to $C(\Delta, L_1[0,1])$, which coincides as a set with $C(\Delta, L_1[0,1])^{(p)}$. If we denote by $\oplus$ and $\odot$ the sum and product by scalars defined in $C(\Delta, L_1[0,1])^{(p)}$ (see \cite[Section 1.d]{LT2}), it can be easily checked that $\Phi(f+g)=\Phi f \oplus \Phi g$ and $\Phi (\lambda f)=\lambda \odot \Phi f$ for any $f,g \in C(\Delta, L_p[0,1])$ and $\lambda \in \R$, so $\Phi$ is a linear map from $C(\Delta, L_p[0,1])$ to $C(\Delta, L_1[0,1])^{(p)}$. It is straightforward to check that $\Phi$ is a lattice homomorphism. We also observe that
    \[\|\Phi f\|_{C(\Delta, L_1[0,1])^{(p)}}=\| f^p\|_{C(\Delta, L_1[0,1])}^\frac{1}{p}=\| f\|_{C(\Delta, L_p[0,1])},\]
    so $\Phi$ is a lattice isometry. Finally, we can argue similarly to check that the map that sends every $h\in C(\Delta, L_1[0,1])^{(p)}$ to $h^\frac{1}{p}$ is the inverse map of $\Phi$, so in particular $\Phi$ is an onto isometry. 
\end{proof}

As we discussed earlier, there is no analogue of $p$-convexification and $p$-concavification for upper $p$-estimates, so the proof given above does not directly adapt to show the existence of an injectively universal Banach lattice for separable Banach lattices with upper $p$-estimates (with constant 1). In analogy with \Cref{thm: universal p convex}, one could expect that if such a Banach lattice existed, it should be of the form $C(\Delta, Z)$ for some separable Banach lattice with upper $p$-estimates $Z$ that is somehow canonical in the setting of upper $p$-estimates. A reasonable candidate could be the order continuous part of $\wlp$. Given a measure space $(\Omega, \Sigma, \mu)$, the \textit{order continuous part of $\wlp(\mu)$}, denoted $\ocwlp(\mu)$, is defined as the closure in $\wlp(\mu)$ of the simple functions. For instance, the prototypical example of a function that belongs to $\wlp[0,1]$ but not to $\ocwlp[0,1]$ is $t^{-\frac{1}{p}}$. The following characterizes when a function belongs to $\ocwlp(\mu)$. We include its proof for the sake of completeness. 

\begin{prop}\label{characterization oc weak Lp}
   Let $f\in \wlp(\mu)$. Then $f\in \ocwlp(\mu)$ if and only if 
   \[\lim_{\mu(A)\rightarrow 0} \mu(A)^{-\frac{1}{p^*}}\int_A |f|\, d\mu =0 \quad \text{and} \quad \lim_{\mu(A)\rightarrow \infty} \mu(A)^{-\frac{1}{p^*}}\int_A |f|\, d\mu =0.\]
\end{prop}

\begin{proof}
    Let us denote by 
    \[Y:= \cbr[3]{f\in \wlp(\mu): \lim_{\mu(A)\rightarrow 0} \mu(A)^{-\frac{1}{p^*}}\int_A |f|\, d\mu =0 \quad \text{and} \quad \lim_{\mu(A)\rightarrow \infty} \mu(A)^{-\frac{1}{p^*}}\int_A |f|\, d\mu =0}.\]
    It is clear that $Y$ is an ideal in $\wlp(\mu)$. Moreover, it is closed. Indeed, if $(f_n)_n\subseteq Y$ is a sequence that converges to $f\in \wlp(\mu)$, then for every $\eps>0$ there is some $N\in \N$ such that $\|f-f_N\|_{\wlp}<\frac{\eps}{2}$. For this function $f_N\in Y$, there are $\delta>0$ and $M>0$ such that for every measurable set $A$ with $\mu(A)<\delta$ or $\mu(A)>M$, it follows that $\mu(A)^{-\frac{1}{p^*}}\int_A |f_N|\, d\mu<\frac{\eps}{2}$. Therefore, 
    \[\mu(A)^{-\frac{1}{p^*}}\int_A |f|\, d\mu\leq \mu(A)^{-\frac{1}{p^*}}\int_A |f_N|\, d\mu+\mu(A)^{-\frac{1}{p^*}}\int_A |f-f_N|\, d\mu <\frac{\eps}{2}+\|f-f_N\|_{\wlp}<\eps,\]
    so $f\in Y$. Now, in order to show that $\ocwlp(\mu)\subseteq Y$, it suffices to check that every characteristic function belongs to $Y$. Let $B$ be a measurable subset with $0<\mu(B)<\infty$ and $\eps>0$. Then, for every measurable set $A$ with $\mu(A)<\eps^p$ we have that
    \[\mu(A)^{-\frac{1}{p^*}}\int_A \chi_B\, d\mu= \mu(A)^{-\frac{1}{p^*}}\mu(A\cap B)\leq \mu(A)^{\frac{1}{p}}<\eps.\]
    On the other hand, if $\mu(A)>(\mu(B)/\eps)^{p^*}$, it follows that
    \[\mu(A)^{-\frac{1}{p^*}}\int_A \chi_B\, d\mu = \mu(A)^{-\frac{1}{p^*}}\mu(A\cap B)\leq \mu(A)^{-\frac{1}{p^*}}\mu(B)<\eps.\]
    We conclude that $\chi_B\in Y$.\medskip

    To show the reverse inclusion, let us fix $f\in Y$ and $\eps>0$, and find $\delta>0$ and $M>0$ satisfying $\mu(A)^{-\frac{1}{p^*}}\int_A |f|\, d\mu<\frac{\eps}{3}$ for every measurable set $A$ with $\mu(A)<\delta$ or $\mu(A)>M$. Since $f\in \wlp(\mu)$, $\mu(\{|f|>n\})$ decreases to 0 as $n$ grows to infinity (see the definition of the quasinorm $\vvvert\cdot \vvvert_{\wlp}$), so there is $n\in \N$ such that $\mu(\{|f|>n\})<\delta$ and hence $\|f\chi_{\{|f|>n\}}\|_{\wlp}<\frac{\eps}{3}$. If we also fix $m\in \N$ such that $\frac{M^\frac{1}{p}}{m}<\frac{\eps}{3}$, we have that for every measurable subset $A$, either $\mu(A)>M$, so that
    \[\mu(A)^{-\frac{1}{p^*}}\int_A |f\chi_{\{|f|\leq \frac{1}{m}\}}|\, d\mu \leq \mu(A)^{-\frac{1}{p^*}}\int_A |f|\, d\mu<\frac{\eps}{3},\]
    or $\mu(A)\leq M$, and hence
    \[\mu(A)^{-\frac{1}{p^*}}\int_A |f\chi_{\{|f|\leq \frac{1}{m}\}}|\, d\mu \leq \frac{\mu(A)^\frac{1}{p}}{m}<\frac{\eps}{3}.\]
    In follows that $\|f\chi_{\{|f|\leq \frac{1}{m}\}}\|_{\wlp}<\frac{\eps}{3}$, so that $\|f-f\chi_{\{\frac{1}{m}<|f|\leq n\}}\|_{\wlp}<\frac{2\eps}{3}$. Note that $g=f\chi_{\{m<|f|\leq n\}}$ is bounded and has finite support, so there exists a simple function $s$ supported in $\supp g$ such that $\|g-s\|_{L_\infty}<\eps/\mu(\supp g)^\frac{1}{p}$ and hence $\|g-s\|_{\wlp}<\frac{\eps}{3}$. Therefore, for every $\eps>0$ we have found a simple function $s$ such that $\|f-s\|_{\wlp}<\eps$, so $f\in \ocwlp(\mu)$.
\end{proof}

It turns out that the space $\ocwlp$ plays a distinguished role in the class of separable Banach lattices with upper $p$-estimates. Indeed, \Cref{thm: factorization Pisier operator} can be improved so that the $(p,\infty)$-convex operator $T$ factors through $\ocwlp$ when $E$ is separable.

\begin{prop}\label{prop: Pisier factorization separable case}
    Let $E$ be a separable Banach space,  $1<p<\infty$ and $T:E\rightarrow L_1(\mu)$ a $(p,\infty)$-convex operator with constant $C$. There exists a normalized $h\in L_1(\mu)_+$ such that $\{h=0\}\subseteq \{Tx=0\}$ for every $x\in E$, and $T=M R$, where $R:E\rightarrow L_{p,\infty}^\circ(h\cdot\mu)$ given by $Rx=h^{-1}Tx$ is bounded with $\|R\|\leq K^{\frac{1}{p^*}}\gamma_p C$ for some $1\leq K\leq 2$, and $M:L_{p,\infty}^\circ(h\mu)\rightarrow L_1(\mu)$ is the multiplication operator by $h$.
\end{prop}

\begin{proof}
    By Pisier's factorization \Cref{thm: factorization Pisier operator}, there exists a normalized function $g\in L_1(\mu)_+$ such that $\{g=0\}\subseteq \{Tx=0\}$ for every $x\in E$ and
    \[\left\|\frac{Tx}{g} \right\|_{L_{p,\infty}(g\mu)}\leq \gamma_p C\|x\|.\]
    Let us take a dense subset $(x_n)_{n=1}^\infty\subset S_E$ and assume without loss of generality that $Tx_n\neq 0$. We define 
    \[f_n=\frac{|Tx_n|}{\|Tx_n\|_{L_1(\mu)}} \quad \text{and} \quad h=\frac{1}{K}\left[\left(\sum_n \frac{f_n}{2^n}\right)\vee g\right],\]
    where $K$ is chosen so that $\|h\|_{L_1(\mu)}=1$, so in particular $1\leq K\leq 2$. Clearly $h$ is positive and $\{h=0\}\subseteq \{g=0\}$. Moreover, for every $x\in E$ and measurable subset $A$ we have
    \[\left(\int_A h\, d\mu \right)^{-\frac{1}{p^*}} \int_A|Tx|\, d\mu \leq K^{\frac{1}{p^*}}\left(\int_A g\, d\mu \right)^{-\frac{1}{p^*}} \int_A|Tx|\, d\mu \leq K^{\frac{1}{p^*}} \left\|\frac{Tx}{g} \right\|_{L_{p,\infty}(g\mu)}\leq K^{\frac{1}{p^*}} \gamma_p C\|x\|, \]
    so the operator $Rx=\frac{Tx}{h}$ is bounded from $E$ to $L_{p,\infty}(h\cdot\mu)$ with $\|R\|\leq K^{\frac{1}{p^*}} \gamma_p C$. Let us show that the range of $R$ is contained in $L_{p,\infty}^\circ(h\cdot \mu)$. To do so, fix $n\in \N$ and consider a measurable subset $A$, so that
    \begin{align*}
        \left(\int_A h\, d\mu \right)^{-\frac{1}{p^*}} \int_A|Tx_n|\, d\mu & \leq (2^n K)^{\frac{1}{p^*}}\left(\int_A f_n\, d\mu \right)^{-\frac{1}{p^*}} \|Tx_n\|_{L_1(\mu)} \int_A f_n\, d\mu \\
        & \leq C(2^n K)^{\frac{1}{p^*}} \left(\int_A f_n\, d\mu \right)^{\frac{1}{p}} \leq C 2^n K \left(\int_A h\, d\mu \right)^{\frac{1}{p}}.
    \end{align*}
    It is clear that $\lim_{\mu(A)\rightarrow 0} \int_A h\, d\mu =0$, so using Proposition \ref{characterization oc weak Lp} we conclude that $Rx_n \in L_{p,\infty}^\circ(h\cdot\mu)$ for every $n\in \N$. Since $R$ is continuous and $(x_n)_n$ is dense in $E$, it follows that $R(E)\subseteq L_{p,\infty}^\circ(h\cdot\mu)$. Finally, it is straightforward to check that $T=MR$.
\end{proof}

In particular, it follows from the above that every positive operator from a separable Banach lattice with upper $p$-estimates to $L_1$ factors through some $\ocwlp$. One could try to apply this result in order to obtain an (isomorphic) analogue of \Cref{thm: universal p convex} for upper $p$-estimates by using the following scheme: If $X$ is a Banach lattice satisfying an upper $p$-estimate with constant 1, let $T:X\rightarrow C(\Delta, L_1[0,1])$ be the isometric lattice embedding given by \cite[Theorem 1.1]{LLOT}, and consider for each $t\in \Delta$ the point evaluation $\delta_t: C(\Delta, L_1[0,1])\rightarrow L_1[0,1]$, which is a lattice homomorphism. Then, $\delta_tT:X\rightarrow L_1[0,1]$ is a $(p,\infty)$-convex operator, so by \Cref{prop: Pisier factorization separable case} it factors through $L_{p,\infty}^\circ(h_t)$ as $\delta_t T=M_tR_t$ for some normalized $h_t\in L_1[0,1]_+$, with $\|R_t\|\leq 2^\frac{1}{p^*} \gamma_p$. Since $([0,1], \mathfrak{m},h\cdot \lambda)$, where $\lambda$ denotes the Lebesgue measure and $\mathfrak{m}$ the $\sigma$-algebra of measurable sets, is a separable, non-atomic probability measure space, by Carath\'eodory's Isomorphism Theorem (\cite{Caratheodory}, see also \cite[Section 41]{Halmos}) it must be measure isomorphic to the unit interval with the Lebesgue measure. In particular, $L_{p,\infty}^\circ(h_t)$ can be identified isometrically with $\ocwlp[0,1]$. Therefore, each of the maps $\delta_tT:X\rightarrow L_1[0,1]$ can be factored through $\ocwlp[0,1]$ with a uniform bound on the norm of the factors. If one could show that the above process of choosing the density $h_t$ could be carried out continuously with respect to $t\in \Delta$, then it would produce a lattice isomorphic embedding $R: X\rightarrow C(\Delta,\ocwlp[0,1])$ that sends every $x\in X$ to the function $Rx(t)=R_tx$, $t\in \Delta$. However, the way the density is obtained in the original proof of \Cref{thm: factorization Pisier operator} (see \cite{Pisier86}) makes it difficult to deduce whether such a choice of density preserving continuity can be made.\medskip

Another unclear aspect of this approach is whether the candidate for the universal space should be $C(\Delta,\ocwlp[0,1])$ or $C(\Delta,\ocwlp[0,\infty))$. Clearly, if the former satisfies the universal property, the latter will as well, as $\ocwlp[0,1]$ embeds lattice isometrically into $\ocwlp[0,\infty)$. However, the converse might not be true. Indeed, $\ocwlp[0,1]$ and $\ocwlp[0,\infty)$ behave very differently, as the following results show.

\begin{prop}\label{prop: embedings into ocwlp 0 infty}
    If $\mu$ is a separable measure over $\Omega$, $\ocwlp(\mu)$ embeds lattice isomorphically into $\ocwlp[0,\infty)$.
\end{prop}

\begin{proof}
    Let us decompose $\mu$ into an atomic part $\mu_a$ and a non-atomic part $\mu_{na}$, so that $\ocwlp(\mu)$ is lattice isomorphic to $\ocwlp(\mu_a)\oplus \ocwlp(\mu_{na})$. Depending on whether it has a finite or infinite number of atoms, $\ocwlp(\mu_a)$ is lattice isometric to $\ell_\pinf^n(\omega)$ or $\ocwlp(\N,\omega)$, where $\omega$ is a weighted counting measure over the set $\{1,\ldots,n\}$ or $\N$, respectively. These spaces are in turn lattice isometric to the span in $\ocwlp[0,\infty)$ of characteristic functions of disjoint sets $A_k$ of measure $\omega_k$ ($k=1,\ldots, n$ or $k\in \N$, respectively). On the other hand, $\mu_{na}$ is measure isomorphic to the Lebesgue measure on $[0,r)$, where $r=\mu_{na}(\Omega)$. Hence, $\ocwlp(\mu)$ is isomorphic to a sublattice of $\ocwlp[0,\infty)$.
\end{proof}

Therefore, $\ocwlp[0,\infty)$ contains any other separable $\ocwlp(\mu)$ as a sublattice. On the other hand, $\ocwlp[0,1]$ does not contain $\ell_\pinf^\circ$ nor $\ocwlp[0,\infty)$ lattice isomorphically:

\begin{prop}\label{prop: ocwellp do not embed into ocwlp01}
    The Banach lattice $\ell_\pinf^\circ$ does not lattice embed into $\ocwlp[0,1]$.
\end{prop}

One should compare this result with \cite{Leung91}, where it is established that $\ell_\pinf$ embeds as a complemented sublattice of $L_\pinf[0,1]$, with a positive projection. The proof of \Cref{prop: ocwellp do not embed into ocwlp01} uses the following lemma (cf.~\cite{CD}):

\begin{lem}\label{lem: disjoint subsequence in ocwlp01 is c0}
    Let $(f_n)_n\subseteq \ocwlp[0,1]$ be a positive, disjoint and normalized (with respect to the quasinorm $\vvvert \cdot\vvvert_{\wlp}$) sequence. Then $(f_n)_n$ has a subsequence equivalent to the $c_0$ basis.
\end{lem}

\begin{proof}
    Since simple functions are dense in $\ocwlp[0,1]$, after a small perturbation, we can assume that $f_n\in L_\infty[0,1]$ for every $n\in \N$. Let $r_n=\lambda(\supp f_n)$, where $\lambda$ denotes the Lebesgue measure. Then $r_n\rightarrow 0$, and passing to a subsequence if necessary, we can assume that there is a positive sequence $(m_n)_n\subseteq \R$ increasing to $\infty$ such that $\|f_n\|_{L_\infty}\leq m_n$ and $\sum_{k=n+1}^\infty r_k\leq m_n^{-p}$ for all $n\in \N$. Let $f$ be $\sum_n f_n$ almost everywhere (note that, since the $f_n$ are pairwise disjoint, this coincides almost everywhere with the supremum). To prove the statement, it suffices to show that $f\in \wlp[0,1]$. Take any $t\in \R_+$, and choose $n$ so that $m_{n-1}\leq t <m_n$, where $m_0=0$. Since $\|f_k\|_{L_\infty}\leq m_k$, $\{f_k>t\}=\emptyset$ for $1\leq k < n$. Thus,
    \[\lambda\{f>t\}=\sum_{k=n}^\infty \lambda\{f_k>t\}\leq\lambda\{f_n>t\}+\sum_{k=n+1}^\infty r_k\leq t^{-p}+ m_n^{-p}\leq 2 t^{-p}.\]
    As $t\in \R_+$ was arbitrary and $f\geq 0$, we conclude that $f\in \wlp[0,1]$, as desired.
\end{proof}

\begin{proof}[Proof of \Cref{prop: ocwellp do not embed into ocwlp01}]
    If such a lattice embedding $T:\ell_\pinf^\circ\rightarrow \ocwlp[0,1]$ existed, then $(Te_n)_n\subseteq \ocwlp[0,1]$ would contain a subsequence equivalent to the basis of $c_0$, where $(e_n)_n$ denotes the canonical basis of $\ell_\pinf^\circ$. In particular, $(e_n)_n$ would also contain a subsequence equivalent to the $c_0$ basis, which is impossible, since every subsequence of $(e_n)_n$ is equivalent to itself. 
\end{proof}

As a consequence of the above, we get:

\begin{cor}\label{cor: ocwlps non isomorphic}
    $\ell_\pinf^\circ$, $\ocwlp[0,1]$ and $\ocwlp[0,\infty)$ are pairwise non-isomorphic as Banach lattices.
\end{cor}

\begin{proof}
    Clearly, $\ell_\pinf^\circ$ cannot be lattice isomorphic to any of the other two spaces, as it is purely atomic. $\ocwlp[0,1]$ and $\ocwlp[0,\infty)$ are not lattice isomorphic, as the latter contains a lattice isometric copy of $\ell_\pinf^\circ$ (the span of the characteristic functions on the intervals $[n-1,n)$), whereas the first one does not by \Cref{prop: ocwellp do not embed into ocwlp01}.
\end{proof}

These facts point towards $C(\Delta, \ocwlp[0,\infty))$ as a more suitable candidate for a universal Banach lattice for upper $p$-estimates. Finally, it should be noted that none of these candidates will be \emph{isometrically} universal for Banach lattices with upper $p$-estimates with constant 1, as both embed lattice isometrically into an $\ell_\infty$-sum of $\wlp$-spaces, and we have shown in \Cref{ex: finite dim BL} that there are Banach lattices with upper $p$-estimates that cannot be embedded isometrically into any space of this form.

\begin{question}\label{quest: universal for separable upe}
    Is there a separable Banach lattice with upper $p$-estimates (with constant 1) that is (isomorphically or isometrically) universal for this class? That is, a separable Banach lattice $Y$ satisfying an upper $p$-estimate with constant 1 such that for every separable Banach lattice $X$ satisfying an upper $p$-estimate with constant 1 there exists a lattice (isometric or isomorphic) embedding of $X$ into $Y$? Do $C(\Delta, \ocwlp[0,1])$ or $C(\Delta, \ocwlp[0,\infty))$ have this isomorphic property?
\end{question}

\section{Operators factoring through a \texorpdfstring{$(\pp)$}{}-convex and a \texorpdfstring{$(\qq)$}{}-concave operator}\label{sec: Reisner}
In \cite{Reisner}, the author studies the ideal of operators that factor through a $p$-convex and a $q$-concave operator, which is denoted $M_{p,q}$. This ideal is shown to be perfect, that is, it coincides with its double adjoint, and can be identified with the ideal $I_{p,q}$ of operators factoring through a diagonal operator from $L_p(\mu)$ to $L_q(\mu)$ defined in \cite{GLR}. In this section, we provide analogous results for the ideal of operators factoring through a $(\pp)$-convex and a $(\qq)$-concave operator, where $1<p,q<\infty$, $p\leq p_2\leq \infty$ and $1\leq q_2\leq q$. Note that in light of \cite[Corollary 16.7]{DJT}, up to constants, the only relevant cases are $p_2\in\{p,\infty\}$ and $q_2\in\{1,q\}$.\medskip  

Throughout this section, we will use the definitions and notation for tensor products and operator ideals given in \cite{GLR}. For instance, recall that an \textit{operator ideal} is a class $A$ of operators between Banach spaces such that for each pair of Banach spaces $E$ and $F$ the set $A(E,F)=A\cap \lL(E,F)$ satisfies that:
\begin{enumerate}
    \item If $x^*\in E^*$ and $y\in F$, then $x^*\otimes y\in A(E,F)$, where $x^*\otimes y (x)=x^*(x)y$ is a rank one operator.
    \item $A(E,F)$ is a linear subspace of $\lL(E,F)$.
    \item If $X$ and $Y$ are Banach spaces, $U\in \lL(X,E)$, $T\in A(E,F)$ and $V\in \lL(F,Y)$, then $VTU\in A(X,Y)$.
\end{enumerate}
Moreover, a function $\alpha:A\rightarrow \R_+$ is an \textit{ideal norm} if one has:
\begin{enumerate}[(4)]
    \item $\alpha(x^*\otimes y)=\|x^*\|\|y\|$ for every $x^*\in E^*$ and $y\in F$.
    \item $\alpha(S+T)\leq \alpha(S)+\alpha (T)$ for every $S,T \in A(E,F)$.
    \item If $X$ and $Y$ are Banach spaces, $U\in \lL(X,E)$, $T\in A(E,F)$ and $V\in \lL(F,Y)$, then $\alpha(VTU)\leq \|V\|\alpha(T)\|U\|$.
\end{enumerate}
A \textit{normed (operator) ideal} $[A,\alpha]$ is an ideal $A$ with an ideal norm $\alpha$, and we say that it is a \textit{Banach ideal} if each $A(E,F)$ is a Banach space when endowed with $\alpha$.
\medskip

Given a normed ideal $[A,\alpha]$, the \textit{adjoint ideal} $[A,\alpha]^*=[A^*,\alpha^*]$ is defined as follows (observe that here $*$ denotes the adjoint of an operator ideal instead of the dual of a Banach space): For each pair of Banach spaces $E$ and $F$, $A^*(E,F)$ is the set of all $T\in \lL(E,F)$ such that there exists $\rho>0$ satisfying that for all finite-dimensional Banach spaces $X,Y$ and all $U\in A(Y,X)$, $V\in \lL(X,E)$ and $W\in \lL(F,Y)$,
\[|\tr WTVU|\leq \rho \|W\|\|V\|\alpha(U),\]
where $\tr$ denotes the trace of an operator defined from a finite-dimensional Banach space to itself. The norm $\alpha^*(T)$ is given by the infimum of all $\rho$ satisfying the above. $\alpha^*$ is always a complete ideal norm. An operator ideal $[A,\alpha]$ is \textit{perfect} if $[A,\alpha]=[A,\alpha]^{**}$. This notion is equivalent to the ideal being \textit{finitely generated}, \textit{maximal} or being the adjoint of some other ideal (cf.~\cite[Theorem 8.11]{Ryan}).\medskip

Regarding tensor products, recall that if $E$ and $F$ are Banach spaces, a norm $\eta$ on $E\otimes F$ is a \textit{reasonable crossnorm} if it satisfies the following properties:
\begin{enumerate}
    \item $\eta(x\otimes y)\leq \|x\|\|y\|$ for every $x\in E$ and $y\in F$.
    \item For every $x^*\in E^*$ and $y^*\in F^*$, the linear functional $x^*\otimes y^*$ on $E\otimes F$ that takes values $x^*\otimes y^*(x\otimes y)=x^*(x)y^*(y)$ on the elementary products $x\otimes y\in E\otimes F$ is bounded, and $\|x^*\otimes y^*\|\leq \|x^*\|\|y^*\|$.
\end{enumerate}
A \textit{uniform crossnorm} is an assignment to each pair of Banach spaces of a reasonable crossnorm on their tensor product in such a way that whenever $E,F,G$ and $H$ are Banach spaces and $S:E\rightarrow G$ and $T:F\rightarrow H$ are operators, then the operator $S\otimes T:E\otimes F\rightarrow G\otimes H$ is bounded and $\|S\otimes T\|\leq \|S\|\|T\|$. It is easy to check that every uniform crossnorm induces an ideal norm over the operator ideal $\cF$ of all finite rank operators via the identification $\cF(E,F)=E^*\otimes F$ described in property (1) in the definition of an operator ideal. We will denote this normed ideal by $[\cF,\eta]$. A uniform crossnorm $\eta$ is \textit{finitely generated} (see \cite[p.~129]{Ryan}, this is also referred to as a $\otimes$-norm in \cite{Reisner}), if for all $u\in E\otimes F$,
\[\eta(u)=\inf \eta(u;M,N),\]
where the infimum is taken over all finite-dimensional subspaces $M\subseteq E$ and $N\subseteq F$ such that $u\in M\otimes N$, and $\eta(u;M,N)$ denotes the $\eta$-norm of $u$ as an element of $M\otimes N$. \medskip
 
Let us fix $1<p,q<\infty$, $p\leq p_2\leq \infty$ and $1\leq q_2\leq q$. We introduce the following definition, which generalizes the ideal $[M_{p,q},\mu_{p,q}]$ of operators that factor through a $p$-convex and a $q$-concave operator introduced in \cite[Section 2]{Reisner}:

\begin{defn}\label{def: ideal M of factorizations through convex and concave}
    Let $E$ and $F$ be Banach spaces. An operator $T:E\rightarrow F$ belongs to the class $M_\ppqq(E,F)$ if there exists a Banach lattice $X$, a $(\pp)$-convex operator $U:E\rightarrow X$ and a $(\qq)$-concave operator $V:X\rightarrow F^{**}$ such that $J_F T=V U$, where $J_F:F\rightarrow F^{**}$ denotes the canonical embedding:
    \[\xymatrix{
		E  \ar@{->}[rd]_{U} \ar@{->}[r]^{T} & F  \ar@{->}[r]^{J_F} & F^{**}  \\
		&  X \ar@{->}[ru]_{V} &
    }\]
    We define $\mu_\ppqq(T)=\inf \cbr[0]{K^{(\pp)}(U)K_{(\qq)}(V)}$, where the infimum is taken over all possible factorizations above.
\end{defn}

As is customary, when $p_2=p$ we will simply write $p$ instead of $(p,p_2)$, and similarly when $q_2=q$. Our first goal is to show that $M_\ppqq$ is a perfect ideal. This will be done by showing that it is the adjoint of $[\cF,\eta_\ppqq]^*$, where $\eta_\ppqq$ is a uniform crossnorm defined as follows:

\begin{defn}\label{def: eta norm}
    Given $u\in E\otimes F$, we define
    \[\eta_\ppqq (u)=\inf\cbr[3]{\theta_\ppqq (\{x_i\otimes y_i\}_{i=1}^n): (x_i)_{i=1}^n \subseteq E, (y_i)_{i=1}^n\subseteq F, u=\sum_{i=1}^n x_i\otimes y_i },\]
    where
    \begin{align*}
        \theta_\ppqq (\{x_i\otimes y_i\}_{i=1}^n)=\sup \Bigg\{&\sum_{i=1}^n \intoo[3]{\sum_{k=1}^\infty |x_k^*(x_i)|^{p_2}}^\frac{1}{p_2} \intoo[3]{\sum_{j=1}^\infty |y_j^*(y_i)|^{q^*_2}}^\frac{1}{q^*_2}: \\
        &  \|(x_k^*)_{k=1}^\infty\|_{\ell_p(E^*)} \leq 1, \|(y_j^*)_{j=1}^\infty\|_{\ell_{q^*}(F^*)} \leq 1\Bigg\}
    \end{align*}
 and $q^*$ and $q^*_2$ denote the conjugate exponents of $q$ and $q_2$, respectively (when $p_2=\infty$ or $q_2=1$, the corresponding $p_2$ or $q_2^*$-sum must be understood as a supremum). 
\end{defn}

Following the steps of \cite{Reisner}, the first thing to show is that:

\begin{prop}\label{prop: eta is computable on finite dimensions}
    $\eta_\ppqq$ is finitely generated. 
\end{prop}

\begin{proof}
    It suffices to check that whenever we fix a representation $u=\sum_{i=1}^n x_i\otimes y_i$ with $(x_i)_{i=1}^n \subseteq M$ and $(y_i)_{i=1}^n\subseteq N$ for some finite-dimensional subspaces $M\subseteq E$ and $N\subseteq F$, then 
    \[\theta_\ppqq (\{x_i\otimes y_i\}_{i=1}^n;E,F)=\theta_\ppqq (\{x_i\otimes y_i\}_{i=1}^n;M,N).\]
    
    Clearly, $\theta_\ppqq (\{x_i\otimes y_i\}_{i=1}^n;E,F)\leq\theta_\ppqq (\{x_i\otimes y_i\}_{i=1}^n;M,N)$, since every pair of sequences $(x_k^*)_{k=1}^\infty\in B_{\ell_p(E^*)}$ and $(y_j^*)_{j=1}^\infty\in B_{\ell_{q^*}(F^*)}$ can be restricted to $M$ and $N$, respectively, to induce sequences in $B_{\ell_p(M^*)}$ and $B_{\ell_{q^*}(N^*)}$, and the values of the sums considered in the computation of $\theta_\ppqq$ remain unchanged.
\medskip

    On the other hand, the inequality $\theta_\ppqq (\{x_i\otimes y_i\}_{i=1}^n;E,F)\geq\theta_\ppqq (\{x_i\otimes y_i\}_{i=1}^n;M,N)$ is a consequence of the Hahn--Banach Theorem: Every pair of sequences $(v_k^*)_{k=1}^\infty\in B_{\ell_p(M^*)}$ and $(w_j^*)_{j=1}^\infty\in B_{\ell_{q^*}(N^*)}$ can be extended to sequences in $B_{\ell_p(E^*)}$ and $B_{\ell_{q^*}(F^*)}$ without altering their values on $(x_i)_{i=1}^n$ and $(y_i)_{i=1}^n$.
\end{proof}

The next proposition shows the connection between the $\eta_\ppqq$-norm and $(\pp)$-convex and $(\qq)$-concave operators.

\begin{prop}\label{prop: eta factors fd}
    For every $u\in E^*\otimes F= \cF(E,F)\subseteq L(E,F)$, 
    \[\eta_\ppqq(u)=\inf\cbr[0]{K^{(\pp)}(R)K_{(\qq)}(S)},\]
    where the infimum is taken over all possible factorizations $u=SR$ with $Z$ a finite-dimensional Banach space with a 1-unconditional basis (endowed with the order induced by the basis), $R:E\rightarrow Z$ and $S:Z\rightarrow F$.
\end{prop}

\begin{proof}
    The proof is just an adaptation of \cite[Proposition 1.3]{Reisner} with the corresponding changes in the exponents. We include it for the sake of completeness.\medskip

    Let $u\in E^*\otimes F$ and assume that it factors as $u=RS$ with $Z$ a finite-dimensional Banach space with a 1-unconditional basis, $R:E\rightarrow Z$ and $S:Z\rightarrow F$. Let $(e_i,e_i^*)_{i=1}^n$ be the corresponding basis of $Z$ and its biorthogonal functionals. We define $x_i^*=R^* e_i^*$ and $y_i=Se_i$, so that $u=SR=\sum_{i=1}^n x_i^*\otimes y_i$. By definition,
    \begin{align*}
        K^{(\pp)}(R)& =\sup \cbr[3]{\norm[3]{\sum_{i=1}^n \intoo[3]{\sum_k |x_i^*(x_k)|^{p_2}}^\frac{1}{p_2}e_i}:  \|(x_k)_{k=1}^\infty\|_{\ell_p(E)}\leq 1},\\
        K_{(\qq)}(S)=K^{(q^*,q_2^*)}(S^*)& =\sup \cbr[3]{\norm[3]{\sum_{i=1}^n \intoo[3]{\sum_k |y_j^*(y_i)|^{q_2^*}}^\frac{1}{q_2^*}e_i^*}:  \|(y_j^*)_{k=1}^\infty\|_{\ell_{q^*}(F^*)}\leq 1}.
    \end{align*}
    Therefore, given $\eps>0$, we can choose $(x_k)_{k=1}^\infty\in \ell_p(E)$ and $(y_j^*)_{k=1}^\infty\in \ell_{q^*}(F^*)$ that $\eps$-almost attain the supremum in the definition of $\theta_\ppqq (\{x_i^*\otimes y_i\}_{i=1}^n)$, so that
    \begin{align*}
        \theta_\ppqq (\{x_i^*\otimes y_i\}_{i=1}^n)& \leq \sum_{i=1}^n \intoo[3]{\sum_{k=1}^\infty |x_k^*(x_i)|^{p_2}}^\frac{1}{p_2} \intoo[3]{\sum_{j=1}^\infty |y_j^*(y_i)|^{q^*_2}}^\frac{1}{q^*_2} +\eps\\
        & \leq  \left \la \sum_{i=1}^n \intoo[3]{\sum_k |x_i^*(x_k)|^{p_2}}^\frac{1}{p_2}e_i, \sum_{i=1}^n \intoo[3]{\sum_k |y_j^*(y_i)|^{q_2^*}}^\frac{1}{q_2^*}e_i^* \right \ra +\eps \\
        & \leq K^{(\pp)}(R)K_{(\qq)}(S)+\eps.
    \end{align*}
    Hence, $\eta_\ppqq(u)\leq \inf\{ K^{(\pp)}(R)K_{(\qq)}(S)\}$.\medskip

    For the reverse inequality, let $\eps>0$ and suppose that $u=\sum_{i=1}^n x_i^*\otimes y_i$ is a representation such that
    \[\theta_\ppqq (\{x_i^*\otimes y_i\}_{i=1}^n)\leq \eta_\ppqq(u) +\eps.\]
    Let $Z$ be $\R^n$ endowed with the norm
    \[\vvvert z \vvvert = \sup \cbr[3]{\sum_{i=1}^n \abs[3]{z_i \intoo[3]{\sum_k |y_j^*(y_i)|^{q_2^*}}^\frac{1}{q_2^*} } :  \|(y_j^*)_{k=1}^\infty\|_{\ell_{q^*}(F^*)}\leq 1} \]
    for $z=(z_i)_{i=1}^n\in \R^n$. Note that the unit vector basis $(e_i)_{i=1}^n$ of $\R^n$ is 1-unconditional for this norm, as $\vvvert z \vvvert$ is only determined by $(|z_i|)$. Let us define $R:E\rightarrow Z$ by $R=\sum_{i=1}^n x_i^*\otimes e_i$ and $S:Z\rightarrow F$ by $Sz=\sum_{i=1}^n z_i y_i$. It follows that $u=SR$, 
    \begin{align*}
         K^{(\pp)}(R)& =\sup \cbr[4]{\left \vvvert \sum_{i=1}^n \intoo[3]{\sum_k |x_i^*(x_k)|^{p_2}}^\frac{1}{p_2}e_i\right \vvvert :  \|(x_k)_{k=1}^\infty\|_{\ell_p(E)}\leq 1}\\
         &= \theta_\ppqq (\{x_i^*\otimes y_i\}_{i=1}^n)\leq \eta_\ppqq(u) +\eps,
    \end{align*}
    and 
    \begin{align*}
        K_{(\qq)}(S)&=K^{(q^*,q_2^*)}(S^*) =\sup \cbr[4]{\left \vvvert \sum_{i=1}^n \intoo[3]{\sum_k |y_j^*(y_i)|^{q_2^*}}^\frac{1}{q_2^*}e_i^* \right \vvvert :  \|(y_j^*)_{k=1}^\infty\|_{\ell_{q^*}(F^*)}\leq 1}\\
        & = \sup \cbr[3]{\sum_{i=1}^n \abs[3]{z_i \intoo[3]{\sum_k |y_j^*(y_i)|^{q_2^*}}^\frac{1}{q_2^*} }: \vvvert z \vvvert \leq 1,  \|(y_j^*)_{k=1}^\infty\|_{\ell_{q^*}(F^*)}\leq 1} \\
        & = \sup \{\vvvert z \vvvert: \vvvert z \vvvert\leq 1\}=1.
    \end{align*}
    Therefore, $\inf \{K^{(\pp)}(R)K_{(\qq)}(S)\}\leq \eta_\ppqq(u)$.
\end{proof}

Next, we show that the ideal $M_\ppqq$ defined in \Cref{def: ideal M of factorizations through convex and concave} is perfect.

\begin{prop}\label{prop: M is perfect ideal}  The following properties hold:
    \begin{enumerate}
        \item $[M_\ppqq,\mu_\ppqq]$ is a perfect normed ideal of operators.
        \item $[M_\ppqq,\mu_\ppqq]=[\cF,\eta_\ppqq]^{**}$.
    \end{enumerate}
\end{prop}

Proving \Cref{prop: M is perfect ideal} requires some previous lemmata. Namely, we need \cite[Lemma 2]{Reisner}, which is also valid for the case $p=\infty$ even though it was originally stated for $1\leq p<\infty$, and an analogue of \cite[Lemma 3]{Reisner}:

\begin{lem}\label{lem: finite approximation of p sums}
    Let $G$ be a finite-dimensional subspace of a Dedekind complete Banach lattice $Y$. Given $\delta>0$ there are pairwise disjoint $(y_i)_{i=1}^n\subseteq Y_+$ and an operator $W:G\rightarrow \spn (y_i)_{i=1}^n$ such that $\|x-Wx\|\leq \delta \|x\|$ for every $x\in G$, and moreover we have that
    \[\norm[3]{\intoo[3]{\sum_{j=1}^m|x_j|^p}^\frac{1}{p} - \intoo[3]{\sum_{j=1}^m|Wx_j|^p}^\frac{1}{p}} \leq \delta \intoo[3]{\sum_{j=1}^m\|x_j\|^p}^\frac{1}{p}\]
    for all $(x_j)_{j=1}^m\subseteq G$ and $1\leq p\leq \infty$.
\end{lem}

\begin{lem}\label{lem: standard argument}
    Assume that $E$ or $F$ are finite-dimensional and $T\in  \lL(E,F)$ is such that $J_FT=VU$ with $X$ a Banach lattice, $U:E\rightarrow X$ a $(\pp)$-convex operator and $V:X\rightarrow F^{**}$ a $(\qq)$-concave operator. Then, for every $\eps>0$ there exists a finite-dimensional $Z$ with a 1-unconditional basis and a factorization
    \[\xymatrix{
		E  \ar@{->}[rd]_{R} \ar@{->}[rr]^{T} & & F    \\
		&  Z \ar@{->}[ru]_{S} 
    }\]
    with $K^{(\pp)}(R)K_{(\qq)}(S)\leq (1+\eps) K^{(\pp)}(U)K_{(\qq)}(V)$.
\end{lem}

\begin{proof}
    Let us first assume that $E$ is finite-dimensional. By \Cref{thm: maximal factorization tau sigma concave operator}, $V:X\rightarrow F^{**}$ factors through a $(\qq)$-concave Banach lattice $Y$ with constant 1 as $V=V_2 V_1$, where $V_1:X\rightarrow Y$ is a lattice homomorphism with $\|V_1\|\leq K_{(\qq)}(V)$ and $V_2:Y\rightarrow F^{**}$ is contractive. Without loss of generality, we can assume that $Y$ is Dedekind complete. Otherwise, we replace $Y$ by $Y^{**}$, $V_1$ by $J_Y V_1$ and $V_2$ by $J_{F^*}^*V_2^{**}$. Let us fix $\eps>0$, and consider $\delta>0$ to be chosen later. We can apply \Cref{lem: finite approximation of p sums} to the finite-dimensional subspace $G=V_1U(E)\subseteq Y$ to find a finite-dimensional sublattice $Z\subseteq Y$ and an isomorphism $W:G\rightarrow Z$ such that for all $(x_j)_{j=1}^m\subseteq G$,
    \[\norm[3]{\intoo[3]{\sum_{j=1}^m|x_j|^{p_2}}^\frac{1}{p_2} - \intoo[3]{\sum_{j=1}^m|Wx_j|^{p_2}}^\frac{1}{p_2}} \leq \delta \intoo[3]{\sum_{j=1}^m\|x_j\|^{p_2}}^\frac{1}{p_2}.\]
    Observe that $V_1 U$ is $(\pp)$-convex with constant $K^{(\pp)}(U)K_{(\qq)}(V)$, as $V_1$ is a lattice homomorphism, and $Z$ is also $(\qq)$-concave with constant 1. It follows that for all $(u_j)_{j=1}^m\subseteq E$
    \begin{align*}
        \norm[3]{\intoo[3]{\sum_{j=1}^m|WV_1U u_j|^{p_2}}^\frac{1}{p_2}}&\leq \norm[3]{\intoo[3]{\sum_{j=1}^m|V_1U u_j|^{p_2}}^\frac{1}{p_2}} + 
        \norm[3]{\intoo[3]{\sum_{j=1}^m|V_1U u_j|^{p_2}}^\frac{1}{p_2} - \intoo[3]{\sum_{j=1}^m|WV_1U u_j|^{p_2}}^\frac{1}{p_2}} \\
        & \leq K^{(\pp)}(U)K_{(\qq)}(V) \intoo[3]{\sum_{j=1}^m\|u_j\|^{p}}^\frac{1}{p} + \delta \intoo[3]{\sum_{j=1}^m\|V_1U u_j\|^{p_2}}^\frac{1}{p_2}\\
        & \leq K^{(\pp)}(U)K_{(\qq)}(V) \intoo[3]{\sum_{j=1}^m\|u_j\|^{p}}^\frac{1}{p} + \delta \|V_1U\| \intoo[3]{\sum_{j=1}^m\| u_j\|^{p_2}}^\frac{1}{p_2}\\
        & \leq K^{(\pp)}(U)K_{(\qq)}(V) (1+\delta) \intoo[3]{\sum_{j=1}^m\|u_j\|^{p}}^\frac{1}{p},
    \end{align*}
    where, in the last step, we have used the fact that $p_2\geq p$. Therefore, $R=WV_1U: E\rightarrow Z$ is $(\pp)$-convex with $K^{(\pp)}(R)\leq  K^{(\pp)}(U)K_{(\qq)}(V) (1+\delta) $.\medskip

    To construct $S$, observe that $W$ satisfies that $(1-\delta) \|x\|\leq \|Wx\|\leq (1+\delta)\|x\|$ for every $x\in G$. Denote by  $k=\dim G$, which is independent of $\eps$ and $\delta$, and $n=\dim Z\geq k$, and let $(e_j)_{j=1}^k$ be a normalized basis of $G$. Then $(We_j)_{j=1}^k$ is a seminormalized basis of $W(G)\subseteq Z$. Let $(z^*_j)_{j=1}^k\subseteq Z^*$ be the Hahn--Banach extensions to $Z$ of a set of seminormalized functionals which are biorthogonal to $(We_j)_{j=1}^k$, so that $\|z^*_j\|\leq \frac{1}{1-\delta}$ for every $j=1,\ldots, k$. If we write $K=\bigcap_{j=1}^k \ker z_j^*$, then $Z=W(G)\oplus K$. Then, every $z\in Z$ can be represented as $z=\sum_{j=1}^k c_j We_j + y$ with $y\in K$ and $c_j=z_j^*(z)$ for every $j$. Let us define $M: Z=W(G)\oplus K \rightarrow Y$ that sends every $z=\sum_{j=1}^k c_j We_j + y$ to $Mz=\sum_{j=1}^k c_j e_j + y$, so that
    \begin{align*}
        \|Mz-z\|& =\norm[3]{\sum_{j=1}^k c_j e_j-\sum_{j=1}^k c_j We_j}\leq \delta \norm[3]{\sum_{j=1}^k c_j e_j} \\
        & =\delta \norm[3]{\sum_{j=1}^k z_j^*(z) e_j} \leq \delta \sum_{j=1}^k\|z_j^*\|\|z\|\leq \frac{\delta k}{1-\delta}\|z\|.
    \end{align*}
    It follows that $M$ is an isomorphism onto its image such that $\|M\|\leq 1+\frac{\delta k}{1-\delta}$ and $MWx=x$ for every $x\in G$. Let us now compose $M$ with $V_2$ to obtain a finite-dimensional subspace of $F^{**}$, to which we can apply the Local Reflexivity Principle (cf.~\cite[Theorem 6.3 and Exercise 1.76]{FHHMZ}) to get an isomorphism onto its image $Q:V_2M(Z)\rightarrow F$ such that $Q=J_F^{-1}$ on $V_2M(Z)\cap J_F(F)$ and $\|Q\|\leq \frac{1+\delta}{1-\delta}$. It follows that the operator $S=QV_2M:Z\rightarrow F$ satisfies
    \[SRu=QV_2MWV_1U u= QV_2V_1U u=QVU u=QJ_FT u=Tu\]
    for every $u\in E$, and moreover it is $(\qq)$-concave, since $Z$ was already $(\qq)$-concave, with constant
    \[K_{(\qq)}(S)=\|S\|\leq \|Q\|\|V_2\|\|M\|\leq \frac{1+\delta}{1-\delta}\intoo[2]{ 1+\frac{\delta k}{1-\delta}}.\]
    Taking $\delta$ small enough, we can achieve $K^{(\pp)}(R)K_{(\qq)}(S)\leq (1+\eps) K^{(\pp)}(U)K_{(\qq)}(V)$, as we wanted to show.\medskip

    Now, if we assume that $F$ is finite-dimensional, then $F=F^{**}$, so we can write $T=VU$. Then, it suffices to apply the first part of the proof to the factorization
    \[\xymatrix{
		F^*  \ar@{->}[rd]_{V^*} \ar@{->}[r]^{T^*} & E^* \ar@{->}[r]^{J_{E^*}} & E^{***}    \\
		&  Z \ar@{->}[ru]_{J_{E^*}U^*} 
    }\]
    and take adjoints in the new factorization to obtain the desired result.
\end{proof}

Now, we are ready to prove \Cref{prop: M is perfect ideal}:

\begin{proof}[Proof of \Cref{prop: M is perfect ideal}]
    The first statement follows from the second one by \cite[Theorem 8.11]{Ryan}. \medskip
    
    Let us first show that $[M_\ppqq,\mu_\ppqq]\subseteq [\cF,\eta_\ppqq]^{**}$. Let $E$ and $F$ be Banach spaces and $T\in M_\ppqq(E,F)$. Assume that $J_FT$ factors as in \Cref{def: ideal M of factorizations through convex and concave}. By \Cref{lem: standard argument}, for every finite-dimensional $G\subseteq E$ and $H\subseteq F^*$ (denote by $i:G\rightarrow E$ and $j:H\rightarrow F^*$ the embeddings) the operator $j^*J_FTi$ admits a factorization
    \[\xymatrix{
		G  \ar@{->}[rd]_{R} \ar@{->}[rr]^{j^*J_FTi} & & H^*    \\
		&  Z \ar@{->}[ru]_{S} 
    }\]
    with $Z$ finite-dimensional with a 1-unconditional basis, such that $K^{(\pp)}(R)K_{(\qq)}(S)\leq (1+\eps) K^{(\pp)}(U)K_{(\qq)}(V)$. Since $G$ and $H^*$ are finite-dimensional, using \cite[Proposition 9.2.2]{Pietsch} and \Cref{prop: eta factors fd} we get
    \[\eta_\ppqq^{**}(j^*J_FTi)=\eta_\ppqq(j^*J_FTi)\leq K^{(\pp)}(R)K_{(\qq)}(S)\leq (1+\eps) K^{(\pp)}(U)K_{(\qq)}(V).\]
    Observe that $[\cF,\eta_\ppqq]^{**}$ is already a perfect ideal, as it is the adjoint of $[\cF,\eta_\ppqq]^{*}$. Therefore, by \cite[Theorem 8.11]{Ryan}, it is a finitely generated ideal, so
    \begin{align*}
        \eta_\ppqq^{**}(T)&=\sup \{\eta_\ppqq^{**}(j^*J_FTi): G\subseteq E, H\subseteq F^*\text{ finite-dimensional}\}\\
        & \leq (1+\eps) K^{(\pp)}(U)K_{(\qq)}(V).
    \end{align*}
    It now follows that $\eta_\ppqq^{**}(T)\leq \mu_\ppqq(T)$.\medskip

    In order to show that $[M_\ppqq,\mu_\ppqq]\supseteq [\cF,\eta_\ppqq]^{**}$ we will make use of a standard ultraproduct argument, so we will omit some of the details. We refer to \cite{Heinrich} for basic terminology and notation on ultraproducts. Again, let us fix Banach spaces $E$ and $F$ and an operator in $[\cF(E;F),\eta_\ppqq]^{**}$. First of all, we make use of the fact that every Banach space embeds isometrically into the ultraproduct of a collection of its finite-dimensional subspaces (see \cite[Proposition 6.2]{Heinrich}). In particular, there exist index sets $K$ and $L$, ultrafilters $\fU_E$ and $\fU_{F^*}$ over $K$ and $L$ (respectively), collections $\{G_k\}_{k\in K}\subseteq\{G: G\subseteq E \text{ finite-dimensional}\}$ and $\{H_l\}_{l\in L}\subseteq\{H: H\subseteq F^* \text{ finite-dimensional}\}$ and isometric embeddings $i_E:E\rightarrow (G_k)_{\fU_E}$ and $j_{F^*}:F^*\rightarrow (H_l)_{\fU_{F^*}}$. It follows that 
    \[\fU_\times=\{C\subseteq K\times L: \exists A\in \fU_E, B\in \fU_{F^*} \text{ such that }A\times B\subseteq C\}\]
    defines a filter on $K\times L$, so by \cite[Proposition 1.1]{Heinrich} there exists an ultrafilter $\fU$ that dominates $\fU_\times$. If we put $G_{k,l}=G_k$ for every $l\in L$ and $H_{k,l}=H_l$ for every $k\in K$, there exist canonical isometric embeddings $\widehat{i}:(G_k)_{\fU_E}\rightarrow (G_{k,l})_\fU$ and $\widehat{j}:(H_l)_{\fU_{F^*}}\rightarrow (H_{k,l})_\fU$. Finally, note that by \cite[Section 7]{Heinrich} there exists a canonical isometric embedding $\varphi: (H_{k,l}^*)_\fU\rightarrow (H_{k,l})_\fU^*$.
\medskip

    Now, let us fix $\eps>0$ and consider for each pair $(k,l)\in K\times L$ the finite-dimensional subspaces $G_{k,l}\subseteq E$ and $H_{k,l}\subseteq F^*$ (this time we denote by $i_{k,l}:G_{k,l}\rightarrow E$ and $j_{k,l}:H_{k,l}\rightarrow F^*$ the embeddings). By \cite[Proposition 9.2.2]{Pietsch}, $T_{k,l}=j_{k,l}^*J_F Ti_{k,l}\in [\cF(G_{k,l},H^*_{k,l}),\eta_\ppqq]^{**}=[\cF(G_{k,l},H^*_{k,l}),\eta_\ppqq]$, so we can use \Cref{prop: eta factors fd} to find a finite-dimensional $Z_{k,l}$ with a 1-unconditional basis and operators $R_{k,l}:G_{k,l}\rightarrow Z_{k,l}$ and $S_{k,l}:Z_{k,l}\rightarrow H_{k,l}$ such that $T_{k,l}$ factors as $S_{k,l}R_{k,l}$ and
    \[K^{(\pp)}(R_{k,l})K_{(\qq)}(S_{k,l})\leq (1+\eps)\eta_\ppqq(j_{k,l}^*J_FTi_{k,l})\leq (1+\eps)\eta_\ppqq^{**}(T)\]
    (the last inequality holds because $[\cF,\eta_\ppqq]^{**}$ is a finitely generated ideal, see \cite[Theorem 8.11]{Ryan}). By rescaling the norm of $Z_{k,l}$ we can assume without loss of generality that $K_{(\qq)}(S_{k,l})=1$. It follows from \cite[Proposition 3.2]{Heinrich} that $X=(Z_{k,l})_\fU$ is a Banach lattice, and it is easy to check that $R=(R_{k,l})_\fU:(G_{k,l})_\fU\rightarrow X$ is $(\pp)$-convex with $K^{(\pp)}(R)\leq \sup_{k,l} K^{(\pp)}(R_{k,l})\leq (1+\eps)\eta_\ppqq^{**}(T)$ and $S=(S_{k,l})_\fU:X\rightarrow (H_{k,l}^*)_\fU$ is $(\qq)$-concave with $K_{(\qq)}(S)\leq \sup_{k,l} K_{(\qq)}(S_{k,l})=1$. Therefore, we have obtained the following commutative diagram:
    \[\xymatrix{
		E  \ar@{->}[r]^{i_E} & (G_k)_{\fU_E} \ar@{->}[r]^{\wh{i}} & (G_{k,l})_\fU \ar@{->}[dr]_{R} \ar@{->}[rr]^{(T_{k,l})_\fU} & & (H_{k,l}^*)_\fU \ar@{->}[r]^{\varphi} & (H_{k,l})_\fU^* \ar@{->}[r]^{\wh{j}^*} & (H_l)_{\fU_{F^*}}^* \ar@{->}[r]^{j_{F^*}^*} & F^{**}  \\
		& & & X \ar@{->}[ur]_{S} & & & &
    }\]
    A careful computation shows that $j_{F^*}^*\wh{j}^*\varphi (T_{k,l})_\fU \wh{i} i_E = J_FT$, so for every $\eps>0$ we have found a factorization of $J_FT$ through a Banach lattice $X$ with operators $U=R \wh{i} i_E$ and $V=j_{F^*}^*\wh{j}^*\varphi S$ such that
    \[\mu_\ppqq(T)\leq K^{(\pp)}(R)K_{(\qq)}(S)\leq  (1+\eps)\eta_\ppqq^{**}(T).\]
\end{proof}

We devote the last part of the section to identifying the ideal $M_\ppqq$. First, observe that \Cref{thm: double factorization Reisner} allows us to redefine the ideal $M_\ppqq$ in an alternative way:

\begin{prop}\label{prop: M equiv to factorization through BLs}
    Let $T:E\rightarrow F$. Then $T\in M_\ppqq(E,F)$ if and only if $J_FT$ factors as
    \[\xymatrix{
		E  \ar@{->}[d]_{U} \ar@{->}[r]^{T} & F  \ar@{->}[r]^{J_F} & F^{**}  \\
		X \ar@{->}[rr]_{Q} &   & Y \ar@{->}[u]_{V}
    }\]
    where $X$ is a $(\pp)$-convex Banach lattice, $Y$ is a $(\qq)$-concave Banach lattice, both with constant 1, $U$ and $V$ are bounded operators and $Q$ is an (almost interval preserving) lattice homomorphism. Moreover, $\mu_\ppqq(T)=\inf \|U\|\|Q\|\|V\|$, where the infimum is computed over all possible factorizations.
\end{prop}

From now on, we consider only $p_2\in\{p,\infty\}$ and $q_2\in\{1,q\}$, as they are the most relevant cases. The following definition is an analogue in our setting of the ideal $I_{p,q}$ of \textit{operators factoring through a diagonal $D:L_p(\mu)\rightarrow L_q(\mu)$} introduced in \cite{GLR}.

\begin{defn}\label{def: ideal I of factorizations through multiplication operator}
    An operator $T:E\rightarrow F$ belongs to $I_\ppqq(E,F)$ (i.e.~factors through a multiplication operator from $L_\pp(g_1\cdot \mu)$ to $L_\qq(g_2\cdot\mu)$) if there exists a positive measure $\mu$, functions $g_1,g_2\in L_1(\mu)$ and operators $A:E\rightarrow L_\pp(g_1\cdot \mu)$, $D:L_\pp(g_1\cdot \mu)\rightarrow L_\qq(g_2\cdot\mu)$ and $B:L_\qq(g_2\cdot \mu)\rightarrow F^{**}$ such that $J_F T=BDA$ and $D$ is a multiplication operator, that is, $Df=gf$ for some measurable function $g$:
    \[\xymatrix{
		E  \ar@{->}[d]_{A} \ar@{->}[r]^{T} & F  \ar@{->}[r]^{J_F} & F^{**}  \\
		L_\pp(g_1\cdot \mu) \ar@{->}[rr]_{D} &   & L_\qq(g_2\cdot\mu) \ar@{->}[u]_{B}
    }\]
    We define $i_\ppqq(T)=\inf \cbr[0]{\|A\|\|D\|\|B\|}$, where the infimum is taken over all possible factorizations above.
\end{defn}

\begin{thm}\label{thm: identification M and I}
    Assume $p>q$. Then $[M_\ppqq,\mu_\ppqq]$ and $[I_\ppqq,i_\ppqq]$ are isomorphic ideals.
\end{thm}

\begin{proof}
    $[M_\ppqq,\mu_\ppqq]\supseteq [I_\ppqq,i_\ppqq]$: Assume that $J_FT$ factors as in \Cref{def: ideal I of factorizations through multiplication operator}. Since $L_\qq(g_2\cdot\mu)$ is $(\qq)$-concave with constant 1, it follows that $B$ is $(\qq)$-concave with constant $\|B\|$. On the other hand, observe that $L_\qq(g_2\cdot\mu)$ is Dedekind complete, so the operator $D$ has a modulus $|D|$, which is given by $|D|f=|g|f$ for every $f\in L_\pp(g_1\cdot \mu)$. Moreover, $\||D|\|=\|D\|$ and $|D|$ is a lattice homomorphism. Therefore,
    \begin{align*}
        \norm[3]{\intoo[3]{\sum_{i=1}^n |Df_i|^{p_2}}^\frac{1}{p_2}}_{\qq}&  = \norm[3]{\intoo[3]{\sum_{i=1}^n ||D|f_i|^{p_2}}^\frac{1}{p_2}}_{\qq}\\
        & \leq \||D|\| \norm[3]{\intoo[3]{\sum_{i=1}^n |f_i|^{p_2}}^\frac{1}{p_2}}_{\pp}\leq \|D\| \intoo[3]{\sum_{i=1}^n \|f_i\|_\pp^{p_2}}^\frac{1}{p_2}
    \end{align*}
    and $D$ is $(\pp)$-convex with constant $\|D\|$. It follows that $DA$ is $(\pp)$-convex with constant $\|A\|\|D\|$, so $T\in M_\ppqq(E,F)$ and $\mu_\ppqq(T)\leq \|A\|\|D\|\|B\|$.\medskip

    \noindent $[M_\ppqq,\mu_\ppqq]\subseteq [I_\ppqq,i_\ppqq]$: We start by considering a factorization $J_FT=VQU$ as in \Cref{prop: M equiv to factorization through BLs} and fixing $q<r<p$. The case $p_2=p$, $q_2=q$ was established in \cite{Reisner}. Here, we show the other three cases, which are summarized in the commutative diagrams below:
    \begin{enumerate}
        \item \textbf{Case 1:} $p_2=\infty$, $q_2=q$. $X$ is $(p,\infty)$-convex, so it is $r$-convex with constant $C_{p,r}$ \cite[Theorem 1.f.7]{LT2}. Let $\widetilde{X}$ be a lattice renorming of $X$ so that the $r$-convexity constant is 1 \cite[Proposition 1.d.8]{LT2}. By rescaling, we can assume that $\widetilde{Q}:\widetilde{X}\rightarrow Y$, the composition of $Q$ with the renorming isomorphism, is a contractive lattice homomorphism. Moreover, $Y$ is already $r$-concave with constant 1, so by Krivine's Factorization Theorem (cf.~\cite[Th\'eor\`eme 2]{KrivineExp2223} or \cite[Lemma 17]{RT}), $\widetilde{Q}$ can be factored through some $L_r(\mu)$ with contractive lattice homomorphisms $R:\widetilde{X}\rightarrow L_r(\mu)$ and $S:L_r(\mu)\rightarrow Y$. Note that $S$ is $q$-concave, so, by \Cref{thm: dual factorization Maurey operator}, it factors through $L_q(g_2\cdot \mu)$ for some change of density $0\leq g_2\in S_{L_1(\mu)}$, with contractive factors $S_1:L_r(\mu)\rightarrow L_q(g_2\cdot \mu)$ and $S_2:L_q(g_2\cdot \mu)\rightarrow Y$, where $S_1$ is the operator of multiplication by $g_2^{-\frac{1}{r}}$. Similarly, by \Cref{thm: factorization Pisier operator}, there exists a change of density $0\leq g_1\in S_{L_1(\mu)}$ such that the composition of the renorming of $X$ with $R$ (which is $(p,\infty)$-convex) factors through $L_\pinf^{[r]}(g_1 \cdot \mu)$ with $R_1:X\rightarrow L_\pinf^{[r]}(g_1 \cdot \mu)$ of norm bounded by $(1-\frac{r}{p})^{\frac{1}{p}-\frac{1}{r}}C_{p,r}\|Q\|$ and $R_2:L_\pinf^{[r]}(g_1 \cdot \mu)\rightarrow L_r(\mu)$ the contractive operator of multiplication by $g_1^\frac{1}{r}$. Composing the renorming isomorphism from $L^{[1]}_\pinf(g_1\cdot  \mu)$ into $L^{[r]}_\pinf(g_1\cdot  \mu)$ given by \eqref{eq: renorming weak Lp} with the multiplication operators $R_2$ and $S_1$, we obtain the operator $D:L^{[1]}_\pinf(g_1\cdot  \mu)\rightarrow L_q(g_2\cdot \mu)$ of multiplication by $\left(\frac{g_1}{g_2}\right)^\frac{1}{r}\in L_r(g_2\cdot \mu)_+$, which has norm bounded by $\left(\frac{p}{p-r}\right)^\frac{1}{r}$. Putting everything together, we conclude that $T\in I_{(p,\infty;q)}(E,F)$ and, since the factorization through $Q$ and the exponent $r\in (q,p)$ were arbitrary, we can take the infimum over all possibilities to get that 
        \[i_{(p,\infty;q)}(T)\leq \inf_{q<r<p}\cbr[3]{\left(\frac{p}{p-r}\right)^{\frac{2}{r}-1}C_{p,r}}\mu_{(p,\infty;q)}(T).\]

        \[\xymatrix{
		E  \ar@{->}[d]_{U} \ar@{->}[rr]^{T} & & F  \ar@{->}[rr]^{J_F} & & F^{**}  \\
        X \ar@/^2.0pc/[rrrr]^{Q} \ar@{<->}[r] \ar@{->}[rd]^{R_1} \ar@{->}[d]_{\widetilde{R_1}}& \widetilde{X} \ar@{->}[rrr]^{\widetilde{Q}} \ar@{->}[rd]^{R} & & & Y \ar@{->}[u]_{V}\\
		L^{[1]}_\pinf(g_1 \mu) \ar@/_2.0pc/[rrrr]_{D} \ar@{<->}[r] & L_\pinf^{[r]}(g_1 \mu) \ar@{->}[r]^{R_2}&  L_r(\mu) \ar@{->}[rru]^{S} \ar@{->}[rr]^{S_1} & & L_q(g_2\mu) \ar@{->}[u]_{S_2}
        }\]
        \item \textbf{Case 2:} $p_2=p$, $q_2=1$. The procedure is similar, but this time, we need to renorm $Y$ with distortion $C_{q^*,r^*}$ so that it becomes $r$-concave with constant 1, and after factoring through $L_r(\mu)$ with Krivine's Theorem, we need to apply \Cref{thm: factorization Maurey operator} and \Cref{thm: dual factorization Pisier operator}, and renorm $L^{[r^*]}_{q,1}(g_2\cdot \mu)$ to recover the standard norm of $L^{[1]}_{q,1}(g_2\cdot \mu)$. Finally, we get that 
        \[i_{(p;q,1)}(T)\leq \inf_{q<r<p}\cbr[3]{\left(\frac{q^*}{q^*-r^*}\right)^{\frac{2}{r^*}-1}C_{q^*,r^*}}\mu_{(p;q,1)}(T).\]

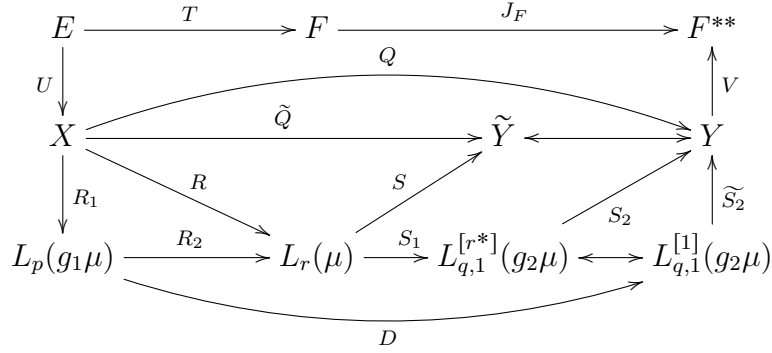
        \[\xymatrix{
		E  \ar@{->}[d]_{U} \ar@{->}[rr]^{T} & & F  \ar@{->}[rr]^{J_F} & & F^{**}  \\
        X \ar@/^2.0pc/[rrrr]^{Q} \ar@{->}[rrd]^{R} \ar@{->}[rrr]^{\widetilde{Q}}  \ar@{->}[d]^{R_1} &   & & \widetilde{Y}  \ar@{<->}[r] & Y \ar@{->}[u]_{V}\\
		L_p(g_1 \mu) \ar@/_2.0pc/[rrrr]_{D} \ar@{->}[rr]^{R_2} &  &  L_r(\mu) \ar@{->}[ru]^{S} \ar@{->}[r]^{S_1} & L^{[r^*]}_{q,1}(g_2\mu) \ar@{<->}[r] \ar@{->}[ru]_{S_2} & L^{[1]}_{q,1}(g_2\mu) \ar@{->}[u]_{\widetilde{S_2}} 
        }\]
        \item \textbf{Case 3:} $p_2=\infty$, $q_2=1$. This time, both $X$ and $Y$ need to be renormed, introducing distortions $C_{p,r}$ and $C_{p^*,r^*}$, respectively, and, after applying Theorems \ref{thm: factorization Pisier operator} and \ref{thm: dual factorization Pisier operator}, we need to renorm the corresponding $L^{[r]}_\pinf(g_1\cdot  \mu)$ and $L^{[r^*]}_{q,1}(g_2\cdot \mu)$ to obtain a multiplication operator $D$ from $L^{[1]}_\pinf(g_1\cdot  \mu)$ into $L^{[1]}_{q,1}(g_2\cdot \mu)$. This time, the relation between the ideal norms is given by
        \[i_{(p,\infty;q,1)}(T)\leq \inf_{q<r<p}\cbr[3]{\left(\frac{p}{p-r}\right)^{\frac{2}{r}-1}\left(\frac{q^*}{q^*-r^*}\right)^{\frac{2}{r^*}-1}C_{p,r}C_{q^*,r^*}}\mu_{(p,\infty;q,1)}(T).\]
        
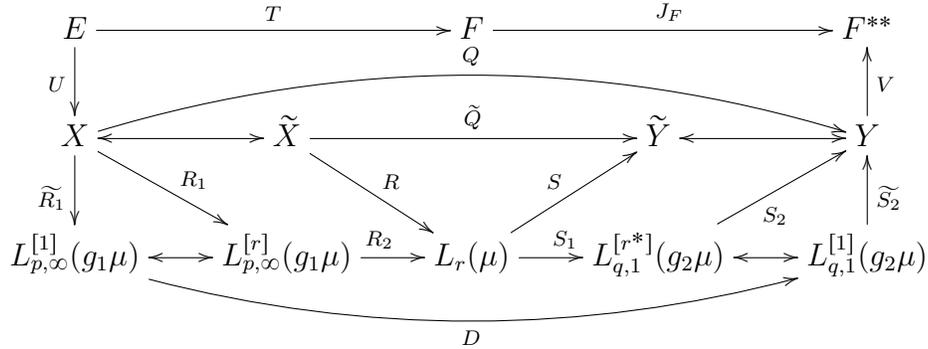
        \[\xymatrix{
		E  \ar@{->}[d]_{U} \ar@{->}[rr]^{T} & & F  \ar@{->}[rr]^{J_F} & & F^{**}  \\
        X \ar@/^2.0pc/[rrrr]^{Q} \ar@{<->}[r] \ar@{->}[rd]^{R_1} \ar@{->}[d]_{\widetilde{R_1}}& \widetilde{X} \ar@{->}[rr]^{\widetilde{Q}} \ar@{->}[rd]^{R} & &\widetilde{Y}  \ar@{<->}[r] & Y \ar@{->}[u]_{V}\\
		L^{[1]}_\pinf(g_1 \mu) \ar@/_2.0pc/[rrrr]_{D} \ar@{<->}[r] & L_\pinf^{[r]}(g_1 \mu) \ar@{->}[r]^{R_2}&  L_r(\mu) \ar@{->}[ru]^{S} \ar@{->}[r]^{S_1} & L^{[r^*]}_{q,1}(g_2\mu) \ar@{<->}[r] \ar@{->}[ru]_{S_2} & L^{[1]}_{q,1}(g_2\mu) \ar@{->}[u]_{\widetilde{S_2}} 
        }\]
    \end{enumerate}
\end{proof}

\section*{Acknowledgments}
We are grateful to Mar\'ia J. Carro and Timur Oikhberg for helpful discussions related to the content of this paper. The research of E.~Garc\'ia-S\'anchez and P.~Tradacete was partially supported by the grants PID2020-116398GB-I00, PID2024-162214NB-I00 and CEX2023-001347-S funded by the MCIN/AEI/10.13039/501100011033. The research of E.~Garc\'ia-S\'anchez was partially supported by the grant CEX2019-000904-S-21-3 funded by the MICIU/AEI/10.13039/501100011033 and by ``ESF+''.

\end{document}